\documentclass[a4paper,11pt]{article}

\usepackage[T1]{fontenc}
\usepackage{lmodern}

\usepackage{dsfont}

\usepackage[latin1]{inputenc}
\usepackage{amsmath}
\usepackage{amsthm}
\usepackage{amssymb}
\usepackage{mathrsfs}
\usepackage{graphicx}
\usepackage[all]{xy}
\usepackage{hyperref}

\usepackage{abstract} 

\newtheorem{thm}{Theorem}[section]
\newtheorem{cor}[thm]{Corollary}
\newtheorem{claim}[thm]{Claim}
\newtheorem{fact}[thm]{Fact}

\newtheorem{lemma}[thm]{Lemma}
\newtheorem{prop}[thm]{Proposition}

\theoremstyle{definition}
\newtheorem{definition}[thm]{Definition}

\newtheorem{app}[thm]{Application}

\newtheorem{question}[thm]{Question}
\newtheorem{problem}[thm]{Problem}

\title{Hyperbolicities in CAT(0) cube complexes}
\date{\today}
\author{Anthony Genevois}

\begin{document}

\maketitle

\begin{abstract}
This paper is a survey dedicated to the following question: given a group acting on a CAT(0) cube complex, how to exploit this action to determine whether or not the group is Gromov / relatively / acylindrically hyperbolic? As much as possible, the different criteria we mention are illustrated by applications. We also propose a model for universal acylindrical actions of cubulable groups, and give a few applications to Morse, stable and hyperbolically embedded subgroups. 
\end{abstract}

\tableofcontents

\section{Introduction}

\hspace{0.5cm} A well-known strategy to study groups from a geometric point of view is to find ``nice'' actions on spaces which are ``nonpositively-curved'', or even better, which are ``negatively-curved''. The most iconic illustration of this idea comes from Gromov's seminal paper \cite{GromovHyperbolic} introducing \emph{hyperbolic groups}. Since then, hyperbolic groups have been generalised in different directions. In this paper, we are interested in Gromov's hyperbolic groups as well as (strongly) relatively hyperbolic groups and the recent acylindrically hyperbolic groups. Proving that a group satisfies some hyperbolicity is very convenient as it provides interesting information of the group; see \cite{GhysdelaHarpe, OsinRH, OsinAcylSurvey} and references therein for more information. However, it may be a difficult task to show that a given group actually has a negatively-curved behavior, motivating the need of general criteria.

In this article, our objective is to stress out the idea that, if we want to determine whether a given group is hyperbolic in some sense, then it may be quite convenient to find an action on a CAT(0) cube complex (usually considered as a generalised tree in higher dimension). So the main question of the article is the following:

\begin{question}\label{mainquestion}
Let $G$ be a group acting on some CAT(0) cube complex $X$. How to exploit the action $G \curvearrowright X$ to determine whether or not $G$ is Gromov / relatively / acylindrically hyperbolic?
\end{question}

Our motivation is twofold. The first point is that the strategy actually works: we are indeed able to exploit the nice geometry of CAT(0) cube complexes in order to state and prove general criteria about hyperbolicity. And secondly, many groups of interest turn out to act on CAT(0) cube complexes, providing a large and interesting collection of potential applications. Along the article, as much as possible the different criteria we will mention will be illustrated by concrete applications, justifying our choice of working with cube complexes. (Indeed, the applications we mention deal with no less than twelve classes of groups!)

Although most of our article is a survey of already published works, some of our results are new, including:
\begin{itemize}
	\item The introduction of \emph{Morse subgroups} (introduced independently in \cite{StronglyQC} under the name \emph{strongly quasiconvex subgroups}) and their characterisation in cubulable groups (see Section \ref{section:Morse}, especially Corollary \ref{cor:MorseCriterion}).
	\item A proof of the freeness of Morse subgroups in freely irreducible right-angled Artin groups (see Appendix \ref{section:MorseInRAAG}).
	\item The characterisation of hyperbolically embedded subgroups in cubulable groups (see Section \ref{section:hypembed}.
	\item The introduction of hyperbolic models for CAT(0) cube complexes, with applications to stable subgroups and regular elements (see Section \ref{section:curvegraph}).  
	\item A short study of crossing graphs of CAT(0) cube complexes, stressing out their similarity with contact graphs (see Appendix \ref{section:crossing}).
\end{itemize}
Along our text, several open questions are left. Some of them being well-known, and other ones being new.

\paragraph{Acknowledgments.} I am grateful to Hung Tran for useful comments on the first version of this paper.

\section{Preliminaries}\label{section:preliminaries}

\noindent
A \textit{cube complex} is a CW complex constructed by gluing together cubes of arbitrary (finite) dimension by isometries along their faces. It is \textit{nonpositively curved} if the link of any of its vertices is a simplicial \textit{flag} complex (ie., $n+1$ vertices span a $n$-simplex if and only if they are pairwise adjacent), and \textit{CAT(0)} if it is nonpositively curved and simply-connected. See \cite[page 111]{MR1744486} for more information.

\medskip \noindent
Fundamental tools when studying CAT(0) cube complexes are \emph{hyperplanes}. Formally, a \textit{hyperplane} $J$ is an equivalence class of edges with respect to the transitive closure of the relation identifying two parallel edges of a square. Notice that a hyperplane is uniquely determined by one of its edges, so if $e \in J$ we say that $J$ is the \textit{hyperplane dual to $e$}. Geometrically, a hyperplane $J$ is rather thought of as the union of the \textit{midcubes} transverse to the edges belonging to $J$ (sometimes referred to as its \emph{geometric realisation}). See Figure \ref{figure27}. The \textit{carrier} $N(J)$ of a hyperplane $J$ is the union of the cubes intersecting (the geometric realisation of) $J$. 
\begin{figure}
\begin{center}
\includegraphics[trim={0 13cm 10cm 0},clip,scale=0.4]{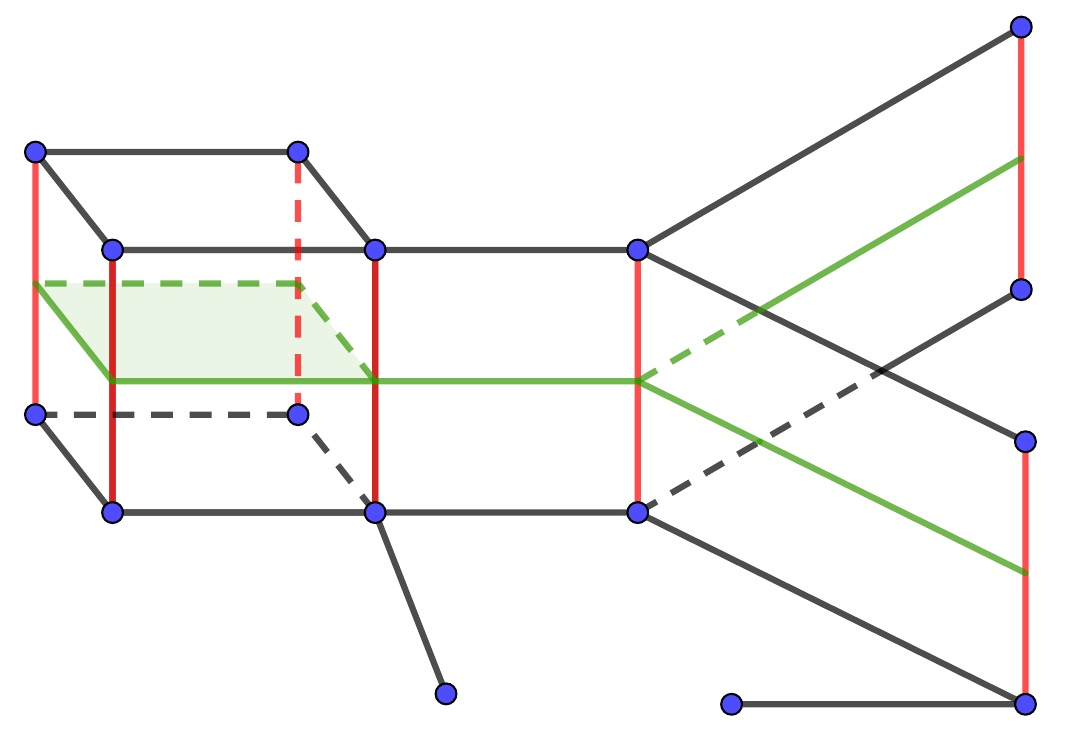}
\caption{A hyperplane (in red) and the associated union of midcubes (in green).}
\label{figure27}
\end{center}
\end{figure}

\medskip \noindent
There exist several metrics naturally defined on a CAT(0) cube complex. For instance, for any $p \in (0,+ \infty)$, the $\ell_p$-norm defined on each cube can be extended to a length metric defined on the whole complex, the \emph{$\ell_p$-metric}. Usually, the $\ell_1$-metric is referred to as the \emph{combinatorial distance} and the $\ell_2$-metric as the \emph{CAT(0) distance}. In this article, we are mainly interested in the combinatorial metric. Actually, unless specified otherwise, we will identify a CAT(0) cube complex with its one-skeleton, thought of as a collection of vertices endowed with a relation of adjacency. In particular, when writing $x \in X$, we always mean that $x$ is a vertex of $X$. 

\medskip \noindent
The following theorem is one of the most fundamental results about the geometry of CAT(0) cube complexes.

\begin{thm}\emph{\cite{MR1347406}}
Let $X$ be a CAT(0) cube complex.
\begin{itemize}
	\item If $J$ is a hyperplane of $X$, the graph $X \backslash \backslash J$ obtained from $X$ by removing the (interiors of the) edges of $J$ contains two connected components. They are convex subgraphs of $X$, referred to as the \emph{halfspaces} delimited by $J$.
	\item A path in $X$ is a geodesic if and only if it crosses each hyperplane at most once.
	\item For every $x,y \in X$, the distance between $x$ and $y$ coincides with the cardinality of the set $\mathcal{W}(x,y)$ of the hyperplanes separating them.
\end{itemize}
\end{thm}

\noindent
Now, we record several results on the geometry of CAT(0) cube complexes which will be used in the rest of the article.

\paragraph{Projections.} Given a CAT(0) cube complex $X$ and a convex subcomplex $C$, we know that, for every vertex $x \in X$, there exists a unique vertex of $C$ minimising the distance to $x$ (see for instance \cite[Lemma 13.8]{MR2377497}); we refer to this new vertex as the \emph{projection of $x$ onto $C$}, and we denote by $\mathrm{proj}_C : X \to C$ the map associating to a vertex of $X$ its projection onto $C$. Below is a list of results which will be useful later.

\begin{prop}\label{prop:proj}
\emph{\cite[Proposition 2.9]{coningoff}}
Let $X$ be a CAT(0) cube complex and $A,B \subset X$ two convex subcomplexes. The projection $\mathrm{proj}_{B}(A)$ is a geodesic subcomplex of $B$. Moreover, the hyperplanes intersecting $\mathrm{proj}_{B}(A)$ are precisely those which intersect both $A$ and $B$. 
\end{prop}

\begin{lemma}\label{lem:vertextoproj}
\emph{\cite[Lemma 13.8]{MR2377497}}
Let $X$ be a CAT(0) cube complex, $Y \subset X$ a convex subcomplex and $x \in X$ a vertex. Any hyperplane separating $x$ from its projection onto $Y$ must separate $x$ from $Y$.
\end{lemma}

\begin{lemma}\label{lem:sepproj}
\emph{\cite[Proposition 2.6]{article3}}
Let $X$ be a CAT(0) cube complex, $C \subset X$ a convex subcomplex and $x,y \in X$ two vertices. The hyperplanes separating the projections of $x$ and $y$ onto $C$ are precisely the hyperplanes separating $x$ and $y$ which intersect $C$. 
\end{lemma}

\begin{lemma}\label{lem:twomin}
\emph{\cite[Corollary 13.10]{MR2377497}}
Let $X$ be a CAT(0) cube complex and $Y_1,Y_2 \subset X$ two convex subcomplexes. If $x \in Y_1$ and $y \in Y_2$ are two vertices minimising the distance between $Y_1$ and $Y_2$, then the hyperplanes separating $x$ and $y$ are precisely the hyperplanes separating $Y_1$ and $Y_2$. 
\end{lemma}

\begin{lemma}\label{lem:projinter}
\emph{\cite[Lemma 2.38]{Qm}}
Let $X$ be a CAT(0) cube complex and $Y_1,Y_2 \subset X$ two intersecting convex subcomplexes. Then $\mathrm{proj}_{Y_1} \circ \mathrm{proj}_{Y_2} = \mathrm{proj}_{Y_1 \cap Y_2}$. 
\end{lemma}

\paragraph{Cycles of subcomplexes.} Given a CAT(0) cube complex, a \emph{cycle of subcomplexes} is a sequence of subcomplexes $(C_1, \ldots, C_r)$ such that, for every $i \in \mathbb{Z}/r \mathbb{Z}$, the subcomplexes $C_i$ and $C_{i+1}$ intersects. 

\begin{prop}\label{prop:cycle}
Let $(A,B,C,D)$ be a cycle of four convex subcomplexes. There exists a combinatorial isometric embedding $[0,p] \times [0,q] \hookrightarrow X$ such that $[0,p] \times \{0\} \subset A$, $\{p\} \times [0,q] \subset B$, $[0,p] \times \{ q \} \subset C$ and $\{0\} \times [0,q] \subset D$. Moreover, the hyperplanes intersecting $[0,p] \times \{ 0\}$ (resp. $\{0\} \times [0,q]$) are disjoint from $B$ and $D$ (resp. $A$ and $C$). 
\end{prop}

\begin{proof}
First of all, let us record a statement which is contained into the proof of \cite[Proposition 2.111]{Qm} (in the context of \emph{quasi-median graphs}, a class of graphs including median graphs, i.e., one-skeleta of CAT(0) cube complexes).

\begin{fact}
If $a$ is a vertex of $A \cap D$ minimising the distance to $B \cap C$ and if $b$ (resp. $c$, $d$) denotes the projection of $a$ onto $B$ (resp. $B \cap C$, $C$), then there exists a combinatorial isometric embedding  $[0,p] \times [0,q] \hookrightarrow X$ such that $(0,0)=a$, $(p,0)=b$, $(p,q)=c$ and $(0,q)=d$. 
\end{fact}

\noindent
By convexity of $A$, $B$, $C$ and $D$, this implies that $[0,p] \times \{0\} \subset A$, $\{p\} \times [0,q] \subset B$, $[0,p] \times \{ q \} \subset C$ and $\{0\} \times [0,q] \subset D$. 

\medskip \noindent
Let $J$ be a hyperplane intersecting $[0,p] \times \{0\}$. We know from Lemma \ref{lem:vertextoproj} that $J$ must be disjoint from $B$. Moreover, if $J$ intersects $D$, then it follows from Helly's property, satisfied by convex subcomplexes in CAT(0) cube complexes, that $J$ must intersect $A \cap D$, which contradicts Lemma \ref{lem:twomin}. Consequently, any hyperplane intersecting $[0,p] \times \{0\}$ must be disjoint from both $B$ and $D$. By symmetry, one shows similarly that any hyperplane intersecting $\{0\} \times [0,q]$ must be disjoint from $A$ and $C$. 
\end{proof}

\paragraph{Quadruples.} Recall that, in a CAT(0) cube complex, the \emph{interval} between two vertices $x$ and $y$, denoted by $I(x,y)$, is the union of all the geodesics joining $x$ and $y$. 

\begin{lemma}\label{lem:quadrangle}
Let $X$ be a CAT(0) cube complex and $x_1,x_2,x_3,x_4 \in X$ four vertices. There exist four vertices $m_1,m_2,m_3,m_4 \in X$ such that 
\begin{itemize}
	\item $I(x_i,m_i) \cup I(m_i,m_{i+1}) \cup I(m_{i+1},x_{i+1}) \subset I(x_i,x_{i+1})$ for every $i \in \mathbb{Z}/4 \mathbb{Z}$;
	\item there exists a combinatorial isometric embedding $[0,a] \times [0,b] \hookrightarrow X$ such that $m_1=(0,0)$, $m_2=(a,0)$, $m_3=(a,b)$, and $m_4=(0,b)$.
\end{itemize}
\end{lemma}

\begin{proof}
Let $\mathcal{U}$ denote the collection of the halfspaces containing exactly one vertex among $x_1,x_2,x_3,x_4$, and let $\mathcal{D}$ denote the collection of the halfspaces containing exactly $x_2,x_4$ among $x_1,x_2,x_3,x_4$. Notice that any two halfspaces of $\mathcal{U} \cup \mathcal{D}$ intersect in the convex hull $C$ of $\{ x_1,x_2,x_3,x_4 \}$, which is precisely the intersection of all the halfspaces containing at least two vertices among $x_1,x_2,x_3,x_4$; and that the collection $\mathcal{U} \cup \mathcal{D}$ is finite since these halfspaces are delimited by hyperplanes separating at least two vertices among $x_1,x_2,x_3,x_4$. It follows from Helly's property, satisfied by convex subcomplexes in CAT(0) cube complexes, that the intersection 
$$Q = \bigcap\limits_{A \in \mathcal{U} \cup \mathcal{D}} A \cap C$$
is non-empty. For every $1 \leq i \leq 4$, let $m_i$ denote the projection of $x_i$ onto $Q$. By construction, the hyperplanes intersecting $Q$ are precisely the hyperplanes separating $\{x_1,x_2 \}$ and $\{x_3,x_4 \}$, and those separating $\{ x_1,x_4 \}$ and $\{x_2,x_3 \}$. Let $\mathcal{H}$ denote the first collection, and $\mathcal{V}$ the second one. It follows from Lemma \ref{lem:sepproj} that
$$d(m_1,m_2)=d(m_3,m_4)= \# \mathcal{H}, \ d(m_2,m_3)=d(m_1,m_4)=\# \mathcal{V},$$  
$$d(m_1,m_3)=d(m_2,m_4)= \# \mathcal{H}+ \# \mathcal{V}.$$
A fortiori, $m_2,m_4 \in I(m_1,m_3)$ and $m_1,m_3 \in I(m_2,m_4)$. It follows from \cite[Lemma 2.110]{Qm} (proved in the context of \emph{quasi-median graphs}, a class of graphs including median graphs, i.e., one-skeleta of CAT(0) cube complexes) that there exists a combinatorial isometric embedding $[0,a] \times [0,b] \hookrightarrow X$ such that $m_1=(0,0)$, $m_2=(a,0)$, $m_3=(a,b)$, and $m_4=(0,b)$.

\medskip \noindent
Now fix some $i \in \mathbb{Z}/4 \mathbb{Z}$ and some geodesics $[x_i,m_i]$, $[m_i,m_{i+1}]$, $[x_{i+1},m_{i+1}]$ respectively between $x_i$ and $m_i$, $m_i$ and $m_{i+1}$, and $x_{i+1}$ and $m_{i+1}$. We claim that the concatenation $[x_i,m_i] \cup [m_i,m_{i+1}] \cup [m_{i+1},x_{i+1}]$ is a geodesic. Indeed, it follows from Lemma \ref{lem:vertextoproj} that no hyperplane intersects both $[x_i,m_i]$ and $[m_i,m_{i+1}]$, nor both $[m_i,m_{i+1}]$ and $[x_{i+1},m_{i+1}]$. Next, suppose that $J$ is a hyperplane intersecting both $[x_i,m_i]$ and $[x_{i+1},m_{i+1}]$. We know from Lemma \ref{lem:vertextoproj} that $J$ must be disjoint from $Q$, so $x_i$ and $x_{i+1}$ belong to the same halfspace $D$ delimited by $J$. Since $J \notin \mathcal{H} \cup \mathcal{V}$, this implies that $D$ belongs to $\mathcal{D}$. By construction of $Q$, necessarily $Q \subset D$, contradicting the fact that $m_i$ and $m_{i+1}$ do not belong to $D$. Consequently, the path $[x_i,m_i] \cup [m_i,m_{i+1}] \cup [m_{i+1},x_{i+1}]$ intersects each hyperplane at most once, proving our claim. A fortiori, $I(x_i,m_i) \cup I(m_i,m_{i+1}) \cup I(m_{i+1},x_{i+1}) \subset I(x_i,x_{i+1})$. 
\end{proof}

\section{Gromov hyperbolicity}\label{section:Gromovhyp}

\noindent
Recall that a geodesic metric space $X$ is \emph{Gromov hyperbolic} (or just \emph{hyperbolic} for short) if there exists some constant $\delta \geq 0$ such that all the geodesic triangles of $X$ are \emph{$\delta$-thin}, i.e., any side is contained into the $\delta$-neighborhood of the union of the two others. The question we are interested in is: when is a CAT(0) cube complex Gromov hyperbolic? 

\medskip \noindent
Of course, first we have to fix the metric we consider since, as mentioned in Section \ref{section:preliminaries}, several metrics are naturally defined on CAT(0) cube complexes. Recall that, for every $p \in (0,+ \infty)$, the $\ell^p$-metric defined on a cube complex is the length metric which extends the $\ell^p$-norms defined on each cube. For finite-dimensional CAT(0) cube complexes, all these metrics turn out to be quasi-isometric; but they may be quite different for infinite-dimensional complexes. Nevertheless, by noticing that an $n$-cube contains a triangle which is not $(n^{1/p}-1)$-thin with respect to the $\ell^p$-norm, only two cases need to be considered: the finite-dimensional situation with respect to the $\ell^1$-metric (the other $\ell^p$-metrics being quasi-isometric to this one); and the infinite-dimensional situation with respect to the $\ell^{\infty}$-metric (the other $\ell^p$-metrics being not able to be hyperbolic). 

\medskip \noindent
The former case has been studied in several places, in particular \cite{CDEHV, Hagenthesis, coningoff}. Our next statement sum up the criteria which can be found there. We begin by defining the needed vocabulary. 
\begin{itemize}
	\item A \emph{flat rectangle} is a combinatorial isometric embedding $[0,p] \times [0,q] \hookrightarrow X$ for some integers $p,q \geq 0$; it is \emph{$L$-thin} for some $L \geq 0$ if $\min(p,q) \leq L$. 
	\item A \emph{facing triple} is the data of three hyperplanes such that no one separates the other two. 
	\item A \emph{join of hyperplanes} $(\mathcal{H},\mathcal{V})$ is the data of two collections of hyperplanes $\mathcal{H}, \mathcal{V}$ which do not contain any facing triple such that any hyperplane of $\mathcal{H}$ is transverse to any hyperplane of $\mathcal{V}$; if moreover $\mathcal{H}, \mathcal{V}$ are collections of pairwise disjoint hyperplanes, then $(\mathcal{H}, \mathcal{V})$ is a \emph{grid of hyperplanes}. The join or grid of hyperplanes $(\mathcal{H}, \mathcal{V})$ is \emph{$L$-thin} for some $L \geq 0$ if $\min(\# \mathcal{H}, \# \mathcal{V}) \leq L$; it is \emph{$L$-thick} if $\# \mathcal{H}, \# \mathcal{V} \geq L$. 
	\item The \emph{crossing graph} $\Delta X$ of a CAT(0) cube complex $X$ is the graph whose vertices are the hyperplanes of $X$ and whose edges link two transverse hyperplanes. It has \emph{thin bicycles} if there exists some $K \geq 0$ such that any bipartite complete subgraph $K_{p,q} \subset \Delta X$ satisfies $\min(p,q) \leq K$.
\end{itemize}
Notice for instance that the crossing graph of a flat rectangle defines a grid of hyperplanes. So flat rectangles, join or grid of hyperplanes, and bipartite complete subgraphs in the crossing graph are three different ways of thinking about ``flat subspaces'' in cube complexes. Now we are ready to state our criteria, saying basically that a cube complex is hyperbolic if and only if its flat subspaces cannot be too ``thick''. 

\begin{thm}\label{thm:hyperbolic1}
Let $X$ be an arbitrary CAT(0) cube complex endowed with the $\ell^1$-metric. The following statements are equivalent:
\begin{itemize}
	\item[(i)] $X$ is hyperbolic;
	\item[(ii)] the flat rectangles of $X$ are uniformly thin;
	\item[(iii)] the joins of hyperplanes of $X$ are uniformly thin;
	\item[(iv)] $X$ is finite-dimensional and its grids of hyperplanes are uniformly thin;
\end{itemize}
moreover, if $X$ is cocompact (i.e., there exists a group acting geometrically on $X$), then the previous statements are also equivalent to:
\begin{itemize}
	\item[(v)] there does not exist a combinatorial isometric embedding $\mathbb{R}^2 \hookrightarrow X$;
	\item[(vi)] the crossing graph of $X$ has thin bicycles.
\end{itemize}
\end{thm}

\begin{proof}
The equivalences $(i) \Leftrightarrow (ii) \Leftrightarrow (iv)$ are proved by \cite[Theorem 3.3]{coningoff}; the equivalence $(i) \Leftrightarrow (ii)$ can also be found in \cite[Corollary 5]{CDEHV}. The implication $(iii) \Rightarrow (iv)$ is clear. The converse follows from the next fact, which is an easy consequence of  \cite[Lemma 3.7]{coningoff}. We recall that $\mathrm{Ram}(\cdot)$ denotes the \emph{Ramsey number}. Explicitly, if $n \geq 0$, $\mathrm{Ram}(n)$ is the smallest integer $k \geq 0$ satisfying the following property: if one colors the edges of a complete graph containing at least $k$ vertices with two colors, it is possible to find a monochromatic complete subgraph containing at least $n$ vertices.  

\begin{fact}\label{fact:pairwisedisjoint}
Let $(\mathcal{H}, \mathcal{V})$ be a join of hyperplanes satisfying $\# \mathcal{H}, \# \mathcal{V} \geq \mathrm{Ram}(k)$ for some $k> \dim(X)$. Then there exist subcollections $\mathcal{H}' \subset \mathcal{H}$ and $\mathcal{V}' \subset \mathcal{V}$ such that $(\mathcal{H}', \mathcal{V}')$ is a grid of hyperplanes satisfying $\# \mathcal{H}' , \mathcal{V}' \geq k$. 
\end{fact}

\noindent
Now, suppose that $X$ is cocompact. The equivalence $(i) \Leftrightarrow (vi)$ is \cite[Theorem 4.1.3]{Hagenthesis}. It remains to show that $(i) \Leftrightarrow (v)$. The implication $(i) \Rightarrow (v)$ is clear, so we only have to prove the converse. Suppose that $X$ is not hyperbolic. Point $(ii)$ implies that, for every $n \geq 1$, there exists a flat rectangle $[0,2n] \times [0,2n] \hookrightarrow X$; let $D_n$ denote its image in $X$. Because $X$ is cocompact, we may suppose without loss of generality that $(n,n)$ belongs to a given ball $B$ for every $n \geq 1$. Next, since $X$ is locally finite (since cocompact), the sequence $(D_n)$ subconverges to some subcomplex $D_{\infty} \subset X$, i.e., there exists a subsequence of $(D_n)$ which is eventually constant on every ball. Necessarily, $D_{\infty}$ is isometric to the square complex $\mathbb{R}^2$, giving a combinatorial isometric embedding $\mathbb{R}^2 \hookrightarrow X$. 
\end{proof}

\noindent
As a consequence of Theorem \ref{thm:hyperbolic1}, we recover a sufficient criterion formulated by Gromov in \cite[Section 4.2.C]{GromovHyperbolic}. (In fact, under these assumptions, Gromov showed more generally that the cube complex can be endowed with a CAT(-1) metric.)

\begin{cor}\label{cor:squareless}
Let $X$ be a CAT(0) cube complex. If no vertex of $X$ has an induced cycle of length four in its link then $X$ is hyperbolic.
\end{cor}

\begin{app}
Fix a graph $\Gamma$ (without multiple edges nor loops) and a collection of groups $\mathcal{G}= \{ G_v \mid v \in V(\Gamma) \}$ indexed by the vertices of $\Gamma$. The \emph{graph product} $\Gamma \mathcal{G}$, as defined in \cite{GreenGP}, is the quotient
$$ \left( \underset{v \in V(\Gamma)}{\ast} G_v \right) / \langle \langle [g,h], \ (u,v) \in E(\Gamma), g \in G_u, h \in G_v \rangle \rangle.$$
Loosely speaking, $\Gamma \mathcal{G}$ is the disjoint union of the $G_v$'s in which two adjacent groups commute. Notice that, if the groups of $\mathcal{G}$ are all infinite cyclic, we recover the right-angled Artin group $A(\Gamma)$; and if the groups of $\mathcal{G}$ are all cyclic of order two, we recover the right-angled Coxeter group $C(\Gamma)$. In \cite{MeierGP}, John Meier use the criterion provided by Corollary \ref{cor:squareless} to characterise precisely when a graph product is hyperbolic, just by looking at $\Gamma$ and the cardinalities of the vertex-groups (trivial, finite, or infinite). As a particular case, a right-angled Coxeter group $C(\Gamma)$ turns out to be hyperbolic if and only if $\Gamma$ is square-free. For an alternative proof of Meier's theorem, based on Theorem~\ref{thm:hyperbolic1} (in a more general context, but which can be adapted to produce a purely cubical argument), see \cite[Theorem 8.30]{Qm}.   
\end{app}

\begin{app}\label{app:braidgrouphyp}
Let $\Gamma$ be a topological graph and $n \geq 1$ an integer. Define the \emph{ordered configuration space} $C_n(\Gamma)$ as
$$\Gamma^n \backslash \{ (x_1, \ldots, x_n) \mid \text{$x_i=x_j$ for some $i \neq j$} \},$$
and the \emph{unordered configuration space} $UC_n(\Gamma)$ as the quotient of $C_n(\Gamma)$ by the free action of the symmetric group $S_n$ which acts by permuting the coordinates. The fundamental group of $UC_n(\Gamma)$ based at some point $\ast$ is the \emph{graph braid group} $B_n(\Gamma, \ast)$. In \cite{GraphBraidGroups}, it is shown how to discretise $UC_n(\Gamma)$ in order to produce a nonpositively-curved cube complex with $B_n(\Gamma,\ast)$ as its fundamental group. Theorem \ref{thm:hyperbolic1} is applied in \cite{MoiSpecialBraid} to determine precisely when a graph braid group is hyperbolic. For instance, if $\Gamma$ is connected, then $B_2(\Gamma, \ast)$ is hyperbolic if and only if $\Gamma$ does not contain a pair of disjoint cycles. 
\end{app}

\noindent
Now, let us turn to the metric $\ell^{\infty}$. This situation was considered in \cite{coningoff}. 

\begin{thm}
Let $X$ be a CAT(0) cube complex endowed with the $\ell^{\infty}$-metric. Then $X$ is hyperbolic if and only if the grids of hyperplanes of $X$ are uniformly thin. 
\end{thm}

\noindent
Loosely speaking, passing from the $\ell^1$-metric to the $\ell^{\infty}$-metric ``kills'' the dimension (since the $\ell^{\infty}$-diameter of a cube remains one whatever its dimension), which explains why one gets Point $(iv)$ of Theorem \ref{thm:hyperbolic1} with the condition on the dimension removed. Interestingly, one obtains hyperbolic infinite-dimensional CAT(0) cube complexes. 

\begin{app}\label{app:smallcancellationhyp}
In \cite{MR2053602}, Wise shows how to endow every small cancellation polygonal complex with a structure of space with walls. The small cancellation condition which we consider here is C'(1/4)-T(4), meaning that every cycle in the link of some vertex has length at least four, and that the length of the intersection between any two polygons must be less than a quarter of the total perimeter of any of the two polygons. Under this condition, the CAT(0) cube complex obtained by cubulating the previous space with walls is finite-dimensional and hyperbolic if there exists a bound on the perimeters of the polygons, and infinite-dimensional otherwise. It is shown in \cite{coningoff} that, with respect to the $\ell^{\infty}$-metric, this infinite-dimensional cube complex is also hyperbolic. This observation is the starting point to the proof of the acylindrical hyperbolicity of C'(1/4)-T(4) small cancellation products; see Application \ref{app:smallcancellationproduct}.
\end{app}

\noindent
So we have a good understanding of the Gromov hyperbolicity of CAT(0) cube complexes. Nevertheless, a major question remains open:

\begin{question}
Is a group which acts geometrically on a CAT(0) cube complex and which does not contain $\mathbb{Z}^2$ as a subgroup Gromov hyperbolic?
\end{question}

\noindent
For background on this question, see \cite{WiseNotPeriodicFlat, SpecialHyp, MoiSpecialBraid, PeriodicFlatsCC, NTY, ContractingCentraliser}. For fun, we mention the following consequence of Caprace and Sageev's work \cite{MR2827012}.

\begin{thm}
A group acting geometrically on some CAT(0) cube complex which does not contain $\mathbb{Z}^2$ as a subgroup must be virtually cyclic or acylindrically hyperbolic.
\end{thm}

\begin{proof}
Let $G$ be a group acting geometrically on some CAT(0) cube complex $X$. According to \cite[Proposition 3.5]{MR2827012}, we may suppose without loss of generality that the action $G \curvearrowright X$ is \emph{essential} (i.e., no halfspace is contained into a neighborhood of its complementary). According to \cite[Theorem 6.3]{MR2827012}, two cases may happen: either $G$ contains a contracting isometry, so that it must be virtually cyclic or acylindrically hyperbolic (see Section \ref{section:contractingisom}); or $X$ decomposes as a Cartesian product of two unbounded complexes. In the latter case, it follows from \cite[Corollary D]{MR2827012} (see also \cite{ContractingCentraliser}) that $G$ contains $\mathbb{Z}^2$ as a subgroup. 
\end{proof}

\section{Morse subgroups}\label{section:Morse}

\noindent
In this section, we are concerned with \emph{Morse subgroups} which will play a fundamental role in the next sections. Loosely speaking, they are subgroups with some hyperbolic behavior. Formally:

\begin{definition}
Let $X$ be a geodesic metric space and $Y \subset X$ a subspace. Then $Y$ is a \emph{Morse subspace} if, for every $A >0$ and every $B \geq 0$, there exists a constant $K \geq 0$ such that any $(A,B)$-quasigeodesic in $X$ between any two points of $Y$ lies in the $K$-neighborhood of $Y$. As a particular case, if $G$ is a finitely generated group, then $H \subset G$ is a \emph{Morse subgroup} if it is a Morse subspace in some (or equivalently, any) Cayley graph of $G$ (constructed from a finite generating set).  
\end{definition}

\noindent
Morse subgroups encompass quasiconvex subgroups in hyperbolic groups, fully relatively quasiconvex subgroups in relatively hyperbolic groups, and hyperbolically embedded subgroups in acylindrically hyperbolic groups \cite{Sistohypembed}. The following result shows that Morse subgroups are convex-cocompact, generalising \cite[Theorem 1.1]{SageevWiseCores}.

\begin{prop}\label{prop:core}
Let $G$ be a group acting geometrically on a CAT(0) cube complex $X$ and $H \leq G$ a Morse subgroup. For any compact subspace $Q \subset X$, there exists a $G$-cocompact convex subcomplex containing $Q$.
\end{prop}

\noindent
The proof reduces essentially to the following lemma (proved in \cite[Theorem H]{MR2413337} for uniformly locally finite CAT(0) cube complexes), where $\mathrm{Ram}(\cdot)$ denotes the \emph{Ramsey number}. Recall that, if $n \geq 0$, $\mathrm{Ram}(n)$ is the smallest integer $k \geq 0$ satisfying the following property: if one colors the edges of a complete graph containing at least $k$ vertices with two colors, it is possible to find a monochromatic complete subgraph containing at least $n$ vertices. Often, it is used to find a subcollection of pairwise disjoint hyperplanes in a collection of hyperplanes of some finite-dimensional CAT(0) cube complex (see for instance \cite[Lemma 3.7]{coningoff}). 

\begin{lemma}\label{lem:convexhull}
Let $X$ be a finite-dimensional CAT(0) cube complex and $S \subset X$ a set of vertices which is combinatorially $K$-quasiconvex. Then the combinatorial convex hull of $S$ is included into the $\mathrm{Ram}( \max(\dim(X)+1,K))$-neighborhood of $S$.
\end{lemma}

\begin{proof}
Let $x \in X$ be a vertex which belongs to the combinatorial convex hull of $S$, and let $p \in S$ be a vertex of $S$ which minimises the distance to $x$. If $d(p,x) \geq \mathrm{Ram}(n)$ for some $n \geq \dim(X)+1$, then there exists a collection of hyperplanes $J_1, \ldots, J_n$ separating $x$ and $p$ such that, for every $2 \leq i \leq n-1$, $J_i$ separates $J_{i-1}$ and $J_{i+1}$. Because $x$ belongs to the combinatorial convex hull of $S$, no hyperplane separates $x$ from $S$. Therefore, there exists some $y \in S$ such that $J_1, \ldots, J_n$ separate $p$ and $y$. Let $m$ denote the median vertex of $\{x,y,p\}$. Because $m$ belongs to a combinatorial geodesic between $x$ and $p$ and that $d(x,p)=d(x,S)$, necessarily $d(m,p)=d(m,S)$. On the other hand, $m$ belongs to a combinatorial geodesic between $y,p \in S$, so the combinatorial $K$-quasiconvexity of $S$ implies $d(m,S) \leq K$, hence $d(m,p) \leq K$. Finally, since the hyperplanes $J_1, \ldots, J_n$ separates $p$ from $\{x,y\}$, we conclude that $n \leq K$.
\end{proof}

\begin{proof}[Proof of Proposition \ref{prop:core}.]
Let $x_0 \in X$ be a base vertex. Because being a Morse subspace is invariant by quasi-isometry, the orbit $H \cdot x_0$ is a Morse subspace. Furthermore, if $R>0$ is such that $Q \subset (H \cdot x_0)^{+R}$, then $(H \cdot x_0)^{+R}$ is again a Morse subspace. Let $Y$ denote its combinatorial convex hull. Because a Morse subspace is combinatorially quasiconvex, we deduce from Lemma \ref{lem:convexhull} that $Y$ is included into some neighborhood of $H \cdot x_0$. This is the cocompact core we are looking for.
\end{proof}

\noindent
A very important result on the geometry of CAT(0) cube complexes is that Morse subspaces turn out to coincide with \emph{contracting subspaces}.

\begin{definition}
Let $X$ be a metric space and $Y \subset X$ a subspace. Then $Y$ is \emph{contracting} if there exists some $K \geq 0$ such that the nearest-point projection onto $Y$ of any ball disjoint from $Y$ has diameter at most $K$.
\end{definition}

\noindent
Before proving the statement we mentioned above, let us state the next characterisation of contracting convex subcomplexes, which was obtained in \cite{article3}. There, the following notation is used: if $Y$ is a subcomplex, $\mathcal{H}(Y)$ denotes the set of hyperplanes separating at least two vertices of $Y$. 

\begin{prop}\label{prop:contracting}
Let $X$ be a CAT(0) cube complex and $Y \subset X$ a convex subcomplex. The following statements are equivalent:
\begin{itemize}
	\item[(i)] $Y$ is contracting;
	\item[(ii)] there exists some constant $C \geq 0$ such that any join of hyperplanes $(\mathcal{H}, \mathcal{V})$ satisfying $\mathcal{V} \subset \mathcal{H}(Y)$ and $\mathcal{H} \cap \mathcal{H}(Y)= \emptyset$ must be $C$-thin;
\end{itemize}
moreover, if $X$ is finite-dimensional, these statements are also equivalent to:
\begin{itemize}
	\item[(iii)] there exists some constant $C \geq 0$ such that any grid of hyperplanes $(\mathcal{H}, \mathcal{V})$ satisfying $\mathcal{V} \subset \mathcal{H}(Y)$ and $\mathcal{H} \cap \mathcal{H}(Y)= \emptyset$ must be $C$-thin;
	\item[(iv)] there exists some constant $C \geq 0$ such that every flat rectangle $R : [0,p] \times [0,q] \hookrightarrow X$ satisfying $R \cap Y = [0,p] \times \{ 0\}$ must be $C$-thin.
\end{itemize}
\end{prop}

\begin{proof}
The equivalence $(i) \Leftrightarrow (ii)$ is \cite[Theorem 3.6]{article3}. The implication $(ii) \Rightarrow (iii)$ is clear. The converse is a consequence of Fact \ref{fact:pairwisedisjoint}. Finally, the equivalence $(iii) \Leftrightarrow (iv)$ is proved in \cite[Theorem 2.7]{MoiSpecialBraid}.
\end{proof}

\noindent
Now we are ready to prove that Morse and contracting subspaces coincide. 

\begin{lemma}\label{MorseContracting}
Let $X$ be a finite-dimensional CAT(0) cube complex and $S \subset X$ a set of vertices. Then $S$ is a Morse subspace if and only if it is contracting.
\end{lemma}

\begin{proof}
It is proved in \cite[Lemma 3.3]{MR3175245} that, in any geodesic metric spaces, a contracting quasi-geodesic always defines a Morse subspace. In fact, the proof does not depend on the fact that the contracting subspace we are looking at is a quasi-geodesic, so that being a contracting subspace implies being a Morse subspace.

\medskip \noindent
Conversely, suppose that $S$ is not contracting. If $S$ is not combinatorially quasiconvex, then it cannot define a Morse subspace, and there is nothing to prove. Consequently, we suppose that $S$ is combinatorially quasiconvex. According to Lemma \ref{lem:convexhull}, the Hausdorff distance between $S$ and its combinatorial convex hull $C$ is finite. Thus, $C$ cannot be contracting according to \cite[Lemma 2.18]{article3}. We deduce from Proposition \ref{prop:contracting} that, for every $n \geq 1$, there exist a grid of hyperplanes $(\mathcal{H}, \mathcal{V})$ satisfying $\mathcal{H} \cap \mathcal{H}(C) = \emptyset$, $\mathcal{V} \subset \mathcal{H}(C)$ and $\# \mathcal{H},\# \mathcal{V} \geq n$. We write $\mathcal{H}= \{ H_1, \ldots, H_r \}$ (resp. $\mathcal{V}= \{ V_1, \ldots, V_s \}$) so that $H_i$ separates $H_{i-1}$ and $H_{i+1}$ for every $2 \leq i \leq r-1$ (resp. $V_i$ separates $V_{i-1}$ and $V_{i+1}$ for every $2 \leq i \leq s-1$); we suppose that $H_i$ separates $H_r$ from $C$ for every $1 \leq i \leq r-1$. By applying Proposition \ref{prop:cycle} to the cycle of convex subcomplexes $(N(V_1),N(H_r),N(V_s),C)$, we find a flat rectangle $[0,a] \times [0,b] \subset X$ with $a \geq s$, $b \geq r$ and $[0,a] \times \{ 0 \} \subset C$. By assumption, we know that $r,s \geq n$, so $m= \min(r,s) \geq n$. Let $\gamma_n$ be the concatenation $$\{0 \} \times [0,m] \bigcup [0,m] \times \{m\} \bigcup \{m \} \times [0,m],$$ which links the two points $(0,0)$ and $(m,0)$ of $C$. Now, noticing that $\gamma_n$ is a $(1/3,0)$-quasi-geodesic, and that $d(\gamma_n,C) \geq m \geq n$, since $H_1, \ldots, H_m$ separates $C$ and $[0,m] \times \{m\}$, we conclude that $C$ is not a Morse subspace. A fortiori, $S$ as well.
\end{proof}

\noindent
By combining the previous statements, we get the following criterion:

\begin{cor}\label{cor:MorseCriterion}
Let $G$ be a group acting geometrically on a CAT(0) cube complex $X$ and $H \subset G$ a subgroup. The following statements are equivalent:
\begin{itemize}
	\item[(i)] $H$ is a Morse subgroup;
	\item[(ii)] for every $x \in X$, the orbit $H \cdot x$ is contracting;
	\item[(iii)] for every $x \in X$, the convex hull of the orbit $H \cdot x$ is contracting;
	\item[(iv)] there exists a contracting convex subcomplex on which $H$ acts cocompactly.
\end{itemize}
\end{cor}

\begin{proof}
The equivalence $(i) \Leftrightarrow (ii)$ follows from Milnor-\v Svarc lemma and Lemma \ref{MorseContracting}. The implications $(i) \Rightarrow (iii) \Rightarrow (iv)$ are contained in the proof of Proposition \ref{prop:core} above. Finally, the implication $(iv) \Rightarrow (i)$ is also a consequence of  Milnor-\v Svarc lemma and Lemma \ref{MorseContracting}.
\end{proof}

\begin{app}
Corollary \ref{cor:MorseCriterion} can be applied to extend \cite[Theorem 1.11]{StronglyQC} (in which Morse subgroups are called \emph{strongly quasiconvex subgroups}).

\begin{prop}\label{prop:RACGmorse}
Let $\Gamma$ be a finite simplicial graph and $\Lambda \subset \Gamma$ an induced subgraph. The subgroup $C(\Lambda)$ in the right-angled Coxeter group $C(\Gamma)$ is a Morse subgroup if and only if every induced square of $\Gamma$ containing two diametrically opposite vertices in $\Lambda$ must be included into $\Lambda$.
\end{prop}

\noindent
We recall that the Cayley graph $X(\Gamma)$ of $C(\Gamma)$ constructed from its canonical generating set is naturally the one-skeleton of a CAT(0) cube complex. (More precisely, the Cayley graph is a median graph, and the cube complex $X(\Gamma)$ obtained from it by \emph{filling in the cubes}, i.e., adding an $n$-cube along every induced subgraph isomorphic to the one-skeleton of an $n$-cube, turns out to be a CAT(0) cube complex.) For every vertex $u \in V(\Gamma)$, we denote by $J_u$ the hyperplane dual to the edge joining $1$ and $u$; every hyperplane of $X(\Gamma)$ is a translate of some $J_v$. It is worth noticing that, for every vertices $u,v \in V(\Gamma)$, the hyperplanes $J_u$ and $J_v$ are transverse if and only if $u$ and $v$ are adjacent vertices of $\Gamma$. Finally, if $\Lambda \subset \Gamma$ is an induced subgraph, we denote by $C(\Lambda)$ the subgroup generated by the vertices of $\Lambda$, and by $X(\Lambda) \subset X(\Gamma)$ the convex subcomplex generated by the elements of $C(\Lambda)$. 

\begin{proof}[Proof of Proposition \ref{prop:RACGmorse}.]
Suppose that $C(\Lambda)$ is not a Morse subgroup. According to Corollary \ref{cor:MorseCriterion}, this means that $X(\Lambda)$ is not contracting. Therefore, according to Proposition \ref{prop:contracting}, there exists a grid of hyperplanes $(\mathcal{H}, \mathcal{V})$ satisfying $\mathcal{V} \subset \mathcal{H}( X(\Lambda))$, $\mathcal{V} \cap \mathcal{H}(X(\Lambda)) = \emptyset$, $\# \mathcal{V} > \# V(\Gamma)+2$ and $\mathcal{H} > \# V(\Gamma)+1 $. Write $\mathcal{V}=\{ V_1, \ldots, V_n \}$ such that $V_i$ separates $V_{i-1}$ and $V_{i+1}$ for every $2 \leq i \leq n-1$; and $\mathcal{H}= \{ H_1, \ldots, H_m \}$ such that $H_i$ separates $H_{i-1}$ and $H_{i+1}$ for every $2 \leq i \leq m-1$, and such that $H_1$ separates $X(\Lambda)$ and $H_m$. Consider the cycle of subcomplexes $(N(V_1), C(\Lambda), N(V_n), N(H_m))$. According to Proposition \ref{prop:cycle}, there exists a flat rectangle $[0,p] \times [0,q] \hookrightarrow X$ such that $[0,p] \times \{ 0 \} \subset X(\Lambda)$, $\{0 \} \times [0,q] \subset N(V_1)$, $\{p\} \times [0,q] \subset N(V_n)$ and $[0,p] \times \{q \} \subset N(H_m)$. Since $\# \mathcal{V} > \# V(\Gamma)+2 $ and $\# \mathcal{H} >\# V(\Gamma)+1$, necessarily $p > \# V(\Gamma) $ and $q > \# V(\Gamma) $. 

\medskip \noindent
Let $a_1 \cdots a_q$ denote the word labelling the path $\{0\} \times [0,q]$ (from $(0,0)$ to $(0,q)$), where $a_1, \ldots, a_q \in \Gamma$ are vertices. Because $q > \# V(\Gamma) $, there must exist $1 \leq i <j \leq q$ such that $a_i$ and $a_j$ are not adjacent in $\Gamma$. Without loss of generality, we may suppose that $a_i$ commutes with $a_k$ for every $1 \leq k <i$. It follows that $a_1 \cdots a_{i-1}J_{a_i}=J_{a_i}$. Since $(0,0) \in C(\Lambda)$ but $J_{a_i} \notin  \mathcal{H}(X(\Lambda))$ according to Proposition \ref{prop:cycle}, it follows that $a_i \notin \Lambda$. 

\medskip \noindent
Similarly, because $p > \# V(\Gamma)$, there must exist two edges of $[0,p] \times \{ 0 \} \subset X(\Lambda)$ labelled by non-adjacent vertices of $\Lambda$, say $u$ and $v$. By noticing that any hyperplane intersecting $[0,p] \times \{0\}$ must be transverse to any hyperplane intersecting $\{0\} \times [0,q]$, it follows that $u$ and $v$ are adjacent to both $a_i$ and $a_j$. Otherwise saying, $a_i,a_j,u,v$ define an induced square of $\Gamma$ such that $u,v \in \Lambda$ are diametrically opposite but $a_i \notin \Lambda$. 

\medskip \noindent
Conversely, suppose that there exists some induced square in $\Gamma$ with two diametrically opposite vertices $u$ and $v$ in $\Lambda$ but with one of its two other vertices, say $a$, not in $\Lambda$. Let $b$ denote the fourth vertex of our square. Consider the two infinite rays
$$1, \ u, \ uv, \ (uv)u, \ (uv)^2, \ldots, \ (uv)^n, \ldots$$
and
$$1, \ a, \ ab, \ (ab)a, \ (ab)^2, \ldots, \ (ab)^n, \ldots$$
say $r_1$ and $r_2$ respectively. Since $u$ and $v$ commute with both $a$ and $b$, it follows that $r_1$ and $r_2$ bound a copy of $[0,+ \infty) \times [0,+ \infty)$ (which is generated by the vertices $gh$ where $g$ and $h$ are prefixes of the infinite words $(uv)^\infty$ and $(ab)^{\infty}$ respectively). As a consequence, for every $n \geq 1$, any hyperplane of $\mathcal{H}_n= \{ (ab)^kJ_a \mid k \leq n \}$ is transverse to any hyperplane of $\mathcal{V}_n = \{ (uv)^k J_u \mid k \leq n\}$. Moreover, notice that $\mathcal{H}_n$ and $\mathcal{V}_n$ do not contain facing triples since they are collections of hyperplanes transverse to the geodesic rays $r_2$ and $r_1$ respectively; and $\mathcal{H}_n \cap \mathcal{H}(X(\Lambda)) = \emptyset$ since $a \notin \Lambda$; and of course $\mathcal{V}_n \subset \mathcal{H}(X(\Lambda))$. It follows from Proposition \ref{prop:contracting} that $X(\Lambda)$ is not contracting, so that $C(\Lambda)$ is not a Morse subgroup according to Corollary \ref{cor:MorseCriterion}. 
\end{proof}
\end{app}

\begin{app}\label{app:MorseRaag}
Working harder, one can show that Morse subgroups in freely irreducible right-angled Artin groups are either finite-index subgroups or free subgroups. We defer the proof to Appendix \ref{section:MorseInRAAG}.
\end{app}

\section{Relative hyperbolicity}

\noindent
In this section, we are interested in the following question: given a group acting geometrically on a CAT(0) cube complex, how to determine whether or not it is relatively hyperbolic? The definition of relative hyperbolicity which we use is the following:

\begin{definition}
A finitely-generated group $G$ is \emph{hyperbolic relative to a collection of subgroups $\mathcal{H}=\{ H_1, \ldots, H_n \}$} if $G$ acts by isometries on a graph $\Gamma$ such that:
\begin{itemize}
	\item $\Gamma$ is Gromov hyperbolic,
	\item $\Gamma$ contains finitely-many orbits of edges,
	\item each vertex-stabilizer is either finite or is conjugate to some $H_i$,
	\item any $H_i$ stabilizes a vertex,
	\item $\Gamma$ is \emph{fine}, i.e., any edge belongs only to finitely-many simple loops (or \emph{cycle}) of a given length.
\end{itemize}
A subgroup conjugate to some $H_i$ is \emph{peripheral}. $G$ is just said \emph{relatively hyperbolic} if it is relatively hyperbolic with respect to a finite collection of proper subgroups.
\end{definition}

\noindent
We refer to \cite{HruskaRH} and references therein for more information on relatively hyperbolic groups. Our main criterion is the following, which is essentially extracted from \cite{coningoff}.

\begin{thm}\label{thm:relhyp}
Let $G$ be a group acting geometrically on some CAT(0) cube complex $X$. Then $G$ is relatively hyperbolic if and only if there exists a collection of convex subcomplexes $\{Y_1, \ldots, Y_n\}$ satisfying the following conditions:
\begin{itemize}
	\item $\mathrm{stab}(Y_i)$ acts geometrically on $Y_i$ for every $1 \leq i \leq n$;
	\item there exists a constant $C_1 \geq 0$ such that, for every $1 \leq i,j \leq n$, any two distinct translates of $Y_i$ and $Y_j$ are both transverse to at most $C_1$ hyperplanes;
	\item there exists a constant $C_2 \geq 0$ such that any $C_2$-thick flat rectangle lies in the $C_2$-neighborhood of a translate of $Y_i$ for some $1 \leq i \leq n$.
\end{itemize}
Moreover, the last point can be replaced with:
\begin{itemize}
	\item there exists a constant $C_2 \geq 0$ such that the image of every combinatorial isometric embedding $\mathbb{R}^2 \hookrightarrow X$ is included into the $C_2$-neighborhood of a translate of $Y_i$ for some $1 \leq i \leq n$. 
\end{itemize}
If these conditions are satisfied, then $G$ is hyperbolic relative to $\{ \mathrm{stab}(Y_i) \mid 1 \leq i \leq n \}$. 
\end{thm}

\noindent
In other contexts, similar statements can be found in \cite{isolated} and \cite{isolatedflatsErratum}. We begin by recalling \cite[Lemma 8.6]{MoiSpecialBraid}, which will be useful in the proof of our theorem.

\begin{lemma}\label{lem:fellowtravel}
Let $X$ be a CAT(0) cube complex and $A,B \subset X$ two $L$-contracting convex subcomplexes. Suppose that any vertex of $X$ has at most $R \geq 2$ neighbors. If there exist $N \geq \max(L,2)$ hyperplanes transverse to both $A$ and $B$, then the inequality
$$\mathrm{diam} \left( A^{+L} \cap B^{+L} \right) \geq \ln(N-1)/\ln(R)$$
holds. 
\end{lemma}

\begin{proof}[Proof of Theorem \ref{thm:relhyp}.]
Suppose that $G$ is hyperbolic relative to $\mathcal{H}=\{H_1, \ldots, H_n\}$. Fix a basepoint $x \in X$, and let $\mathcal{C}= \{ gH_i' \cdot x \mid 1 \leq i \leq n, g \in G\}$. According to \cite[Theorems A.1 and 5.1]{TreeGraded}, $X$ is asymptotically tree-graded with respect to $\mathcal{C}$. Moreover, it follows from \cite[Lemma 4.15]{TreeGraded} that any element of $\mathcal{C}$ is a Morse subset of $X$; as a consequence of Lemma \ref{lem:convexhull}, $X$ is also asymptotically tree-graded with respect to $\mathcal{D}= \{ \text{convex hull of $C$} \mid C \in \mathcal{C} \}$. For every $1 \leq i \leq n$, let $Y_i \in \mathcal{D}$ denote the convex hull of the orbit $H_i \cdot x$; since the Hausdorff distance between $H_i \cdot x$ and $Y_i$ is finite, $H_i$ acts geometrically on $Y_i$; a fortiori, $\mathrm{stab}(Y_i)$ acts geometrically on $Y_i$. 

\medskip \noindent
We know from Condition $(\alpha_1)$ in \cite[Theorem 4.1]{TreeGraded} that, for every $\delta$, there exists some constant $K$ such that $\mathrm{diam} (A^{+\delta} \cap B^{+ \delta}) \leq K$ for every distinct $A,B \in \mathcal{D}$. It follows from Lemma \ref{lem:fellowtravel} that there exists a constant $C_1 \geq 0$ such that, for every $1 \leq i,j \leq n$, any two distinct translates of $Y_i$ and $Y_j$ are both transverse to at most $C_1$ hyperplanes. Next, we know from Condition $(\alpha_3)$ of \cite[Theorem 4.1]{TreeGraded} that there exists a constant $C_2 \geq 0$ such that any $C_2$-thick flat rectangle of $X$ is included into the $C_2$-neighborhood of some element of $\mathcal{C}$. Finally, by combining Conditions $(\alpha_1)$ and $(\alpha_2)$, we deduce that there exists a constant $C_3 \geq 0$ such that the image of every combinatorial isometric embedding $\mathbb{R}^2 \hookrightarrow X$ is included into the $C_3$-neighborhood of an element of $\mathcal{C}$.

\medskip \noindent
Conversely, suppose that there exists a collection of convex subcomplexes $\{ Y_1, \ldots, Y_n\}$ satisfying the first three conditions mentioned in our theorem. Let $\dot{X}$ denote the graph obtained from the one-skeleton of $X$ by adding, for any translate $Y$ of some $Y_i$, a new vertex $v_Y$ and edges from $v_Y$ to any vertex of $Y$. According to \cite[Theorems 4.1 and 5.7]{coningoff}, $\dot{X}$ is a fine hyperbolic graph on which $G$ acts. Notice that, as a consequence of \cite[Lemma 8.8]{MoiSpecialBraid}, our third condition can be replaced with the last point in our statement. Because $G$ acts geometrically on $X$, we also know that $\dot{X}$ contains finitely many orbits of edges, and that stabilisers of vertices of $X$ are finite. Consequently, $G$ is hyperbolic relative to $\{ \mathrm{stab}(Y_i) \mid 1 \leq i \leq n \}$. 
\end{proof}

\begin{app}
Theorem \ref{thm:relhyp} has been applied in \cite{coningoff} to determine precisely when a right-angled Coxeter group $C(\Gamma)$ is relatively hyperbolic, just by looking at the graph $\Gamma$. (This characterisation was originally proved in \cite{BHSC}.) Moreover, in that case, we get a minimal collection of peripheral subgroups of $C(\Gamma)$. (See also \cite[Theorem 8.33]{Qm} for a generalisation of the argument to arbitrary graph products.)
\end{app}

\begin{app}
Thanks to Theorem \ref{thm:relhyp}, a sufficient criterion of relative hyperbolicity of graph braid groups $B_2(\Gamma)$ was obtained in \cite[Theorem 9.41]{MoiSpecialBraid}. For instance, if $\Gamma$ is the union of two bouquets of circles whose centers are linked by a segment, then $B_2(\Gamma)$ is hyperbolic relative to subgroups isomorphic to direct products of free groups. A full characterisation of relatively hyperbolic graph braid groups remains unknown.
\end{app}

\noindent
In general, Theorem \ref{thm:relhyp} is difficult to apply, essentially because one has to guess the peripheral subgroups and the convex subcomplexes on which they act. For instance, determining which graph braid groups (see Application \ref{app:braidgrouphyp}) are relatively hyperbolic is an open question; see \cite{MoiSpecialBraid}. So finding other criteria is an interesting problem. 

\begin{problem}\label{problem:relahyp}
Find criteria of relative hyperbolicity of groups acting geometrically on CAT(0) cube complexes which do not refer to peripheral subgroups.
\end{problem}

\noindent
Interestingly, in the context of virtually (cocompact) special groups (as defined by Haglund and Wise in \cite{MR2377497}), Theorem \ref{thm:relhyp} provides the following more algebraic statement, as shown in \cite[Theorem 8.1]{MoiSpecialBraid}.

\begin{thm}\label{thm:specialrelhyp}
Let $G$ be a special group and $\mathcal{H}$ a finite collection of subgroups. Then $G$ is hyperbolic relative to $\mathcal{H}$ if and only if the following conditions are satisfied:
\begin{itemize}
	\item each subgroup of $\mathcal{H}$ is convex-cocompact;
	\item $\mathcal{H}$ is an almost malnormal collection (i.e., for every $H,K \in \mathcal{H}$ and $g \in G$, if $H \cap gKg^{-1}$ is infinite then $H=K=gKg^{-1}$);
	\item every abelian subgroup of $G$ which is not virtually cyclic is contained into a conjugate of some group of $\mathcal{H}$. 
\end{itemize}
\end{thm}

\noindent
It is worth noticing that the characterisation of relatively hyperbolic right-angled Coxeter groups (proved in \cite{BHSC, coningoff}), and more generally the characterisation of relatively hyperbolic graph products of finite groups (which is a particular case of \cite[Theorem 8.33]{Qm}), follows easily from Theorem \ref{thm:specialrelhyp}. However, this criterion does not provide a purely algebraic characterisation of relatively hyperbolic virtually special groups, since the subgroups need to be convex-cocompact. But, convex-cocompactness is not an algebraic property: with respect to the canonical action $\mathbb{Z}^2 \curvearrowright \mathbb{R}^2$, the cyclic subgroup generated by $(0,1)$ is convex-cocompact, whereas the same subgroup is not convex-cocompact with respect to the action $\mathbb{Z}^2 \curvearrowright \mathbb{R}^2$ defined by $(0,1) : (x,y) \mapsto (x+1,y+1)$ and $(1,0) : (x,y) \mapsto (x+1,y)$. Nevertheless, the convex-cocompactness required in the previous statement would be unnecessary if the following question admits a positive answer (at least in the context of special groups):

\begin{question}\label{question:malnormalMorse}
Let $G$ be a group acting geometrically on a CAT(0) cube complex and $H \subset G$ a finitely generated subgroup. If $H$ is almost malnormal, must it be a Morse subgroup?
\end{question}

\noindent
As a consequence of the previous theorem, one gets the following simple characterisation of virtually special groups which are hyperbolic relative to virtually abelian subgroups \cite[Theorem 8.14]{MoiSpecialBraid}:

\begin{thm}\label{thm:RHspecialAbelian}
Let $G$ be a virtually special group. Then $G$ is hyperbolic relative to virtually abelian groups if and only if $G$ does not contain $\mathbb{F}_2 \times \mathbb{Z}$ as a subgroup.
\end{thm}

\begin{app}
Theorem \ref{thm:RHspecialAbelian} was applied in \cite{MoiSpecialBraid} to determine precisely when a given graph braid group is hyperbolic relative to abelian subgroups. As a particular case, if $\Gamma$ is a connected finite graph, the braid group $B_2(\Gamma)$ is hyperbolic relative to abelian subgroup if and only if $\Gamma$ does not contain a cycle which is disjoint from two other cycles. 
\end{app}

\noindent
In another direction, a very interesting attempt to study relative hyperbolicity from the (simplicial) boundary has been made in \cite[Theorems 3.1 and 3.7]{ThicknessBH}. However, such a criterion seems to be highly difficult to apply. We conclude this section with an open question in the spirit of Problem \ref{problem:relahyp}.

\begin{question}
If a group acting geometrically on a CAT(0) cube complex has exponential divergence, must it be relatively hyperbolic?
\end{question}

\noindent
It was observed in \cite{BHSC} that the answer is positive for right-angled Coxeter groups. %The previous question is also motivated by Gromov's definition of \emph{conformally hyperbolic} groups \cite[Appendix to \S 8]{GromovAs}: a finitely generated group is \emph{conformally hyperbolic} if its Floyd boundary contains at least three points. Because the divergence of a conformally hyperbolic group must be exponential, a positive answer to the previous question would imply that conformally hyperbolic and relatively hyperbolic groups acting geometrically on CAT(0) cube complexes coincide. 

\section{Acylindrical hyperbolicity}

\subsection{Acylindrical actions}

\noindent
From now on, we are interested in the following question: how can one prove that a given group is \emph{acylindrically hyperbolic} from an action on a CAT(0) cube complex? Let us begin by recalling Osin's definition of \emph{acylindrically hyperbolic} groups \cite{OsinAcyl}. 

\begin{definition}
Let $G$ be a group acting on a metric space $X$. The action is \emph{acylindrical} if, for every $d \geq 0$, there exist constants $N,R \geq 0$ such that, for every points $x,y \in X$ at distance at least $R$ apart, the set $\{ g \in G \mid d(x,gx),d(y,gy) \leq d\}$ has cardinality at most $N$. 
A group is \emph{acylindrically hyperbolic} if it acts acylindrically and non-elementarily (i.e., with a limit set containing at least three points) on some hyperbolic space. 
\end{definition}

\noindent
So, in order to prove that a given group is acylindrically hyperbolic, one possibility is to try to make it act acylindrically on some hyperbolic CAT(0) cube complex. However, it is often difficult to show that a given action is acylindrical. The main reason is that we are considering the elements of a group which do not move ``too much'' a given pair of points. Instead, it would be easier to consider stabilisers. More precisely, a property which should be easier to prove would be:

\begin{definition}
Let $G$ be a group acting on a metric space $X$. The action is \emph{weakly acylindrical} if there exist constants $N,R \geq 0$ such that, for every points $x,y \in X$ at distance at least $R$ apart, the intersection $\mathrm{stab}(x) \cap \mathrm{stab}(y)$ has cardinality at most $N$. 
\end{definition}

\noindent
Interestingly, it may happen that, for some specific spaces, acylindrical and weakly acylindrical actions coincide. For instance, such an equivalence occurs for trees. The first non-trivial example of this phenomenon appears in \cite{BowditchAcyl}, in which Bowditch shows that the mapping class group of a surface acts acylindrically on the associated curve graph. Independently, Martin observed the same phenomenon for hyperbolic CAT(0) square complexes \cite{articleMartin}. This statement was generalised to higher dimensions in \cite[Theorem 8.33]{coningoff}. One gets:

\begin{thm}\label{thm:acylcriterion}
Let $G$ be a group acting on some hyperbolic CAT(0) cube complex $X$. The following statements are equivalent:
\begin{itemize}
	\item the action $G \curvearrowright X$ is acylindrical;
	\item there exist constants $R,N \geq 0$ such that, for every vertices $x,y \in X$ satisfying $d(x,y) \geq R$, the intersection $\mathrm{stab}(x) \cap \mathrm{stab}(y)$ has cardinality at most $R$;
	\item there exist constants $R,N \geq 0$ such that, for every hyperplanes $J_1,J_2$ of $X$ separated by at least $R$ hyperplanes, the intersection $\mathrm{stab}(J_1) \cap \mathrm{stab}(J_2)$ has cardinality at most $N$. 
\end{itemize}
\end{thm}

\begin{app}
Introduced in \cite{HigmanGroups}, \emph{Higman's group} $H_n$ is defined as
$$H_n = \langle a_1, \ldots, a_n \mid a_i a_{i+1} a_i^{-1}= a_{i+1}^2, \ i \in \mathbb{Z}/n \mathbb{Z} \rangle.$$
This group turns out to be the fundamental group of a negatively-curved polygon of groups if $n \geq 5$, so that $H_n$ acts (with infinite vertex-stabilisers) on a CAT(-1) polygonal complex $X$. In \cite{articleMartin}, Martin subdivided $X$ as a CAT(0) square complex and applied Theorem \ref{thm:acylcriterion} to show that the action $G \curvearrowright X$ is acylindrical. A fortiori, this proves that Higman's group $H_n$ is acylindrically hyperbolic. 
\end{app}

\noindent
So far, we have worked with CAT(0) cube complexes which are hyperbolic with respect to the $\ell^1$-metric. But, as noticed in Section \ref{section:Gromovhyp}, infinite-dimensional CAT(0) cube complexes may be hyperbolic with respect to the $\ell^{\infty}$-metric. Acylindrical actions in this context were considered in \cite{coningoff}. However, we were not able to obtain the exact analogue of Theorem \ref{thm:acylcriterion}: instead, acylindrical actions were replaced with \emph{non-uniformly acylindrical} actions.

\begin{definition}
Let $G$ be a group acting on a metric space $X$. The action is \emph{non-uniformly acylindrical} if, for every $d \geq 0$, there exists some constant $R \geq 0$ such that, for every points $x,y \in X$ at distance at least $R$ apart, the set $\{ g \in G \mid d(x,gx),d(y,gy) \leq d\}$ is finite. 
\end{definition}

\noindent
Notice that, if a group $G$ acts non-elementarily and non-uniformly acylindrically on some hyperbolic space, then $G$ must be acylindrically hyperbolic according to \cite{OsinAcyl} since $G$ contains infinitely many pairwise independent \emph{WPD isometries} (see Section \ref{section:WPD} for a definition). Our analogue of Theorem \ref{thm:acylcriterion} for $\ell^{\infty}$-metrics is:

\begin{thm}\label{thm:nonunifacyl}
Let $G$ be a group acting on some CAT(0) cube complex $X$ endowed with the $\ell^{\infty}$-metric. Suppose that $X$ is hyperbolic with respect to this metric, which we denote by $d_{\infty}$. The following statements are equivalent:
\begin{itemize}
	\item the action $G \curvearrowright (X,d_{\infty})$ is non-uniformly acylindrical;
	\item there exist some constant $R \geq 0$ such that, for every vertices $x,y \in X$ satisfying $d_{\infty}(x,y) \geq R$, the intersection $\mathrm{stab}(x) \cap \mathrm{stab}(y)$ is finite;
	\item there exist some constant $R \geq 0$ such that, for every hyperplanes $J_1,J_2$ of $X$ separated by at least $R$ pairwise disjoint hyperplanes, the intersection $\mathrm{stab}(J_1) \cap \mathrm{stab}(J_2)$ is finite. 
\end{itemize}
\end{thm}

\begin{app}\label{app:smallcancellationproduct}
Define a \emph{small cancellation product} as a C'(1/4)-T(4) small cancellation quotient of a free product. As mentioned in Application \ref{app:smallcancellationhyp}, such a product acts on a (possibly infinite-dimensional) CAT(0) cube complex which is hyperbolic with respect to the $\ell^{\infty}$-metric. In \cite[Theorem 8.23]{coningoff}, it is shown thanks to Theorem \ref{thm:nonunifacyl} that this action is non-uniformly acylindrical, proving that small cancellation products are acylindrically hyperbolic. 
\end{app}

\subsection{Contracting isometries}\label{section:contractingisom}

\noindent
Interestingly, non-hyperbolic CAT(0) cube complexes may also be useful to prove that some groups are acylindrically hyperbolic. Indeed, it follows from \cite[Theorem H]{BBF} (see also \cite{arXiv:1112.2666} and Corollary \ref{cor:altBBF}) that a group acting properly on a CAT(0) cube complex $X$ with a \emph{contracting isometry} must be either virtually cyclic or acylindrically hyperbolic. An isometry $g \in \mathrm{Isom}(X)$ is \emph{contracting} if there exists some $x \in X$ such that the map $n \mapsto g^n \cdot x$ induces a quasi-isometric embedding $\mathbb{Z} \to X$ and such that the orbit $\langle g \rangle \cdot x$ is contracting. 

\medskip \noindent
So the first natural question which interests us is: how to recognize contracting isometries of CAT(0) cube complexes? The first partial answer was given in \cite{BehrstockCharney} in the context of right-angled Artin groups; next, the criterion was generalised in \cite{MR3339446} to uniformly locally finite CAT(0) cube complexes; finally, the following statement was proved in \cite{article3}. (It is worth noticing that, although \cite{BehrstockCharney, MR3339446} and \cite{article3} study contracting isometries with respect to different metrics (respectively the CAT(0) and the combinatorial metrics), a comparison of the characterisations shows that, given a finite-dimensional CAT(0) cube complex, an isometry is contracting with respect to the CAT(0) metric if and only if it is contracting with respect to the combinatorial metric.)

\begin{thm}\label{thm:contractingisom}
Let $X$ be a CAT(0) cube complex and $g \in \mathrm{Isom}(X)$ a loxodromic isometry with $\gamma \subset X$ as a combinatorial axis. The following statements are equivalent:
\begin{itemize}
	\item $g$ is a contracting isometry;
	\item there exists some constant $C\geq 0$ such that every join of hyperplanes $(\mathcal{H}, \mathcal{V})$ satisfying $\mathcal{H} \subset \mathcal{H}(\gamma)$ must be $C$-thin;
	\item $g$ skewers a pair of well-separated hyperplanes.
\end{itemize}
\end{thm}

\noindent
We recall from \cite{article3} that two hyperplanes $J_1$ and $J_2$ are \emph{well-separated} if there exists some $L \geq 0$ such that any collection of hyperplanes transverse to both $J_1$ and $J_2$ which does not contain facing triples has cardinality at most $L$. Also, an isometry $g \in \mathrm{Isom}(X)$ \emph{skewers} a pair of hyperplanes $J_1$ and $J_2$ if there exist an integer $n \in \mathbb{Z}$ and some halfspaces $D_1,D_2$ respectively delimited by $J_1,J_2$ such that $g^n \cdot D_1 \subset D_2 \subset D_1$. 

\begin{app}
Applying (a special case of) Theorem \ref{thm:contractingisom}, it is proved in \cite{BehrstockCharney} that an element of a right-angled Artin group induced a contracting isometry on the universal cover of the Salvetti complex if and only if it is not contained into a \emph{join} subgroup. (Notice that a flaw in \cite{BehrstockCharney} is mentioned and corrected in \cite[Remark 6.21]{MinasyanOsin}.) As a consequence, a right-angled Artin group $A(\Gamma)$ is acylindrically hyperbolic if and only if $\Gamma$ contains at least two vertices and does not decompose as a \emph{join}; or equivalently, if $A(\Gamma)$ is not cyclic and does not decompose as a direct product. 
\end{app}

\noindent
Alternatively, it is possible to characterise contracting isometries from the boundary of the CAT(0) cube complex which we consider. Based on this idea, the following criterion was proved in \cite{article3}. We refer respectively to Appendix \ref{section:MorseInRAAG} and to \cite{simplicialboundary} for the vocabulary related to the combinatorial boundary and to the simplicial boundary. 

\begin{thm}\label{thm:contractingboundary}
Let $X$ be a locally finite CAT(0) cube complex and $g \in \mathrm{Isom}(X)$ an isometry with a combinatorial axis $\gamma$. The following statements are equivalent:
\begin{itemize}
	\item $g$ is a contracting isometry of $X$;
	\item $\gamma(+ \infty)$ is an isolated point in the combinatorial boundary of $X$; 
\end{itemize}
moreover, if $X$ is uniformly locally finite, the previous statements are also equivalent to:
\begin{itemize}
	\item $\gamma(+ \infty)$ is an isolated point in the simplicial boundary of $X$.
\end{itemize}
\end{thm}

\begin{proof}
The equivalence between the first two points is \cite[Theorem 4.17]{article3}. The equivalence with the third point follows from \cite[Lemma 5.2.7]{Hagenthesis}. 
\end{proof}

\begin{app}
Let $\mathcal{P}= \langle \Sigma \mid \mathcal{R} \rangle$ be a semigroup presentation. The associated \emph{Squier complex} $S(\mathcal{P})$ is the square complex whose:
\begin{itemize}
	\item vertices are the words written over $\Sigma$;
	\item edges are written as $[a,u \to v,b]$ and link two words $aub$ and $avb$ if $u=v$ or $v=u$ is a relation of $\mathcal{R}$;
	\item squares are written as $[a,u \to v,b, p\to q,c]$ and have $aubpc$, $avbpc$, $aubqc$ and $avbqc$ as corners.
\end{itemize}
Given a baseword $w \in \Sigma^+$, the \emph{diagram group} $D(\mathcal{P},w)$ is the fundamental group of $S(\mathcal{P})$ based at $w$. We refer to \cite{MR1396957} for more information on these groups. Theorem~\ref{thm:contractingboundary} was applied to diagram groups in \cite{article3}. As a consequence, an easy method is given to determine whether or not an element of a diagram group induces a contracting isometry on the CAT(0) cube complex constructed by Farley \cite{MR1978047}. A characterisation of acylindrically hyperbolic \emph{cocompact diagram groups} is also provided.
\end{app}

\noindent
Of course, a natural question is: when does a given action on a CAT(0) cube complex contain a contracting isometry? Our following criterion was proved in Caprace and Sageev's seminal paper \cite[Theorem 6.3]{MR2827012}. 

\begin{thm}\label{thm:CS}
Let $G$ be a group acting essentially without fixed point at infinity on some finite-dimensional CAT(0) cube complex $X$. Either $X$ is a product two unbounded subcomplexes or $G$ contains a contracting isometry. If in addition $X$ is locally finite and $G$ acts cocompactly, then the same conclusion holds even if $G$ fixes a point at infinity.
\end{thm}

\noindent
Recall that an action $G \curvearrowright X$ on a CAT(0) cube complex $X$ is \emph{essential} if, for every point $x \in X$ and every halfspace $D$, the orbit $G \cdot x$ does not lie in a neighborhood of $D$. It is worth noticing that an action can often be made essential thanks to \cite[Proposition 3.5]{MR2827012}. The boundary which is considered in this statement is the CAT(0) boundary; see \cite[Proposition 2.26]{CIF} to compare with the Roller boundary. For the combinatorial boundary, see \cite{article3}, where it is proved that, under some assumptions on the action, the existence of contracting isometries is equivalent to the existence of an isolated point in the combinatorial boundary. Also, it is worth noticing that, if a group acts on a CAT(0) cube complex with a contracting isometry, then it does fix a point at infinite, as a consequence of the North-South dynamic of contracting isometries; this justifies the corresponding assumption in Caprace and Sageev's statement. 

\begin{cor}\label{cor:acylhypiffcontracting}
Let $G$ be a group acting geometrically on a CAT(0) cube complex $X$. Then $G$ is acylindrically hyperbolic if and only if it is not virtually cyclic and it contains an element inducing a contracting isometry of $X$. 
\end{cor}

\begin{proof}
As a consequence of \cite[Proposition 3.5]{MR2827012}, we may suppose without loss of generality that the action $G \curvearrowright X$ is special. By applying Theorem \ref{thm:CS}, two cases may happen. Either $G$ contains a contracting isometry, so that $G$ must be either virtually cyclic or acylindrically hyperbolic; or $X$ decomposes as a product of two unbounded subcomplexes. In the latter case, it follows that $G$ unconstricted, i.e., $G$ has not cut points in its asymptotic cones, which implies that $G$ is not acylindrically hyperbolic according to \cite{Sistohypembed}. (Alternatively, we can argue that $G$ has linear divergence, which also implies that it cannot be acylindrically hyperbolic.)
\end{proof}

\noindent
Therefore, contracting isometries play a crucial role in the geometry of groups acting (geometrically) on CAT(0) cube complexes. An interesting problem would be to identify these elements purely algebraically.

\begin{problem}
Let $G$ be a group acting geometrically on some CAT(0) cube complex $X$. Characterize algebraically the elements of $G$ inducing contracting isometries on $X$.
\end{problem}

\noindent
An investigation of the examples mentioned in this article suggests the following answer. Let $G$ be a group acting geometrically on a CAT(0) cube complex. Fix an infinite-order element $g$ and define its \emph{stable centraliser} as
$$SC(g)= \{ h\in G \mid \exists n \in \mathbb{Z} \backslash \{0 \}, \ [h,g^n]=1 \}.$$ 
Is it true that $g$ induces a contracting isometry on $X$ if and only if $SC(g)$ is virtually cyclic? Although the answer is negative in full generality, it turns out to be positive for several families of cube complexes. For instance:

\begin{thm}\label{thm:Egspecial}
Let $G$ be a group acting geometrically on a CAT(0) cube complex $X$. Assume that, for every hyperplane $J$ and every element $g \in G$, the two hyperplanes $J$ and $gJ$ are neither transverse nor tangent. Then an infinite-order element of $G$ defines a contracting isometry of $X$ if and only its stable centraliser is virtually cyclic. 
\end{thm}

\noindent
The scope of this theorem, proved in \cite{ContractingCentraliser}, includes for instance cocompact special groups as defined in \cite{MR2377497}.

\begin{app}
Theorem \ref{thm:Egspecial} was applied to graph braid groups in \cite{MoiSpecialBraid}. As a consequence, it is possible to determine precisely when a given graph braid group is acylindrically hyperbolic. In particular, if $\Gamma$ is any connected topological graph, distinct from a cycle and from a star with three arms, then the braid group $B_n(\Gamma)$ is acylindrically hyperbolic for every $n \geq 1$. 
\end{app}

\subsection{WPD isometries}\label{section:WPD}

\noindent
In the previous section, we mentioned \cite[Theorem H]{BBF} in order to justify the acylindrical hyperbolicity of groups (which are not virtually cyclic) acting properly on CAT(0) cube complexes with one contracting isometry. But the conclusion is in fact more general, allowing actions with large stabilisers. It turns out that groups (which are not virtually cyclic) acting on CAT(0) cube complexes with one \emph{WPD} contracting isometry are acylindrically hyperbolic. (See also Corollary \ref{cor:altBBF}.)

\begin{definition}
Let $G$ be a group acting on a metric space $X$. An element $g \in G$ is \emph{WPD} (for \emph{Weak Proper Discontinuous}) if, for every $d \geq 0$ and every $x \in X$, there exists some $N \geq 0$ such that the set $\{ h \in G \mid d(x,hx),d(g^Nx,hg^Nx) \leq d\}$ is finite. 
\end{definition}

\noindent
This motivates the following question: when is a contracting isometry WPD? The following answer was proved in \cite{MoiAcylHyp}; compare with Theorem \ref{thm:contractingisom}. 

\begin{thm}\label{thm:WPDcontracting}
Let $X$ be a CAT(0) cube complex and $g \in \mathrm{Isom}(X)$ an isometry. Then $g$ is a contracting isometry if and only if $g$ skewers a pairs of well-separated hyperplanes $J_1,J_2$ such that $\mathrm{stab}(J_1) \cap \mathrm{stab}(J_2)$ is finite.
\end{thm}

\noindent
So we know how to recognize WPD contracting isometries. But now we want to be able to show that such isometries exist. The first result in this direction was obtained in \cite{MinasyanOsin} in the context of trees (which are one-dimensional CAT(0) cube complexes, and hyperbolic so that every loxodromic isometry turns out to be contracting). 

\begin{thm}\label{thm:MO}
Let $G$ be a group acting minimally on a simplicial trees $T$. Suppose that $G$ does not fix any point of $\partial T$. If there exist two vertices $u,v \in T$ such that $\mathrm{stab}(u) \cap \mathrm{stab}(v)$ is finite, then $G$ contains a WPD isometry. A fortiori, $G$ is either virtually cyclic or acylindrically hyperbolic. 
\end{thm}

\noindent
Combined with Bass-Serre theory, this criterion turns out to be extremely fruitful. 

\begin{app}
As shown in \cite{SQonerelator}, one-relator groups with at least three generators split as HNN extensions. In \cite{MinasyanOsin}, Theorem \ref{thm:MO} is applied to the action on the corresponding Bass-Serre tree. Thus, one-relator groups with at least three generators are acylindrically hyperbolic.
\end{app}

\begin{app}
For any field $k$, let $k[x,y]$ denote the algebra of polynomials on two variables with coefficients in $k$. It is known that $k[x,y]$ splits as an amalgamated product, see for instance \cite{AutoPolynomial}; and Theorem \ref{thm:MO} is applied to the action on the corresponding Bass-Serre tree in \cite{MinasyanOsin}. Thus, the group $k[x,y]$ is acylindrically hyperbolic.
\end{app}

\begin{app}
Let $M$ be a compact irreducible 3-manifold and $G$ (a subgroup of) the fundamental group of $M$. By applying Theorem \ref{thm:MO} to the action of $G$ on the Bass-Serre tree associated to the JSJ-decomposition of $M$, it is proved in \cite{MinasyanOsin} that three exclusive cases may happen: $G$ is acylindrically hyperbolic; or $G$ is virtually polycyclic; or $G$ contains an infinite cyclic normal subgroup $Z$ such that $G/Z$ is acylindrically hyperbolic.
\end{app}

\begin{app}
Let $\Gamma$ be simplicial graph with at least two vertices and $\mathcal{G}$ a collection of non-trivial groups indexed by $V(\Gamma)$. To any vertex of $\Gamma$ corresponds a natural decomposition of the graph product $\Gamma \mathcal{G}$ as an amalgamated product. By applying Theorem \ref{thm:MO} to the collection of actions of $\Gamma \mathcal{G}$ on the corresponding Bass-Serre trees, it is proved in \cite{MinasyanOsin} that $\Gamma \mathcal{G}$ is virtually cyclic or acylindrically hyperbolic if $\Gamma$ does not split as a join. 
\end{app}

\noindent
Theorem \ref{thm:MO} was generalised in \cite{ChatterjiMartin} to higher dimensional CAT(0) cube complexes which are ``barely'' hyperbolic, i.e., which does not split as a Cartesian product (seeing this property as a hyperbolic behavior is motivated by Theorem \ref{thm:CS}). (An alternative proof of the next statement, based on Theorem \ref{thm:WPDcontracting}, can be found in \cite{MoiAcylHyp}.)

\begin{thm}\label{thm:CM1}
Let $G$ be a group acting essentially and non-elementarily on an irreducible finite-dimensional CAT(0) cube complex. If there exist two hyperplanes whose stabilisers intersect along a finite subgroup, then $G$ contains a WPD element which skewers a pair of \"uber-separated hyperplanes. A fortiori, $G$ is acylindrically hyperbolic. 
\end{thm}

\noindent
Two hyperplanes $J_1$ and $J_2$ are \emph{\"uber-separated} if no hyperplane is transverse to both of them and if any two hyperplanes transverse to $J_1,J_2$ respectively must be disjoint. Notice that it follows from Theorem \ref{thm:contractingisom} that an isometry which skewers a pair of \"uber-separated hyperplanes must be contracting since two \"uber-separated hyperplanes are clearly well-separated. An interesting consequence of Theorem \ref{thm:CM1} is:

\begin{cor}
Let $G$ be a group acting essentially and non-elementarily on an irreducible finite-dimensional CAT(0) cube complex. If the action is non-uniformly weakly acylindrical, then $G$ is acylindrically hyperbolic. 
\end{cor}

\noindent
An action of a group $G$ on a metric space $X$ is \emph{non-uniformly weakly acylindrical} if, for every $d \geq 0$, there exists some constant $R \geq 0$ such that, for every points $x,y \in X$ at distance at least $R$ apart, the intersection $\mathrm{stab}(x) \cap \mathrm{stab}(y)$ is finite. Notice that we met this condition in Theorem \ref{thm:nonunifacyl}. 

\begin{app}
Let $\Gamma$ be a \emph{Coxeter graph}, i.e., a finite simplicial graph endowed with a map $m : E(\Gamma) \to \mathbb{N}$ labelling its edges. The corresponding \emph{Artin group} is defined by the presentation
$$A= \langle V(\Gamma) \mid \underset{\text{$m(u,v)$ letters}}{\underbrace{uvu \cdots}} = \underset{\text{$m(u,v)$ letters}}{\underbrace{vuv\cdots}}, \ (u,v) \in E(\Gamma) \rangle.$$
The Artin group $A$ is of \emph{FC type} if, for every complete subgraph $\Lambda \subset \Gamma$, the Coxeter group
$$\langle V(\Lambda) \mid w^2=1, \underset{\text{$m(u,v)$ letters}}{\underbrace{uvu \cdots}} = \underset{\text{$m(u,v)$ letters}}{\underbrace{vuv\cdots}}, \ w \in V(\Gamma), (u,v) \in E(\Gamma) \rangle$$
is finite. Such an Artin group acts on the corresponding \emph{Deligne complex}, which turns out to be a CAT(0) cube complex \cite{MR1303028}. Theorem \ref{thm:CM1} is applied to this complex in \cite{ChatterjiMartin}, proving that Artin groups of FC types whose underlying Coxeter graphs have diameter at least three are acylindrically hyperbolic. Very recently, this result has been generalised in a wider context by \cite{CliqueCubeComplex}. 
\end{app}

\noindent
Another generalisation of Theorem \ref{thm:MO} was proved in \cite{ChatterjiMartin}. 

\begin{thm}\label{thm:CM2}
Let $G$ be a group acting essentially and non-elementarily on an irreducible finite-dimensional cocompact CAT(0) cube complex with no free face. If there exist two points whose stabilisers intersect along a finite subgroup, then $G$ contains a WPD element which skewers a pair of \"uber-separated hyperplanes. A fortiori, $G$ is acylindrically hyperbolic. 
\end{thm}

\begin{app}
According to \cite{TameSL}, the group of $\mathrm{tame}(\mathrm{SL}_2(\mathbb{C}))$, a subgroup of the 3-dimensional Cremona group $\mathrm{Bir}(\mathbb{P}^3(\mathbb{C}))$, acts cocompactly, essentially and non-elementarily on a hyperbolic CAT(0) cube complex without free faces. In \cite{TameSLacylhyp}, it is proved that Theorem \ref{thm:CM2} applies, so that $\mathrm{tame}(\mathrm{SL}_2(\mathbb{C}))$ turns out to be acylindrically hyperbolic. 
\end{app}

\noindent
So far, we have met weakly acylindrical actions and non-uniformly weakly acylindrical actions as relevant types of actions on CAT(0) cube complexes. Theorem \ref{thm:acylcriterion} also suggests the following definition.

\begin{definition}
Let $G$ be a group acting on some CAT(0) cube complex $X$. The action $G \curvearrowright X$ is \emph{acylindrical action on the hyperplanes} if there exists constants $R,N \geq 0$ such that, for every hyperplanes $J_1$ and $J_2$ separated by at least $R$ other hyperplanes, the intersection $\mathrm{stab}(J_1) \cap \mathrm{stab}(J_2)$ has cardinality at most $N$. 
\end{definition}

\noindent
These actions were introduced and studied independently in \cite{CubicalAccessibility} and \cite{MoiAcylHyp}. In the second reference, the following criterion is proved. 

\begin{thm}
Let $G$ be a group acting essentially on a finite-dimensional CAT(0) cube complex. If the action is acylindrical on the hyperplanes, then $G$ contains a WPD contracting isometry. A fortiori, $G$ is either virtually cyclic or acylindrically hyperbolic. 
\end{thm}

\subsection{Hyperbolically embedded subgroups}\label{section:hypembed}

\noindent
So far, we have essentially deduced the acylindrical hyperbolicity of groups acting on CAT(0) cube complexes from the existence of particular isometries. Otherwise saying, we have considered only cyclic subgroups. However, in \cite{DGO}, acylindrical hyperbolicity is studied from non-necessarily cyclic subgroups called \emph{hyperbolically embedded subgroups}. This section is dedicated to these subgroups. It is worth noticing that hyperbolically embedded subgroups satisfy interesting properties. For instance, they are Lipschitz quasi-retracts of the whole groups \cite[Theorem 4.31]{DGO}, so that the geometries of these subgroups are linked to the geometry of the whole group (see \cite[Corollary 4.32]{DGO}). Consequently, characterising these subgroups in order to recognize them more easily is an interesting general problem. Our main criterion is the following:

\begin{thm}\label{thm:hypembed}
Let $G$ be a group acting geometrically on some CAT(0) cube complex and $\mathcal{H}$ a finite collection of subgroups of $G$. Then $\mathcal{H}$ is hyperbolically embedded if and only if it is an almost malnormal collection of Morse subgroups. 
\end{thm}

\noindent
Notice that we do not know if a (finitely generated) malnormal subgroup is automatically a Morse subgroup; see Question \ref{question:malnormalMorse}.

\begin{proof}[Proof of Theorem \ref{thm:hypembed}.]
Suppose that $\mathcal{H}$ is an almost malnormal collection of Morse subgroups. According to Corollary \ref{cor:MorseCriterion}, each subgroup $H \in \mathcal{H}$ acts geometrically on a contracting convex subcomplex $Y(H) \subset X$; moreover, we may suppose that $Y(H)$ is a neighborhood of the orbit $H \cdot x_0$ where $x_0 \in X$ is a base vertex we fix. Let $\mathcal{Z}$ denote the collection of the translates of all the $Y(H)$'s. 

\begin{claim}\label{claim:projbounded}
There exists a constant $C_1 \geq 0$ such that, for every distinct $Z_1,Z_2 \in \mathcal{Z}$, the projection of $Z_2$ onto $Z_1$ has diameter at most $C_1$.
\end{claim}

\noindent
Our claim follows directly from Lemma \ref{lem:fellowtravalmalnormal} below and Lemma \ref{lem:fellowtravel}. 

\medskip \noindent
From now, we denote by $d_C(A,B)$ the diameter of the union of the projections of $A$ and $B$ onto $C$. 

\begin{claim}
There exists a constant $C_2 \geq 0$ such that, for every pairwise distinct elements $A,B,C \in \mathcal{Z}$, at most one of $d_A(B,C)$, $d_B(A,C)$, $d_C(A,B)$ is greater than $C_2$.
\end{claim}

\noindent
Let $K$ denote the constant given by Point $(ii)$ in Proposition \ref{prop:contracting} applied to $Y$ (or equivalently, to any element of $\mathcal{Z}$). Suppose that $d_A(B,C) > 2C_1+K$. Let $x \in \mathrm{proj}_A(B)$ and $y \in \mathrm{proj}_A(C)$ be two vertices minimising the distance between $\mathrm{proj}_A(B)$ and $\mathrm{proj}_A(C)$. Notice that 
$$d(x,y) \geq \mathrm{diam} \left( \mathrm{proj}_A(B) \cup \mathrm{proj}_A(C) \right) - \mathrm{diam} \left( \mathrm{proj}_A(B) \right) - \mathrm{diam} \left( \mathrm{proj}_A(C) \right) >K.$$
Let $J$ be a hyperplane separating $x$ and $y$. According to Lemma \ref{lem:twomin}, $J$ is disjoint from $\mathrm{proj}_A(B)$ and $\mathrm{proj}_A(C)$, so that, according to Proposition \ref{prop:proj}, $J$ must be disjoint from $B$ and $C$. As a consequence of Lemma \ref{lem:vertextoproj}, we know that $J$ cannot separate $B$ and $\mathrm{proj}_A(B)$ since $J$ intersects $A$; similarly, $J$ cannot separate $C$ and $\mathrm{proj}_A(C)$. Therefore, $J$ separates $B$ and $C$. Thus, we have proved that the $\mathcal{H}(x \mid y)$ of the hyperplanes separating $x$ and $y$ is included into the set $\mathcal{H}_A(B \mid C)$ of the hyperplanes intersecting $A$ and separating $B$ and $C$. A fortiori, $\# \mathcal{H}_A(B \mid C)>K$.

\medskip \noindent
Similarly, if $d_B(A,C) \geq 2C_1+M_1$ for some $M_1 \geq 0$, then $\# \mathcal{H}_B(A \mid C ) \geq M_1$. Notice however that $(\mathcal{H}_A(B \mid C), \mathcal{H}_B(A \mid C))$ defines a join of hyperplanes satisfying $\mathcal{H}_A(B \mid C) \subset \mathcal{H}(A)$ and $\mathcal{H}_B(A \mid C) \cap \mathcal{H}(A)= \emptyset$. Therefore, since $\# \mathcal{H}_A(B \mid C)>K$, necessarily 
$$M_1 \leq \# \mathcal{H}_B(A \mid C) \leq K.$$
A fortiori, $d_B(A,C) \leq 2C_1+K$. Similarly, one shows that $d_C(A,B) \leq 2C_1+K$. Consequently, $C_2=2C_1+K$ is the constant we are looking for. 

\begin{claim}
For any distinct $A,B \in \mathcal{Z}$, the set $\{ C \in \mathcal{Z} \mid d_C(A,B)> 3C_1 \}$ is finite. 
\end{claim}

\noindent
Let $C_1, \ldots, C_r \in \mathcal{Z}$ be a collection of subcomplexes satisfying $d_{C_i}(A,B) >3C_1= 2C_1+C_1$; recall from the proof of the previous claim that this implies that $\# \mathcal{H}_{C_i}(A \mid B) >C_1$. As a consequence, we deduce from Claim \ref{claim:projbounded} that, for every distinct $1 \leq i,j \leq r$, necessarily $\mathcal{H}_{C_i}(A \mid B) \neq \mathcal{H}_{C_j}(A \mid B)$ since otherwise the projection of $C_i$ onto $C_j$ would have diameter greater than $C_1$. Consequently, $r \leq 2^{\# \mathcal{H}(A \mid B)}<+ \infty$. This concludes the proof of our third and last claim. 

\medskip \noindent
Our three previous claims allow us to apply \cite[Theorem A]{BBF}. Thus, we get a geodesic metric space $\mathcal{C}(\mathcal{Z})$ on which $G$ acts equipped with an equivariant embedding $\mathcal{Z} \hookrightarrow \mathcal{C}(\mathcal{Z})$ which is isometric on each $Z \in \mathcal{Z}$. As a consequence, each $H \in \mathcal{H}$ acts properly on $\mathcal{C}(\mathcal{Z})$ and each $Z \in \mathcal{Z}$ is contained into a neighborhood of (the image of) the orbit of our basepoint $x_0$ under the coset of some subgroup of $\mathcal{H}$. Moreover, according \cite[Theorem 6.4]{MetricRelativeHyp}, the space $\mathcal{C}(\mathcal{Z})$ is hyperbolic relative to (the image of) $\mathcal{Z}$, and a fortiori relative to the orbits of (the image of) $x_0$ under the cosets of the subgroups of $\mathcal{H}$. Now, it follows from Sisto's criterion \cite[Theorem 6.4]{MetricRelativeHyp} that $\mathcal{H}$ is a hyperbolically embedded collection of subgroups. 

\medskip \noindent
Conversely, a hyperbolically embedded collection of subgroups is always an almost malnormal collections of Morse subgroups according to \cite[Proposition 4.33]{DGO} and \cite[Theorem 2]{Sistohypembed}. 
\end{proof}

\begin{lemma}\label{lem:fellowtravalmalnormal}
Let $G$ be a group with a uniform bound on the size of its finite subgroups and $H$ an almost malnormal subgroup. Suppose that $G$ acts metrically properly on some geodesic metric space $X$, and that there exists a subspace $Y \subset X$ on which $H$ acts geometrically. For every $L \geq 0$, there exists a constant $A \geq 0$ such that $Y^{+L} \cap gY^{+L}$ has diameter at most $A$ for every $g \in G$.
\end{lemma}

\begin{proof}
Fix a basepoint $x \in Y$. Because $H$ acts geometrically on $Y$, there exists a constant $C \geq 0$ such that $Y$ is covered by $H$-translates of the ball $B(x,C)$. Suppose that the diameter of $Y^{+L} \cap gY^{+L}$ is at least $n(2C+1)$ for some $n \geq 1$. As a consequence, there exist $a_1, \ldots , a_n \in Y^{+L} \cap gY^{+L}$ such that $d(a_i,a_j) \geq 2C+1$ for every distinct $1 \leq i,j \leq n$. For every $1 \leq i \leq n$, fix $b_i \in Y$ and $c_i \in gY$ such that $d(a_i,b_i) \leq L$ and $d(a_i,c_i) \leq L$. For every $1 \leq i \leq n$, there exist $h_i \in H$ and $h_k \in H^g$ such that $d(b_i,h_ix) \leq C$ and $d(c_i,k_ix) \leq C$. Notice that, for every $1 \leq i \leq n$, one has
$$d(h_ix,k_ix) \leq d(h_ix,b_i)+d(b_i,a_i)+d(a_i,c_i)+d(c_i,k_ix) \leq 2(L+C),$$
or equivalently, $d(x,h_i^{-1}k_i x) \leq 2(L+C)$. Now, because $G$ acts metrically properly on $X$, there exists some $N \geq 0$ such that at most $N$ elements of $G$ may satisfy this inequality. Consequently, if $n > N \cdot \#( H \cap H^g)$, then $\{ h_i^{-1}k_i \mid 1 \leq i \leq n\}$ must contain more than $\# (H \cap H^g)$ pairwise equal elements, say $h_1^{-1}k_1, \ldots, h_s^{-1}k_s$; equivalently, $h_1h_i^{-1}k_ik_1^{-1}=1$ for every $1 \leq i \leq s$. For convenience, set $p_i=h_1h_i^{-1}=k_1k_i^{-1}$ for every $1 \leq i \leq s$; notice that $p_i \in H \cap H^g$. Next, for every distinct $1 \leq i,j \leq s$, one has
$$\begin{array}{lcl} d(p_ia_i, p_ja_i) & \geq & d(a_i,a_j) -d(p_ia_i, p_ja_j) \geq d(a_i,a_j)-d(h_i^{-1}a_i,x)- d(x,h_j^{-1}a_j) \\ \\ & \geq & 2C+1-C-C=1 \end{array}$$
A fortiori, $p_i \neq p_j$. Thus, we have constructed more than $\# (H \cap H^g)$ pairwise distinct elements in $H \cap H^g$, which is of course impossible. Therefore, $n \leq  N \cdot \#( H \cap H^g)$. We conclude that $A=NF (2C+1)$ is the constant we are looking for, where $F$ denotes the maximal cardinality of a finite subgroup of $G$. 
\end{proof}

\begin{app}
Any hyperbolically embedded subgroup of a freely irreducible right-angled Artin group must be either a finite-index subgroup or a free subgroup. This statement is direct consequence of Theorems \ref{thm:hypembed} and \ref{thm:MorseinRAAG}.
\end{app}

\noindent
It is interesting to notice that, as a consequence of \cite[Theorem 1.4]{OsinAcyl}, Corollary \ref{cor:MorseCriterion} and Proposition \ref{prop:loxoinCAT0} below, a cyclic subgroup $H$ of some group $G$ acting geometrically on a CAT(0) cube complex is Morse if and only if the subgroup 
$$E(H) = \{ g  \in G \mid \# (H \cap H^g) =+ \infty \}$$
is hyperbolically embedded. Loosely speaking, you make your subgroup almost malnormal to get a hyperbolically embedded subgroup. This implies that any cyclic Morse subgroup is a finite-index subgroup of some hyperbolically embedded subgroup. However, such a phenomenon does not occur in full generality for other kinds of subgroups, even in elementary situations. For instance, consider the free group $G = \langle a,b \mid \ \rangle$ and its subgroup $H = \langle a,bab^{-1} \rangle$. Let $K$ be an arbitrary malnormal subgroup of $G$ containing a subgroup commensurable to $H$. Notice that there exists some integer $n \geq 1$ such that $a^n$ and $ba^nb^{-1}$ both belong to $K$. Since the intersections $K \cap aKa^{-1}$ and $K \cap bKb^{-1}$ are infinite, necessarily $a$ and $b$ both belong to $K$, hence $K=G$. Consequently, no malnormal subgroup of $G$, and a fortiori no hyperbolically embedded subgroup of $G$, is commensurable to $H$.

\medskip \noindent
Nevertheless, we able to prove:

\begin{thm}\label{thm:MorseAcylHyp}
Let $G$ be a group acting geometrically on some CAT(0) cube complex. The following statements are equivalent:
\begin{itemize}
	\item $G$ is acylindrically hyperbolic;
	\item $G$ contains an infinite stable subgroup of infinite index;
	\item $G$ contains an infinite Morse subgroup of infinite index.
\end{itemize}
\end{thm}

\noindent
Recall from \cite{StableSubgroups} that a subgroup $H$ in a finitely generated group $G$ is \emph{stable} if, for any constants $A \geq 1$ and $B \geq 0$, there exists a constant $K \geq 0$ such that the Hausdorff distance between any two $(A,B)$-quasi-geodesics linking two points of $H$ is at most $K$. Equivalently, stable subgroups are hyperbolic Morse subgroups. The criterion used to prove the acylindrical hyperbolicity of our group in the previous statement will be the following. We refer to Appendix \ref{section:MorseInRAAG} for the definition of the vocabulary related to the combinatorial boundary. 

\begin{prop}\label{prop:discboundary}
Let $G$ be a group acting geometrically on some CAT(0) cube complex $X$. The following statements are equivalent:
\begin{itemize}
	\item[(i)] $G$ contains a contracting isometry;
	\item[(ii)] the combinatorial boundary $\partial^cX$ contains an isolated point;
	\item[(iii)] the combinatorial boundary $\partial^cX$ is not $\prec$-connected.
\end{itemize}
\end{prop}

\begin{proof}
The implication $(i) \Rightarrow (ii)$ follows from Theorem \ref{thm:contractingboundary}. The implication $(ii) \Rightarrow (iii)$ is clear. Now suppose that $G$ does not contain contracting isometries. It follows from \cite[Theorem 5.46]{article3} (which is an easy consequence of Theorem \ref{thm:CS}) that $X$ contains a $G$-invariant convex subcomplex $Y$ which decomposes as a Cartesian product of two unbounded subcomplexes. Because $G$ acts cocompactly on both $X$ and $Y$, necessarily $X$ is neighborhood of $Y$, so that $\partial^c Y= \partial^cX$. But $\partial^cY$ must be connected as any combinatorial boundary of a product of two unbounded complexes. This proves the implication $(iii) \Rightarrow (i)$. 
\end{proof}

\begin{proof}[Proof of Theorem \ref{thm:MorseAcylHyp}.]
If $G$ is acylindrically hyperbolic, then $G$ contains a Morse element $g \in G$ according to \cite{Sistohypembed}. Thus, $\langle g \rangle$ is a stable subgroup of $G$, which has infinite index since $G$ is not virtually cyclic. Next, it is clear that if $G$ contains an infinite stable subgroup of infinite index then it must contain an infinite Morse subgroup of infinite index since a stable subgroup is a Morse subgroup as well. From now on, suppose that $G$ contains an infinite Morse subgroup $H \leq G$ of infinite index. 

\medskip \noindent
According to Corollary \ref{cor:MorseCriterion}, there exists an $H$-cocompact contractible convex subcomplex $Y \subset X$. As a consequence of \cite[Remark 4.15]{article3}, the combinatorial boundary $\partial^cY$ of $Y$ is \emph{full} in $\partial^cX$, i.e., any element of $\partial^cX$ which is $\prec$-comparable to an element of $\partial^cY$ must belong to $\partial^cY$. Therefore, three cases may happen: $\partial^cY$ may be empty; $\partial^cY$ may coincide with $\partial^cX$; or $\partial^cX$ may not be $\prec$-connected. In the latter case, we deduce from Proposition \ref{prop:discboundary} that $G$ contains a contracting isometry, so that $G$ must be either virtually cyclic or acylindrically hyperbolic according to Corollary \ref{cor:acylhypiffcontracting}; because $G$ contains an infinite Morse subgroup of infinite index, it cannot be virtually cyclic, so we get the desired conclusion. 

\medskip \noindent
Next, notice that $\partial^cY$ cannot be empty since $H$ is infinite. Moreover, since $H$ has infinite index in $G$, necessarily $\partial^c Y \subsetneq \partial^c X$. Indeed, suppose that $\partial^c Y= \partial^c X$. We deduce from Lemma \ref{lem:wholeboundary} below that $H$ acts cocompactly on $X$. Let $Q$ be a finite fundamental domain for $G \curvearrowright X$ and $C$ a finite fundamental domain for $H \curvearrowright X$ which contains $Q$. Because $C$ is finite, there exist $g_1, \ldots, g_m \in G$ such that $C \subset \bigcup\limits_{i=1}^m g_iQ$; and because the action $G \curvearrowright X$ is properly discontinuous, $S= \{ g \in G \mid gQ \cap Q \neq \emptyset \}$ is finite. Fix some vertex $x_0 \in Q$. Now, if $g \in G$, there exists $h \in H$ such that $hg \cdot x_0 \in C$, and then $g_i^{-1}hg \cdot x_0 \in Q$ for some $1 \leq i \leq m$. Therefore, $g \in Hg_iS$. We conclude that $H$ is a finite-index subgroup.

\medskip \noindent
This concludes the proof of our theorem. 
\end{proof}

\begin{lemma}\label{lem:wholeboundary}
Let $X$ be a locally finite CAT(0) cube complex and $Y$ a convex subcomplex. The equality $\partial^c Y = \partial^cX$ implies that $X$ is neighborhood of $Y$.
\end{lemma} 

\begin{proof}
Suppose that $X$ is not contained into a neighborhood of $Y$. So there exists a sequence of vertices $(x_n)$ satisfying $d(x_n,Y) \underset{n \to + \infty}{\longrightarrow} + \infty$. Let $J_1^n,\ldots,J_{k(n)}^n$ denote the hyperplanes separating $x_n$ from its projection onto $Y$; notice that $J_i^n$ separates $x_n$ from $Y$ according to Lemma \ref{lem:vertextoproj}. Fix some base vertex $x \notin Y$; if such a vertex does not exist, then $X=Y$ and there is nothing to prove. For every $n \geq 1$, let $y_n$ be the projection of $x$ onto the halfspace delimited by $J_{k(n)}^n$ which is disjoint from $Y$, and fix some geodesic $[x,y_n]$ between $x$ and $y_n$. Because $X$ is locally finite, our sequence $([x,y_n])$ must have a subsequence converging to some combinatorial ray $\rho$. By construction, $\mathcal{H}(\rho)$ contains infinitely many hyperplanes disjoint from $Y$, so that \cite[Lemma 4.5]{article3} implies $\rho(+ \infty) \notin \partial^c Y$. A fortiori, $\partial^cY \subsetneq \partial^cX$. This proves our lemma.
\end{proof}

\subsection{Quasi-isometry}

\noindent
It is worth noticing that being acylindrically hyperbolic is stable under quasi-isometry among cubulable groups. In fact, this is true more generally for CAT(0) groups. (But it is an open question in full generality \cite[Problem 9.1]{DGO}.) Let us show the following statement:

\begin{thm}\label{thm:CAT0}
Let $G$ be a group which is not virtually cyclic and which acts geometrically on a CAT(0) space $X$. The following assertions are equivalent:
\begin{itemize}
	\item[(i)] $G$ is acylindrically hyperbolic;
	\item[(ii)] $G$ contains a contracting isometry;
	\item[(iii)] the contracting boundary $\partial_cX$ is non-empty;
	\item[(iv)] the divergence of $X$ is superlinear.
\end{itemize}
\end{thm}

\noindent
In particular, notice that the points $(iii)$ and $(iv)$ are invariant under quasi-isometries (see respectively \cite[Theorem 3.10]{MR3339446} and \cite[Proposition 2.1]{divergence}), so that:

\begin{cor}\label{cor:qi}
Among CAT(0) groups, being acylindrically hyperbolic is a quasi-isometric invariant.
\end{cor}

\noindent
In \cite{arXiv:1112.2666}, and in a more general form in \cite{BBF}, it is proved that if a group $G$ acts on a CAT(0) space $X$ and if $G$ contains a contracting isometry, then it is possible to construct a new action of $G$ on a some hyperbolic space $Y$ (in fact, a quasi-tree) such that the previous contracting isometry becomes a loxodromic isometry of $Y$. The general idea is that it is possible to associate an action on some hyperbolic space to any action (on arbitrary metric spaces) containing isometries ``which behave like isometries of hyperbolic spaces''. In particular, this allows Sisto to prove a strong version of the implication $(ii) \Rightarrow (i)$ of our theorem:

\begin{thm}\emph{\cite{arXiv:1112.2666}}
Let $G$ be a group which is not virtually cyclic and which acts properly discontinuously on a CAT(0) space. If $G$ contains a contracting isometry, then it is acylindrically hyperbolic.
\end{thm}

\noindent
In another article \cite{Sistohypembed}, Sisto proves a kind of reciprocal, in the sense that, for any geometric action of an acylindrically hyperbolic group on an arbitrary metric space, our group must contain an isometry which ``which behave like isometries of hyperbolic spaces'', but with a different meaning:

\begin{thm}\label{thm4}\emph{\cite[Theorem 1]{Sistohypembed}}
Any acylindrically hyperbolic group contains a Morse element.
\end{thm}

\noindent
Given a CAT(0) space $X$ and some of its isometry $g \in \mathrm{Isom}(X)$, we say that $g$ is a \emph{Morse isometry} if $g$ is a loxodromic isometry, with some axis $\gamma$, such that for any $k,L \geq 1$, there exists a constant $C=C(k,L)$ so that any $(k,L)$-quasigeodesic between two points of $\gamma$ stays into the $C$-neighborhood of $\gamma$; the definition does not depend on the choice of the axis. Thus, if an acylindrically hyperbolic group acts geometrically on a CAT(0) space, then it must contain a Morse isometry.

\medskip \noindent
In general, a Morse isometry is not necessarily contracting, but the two notions turn out to coincide in CAT(0) spaces:

\begin{thm}\label{thm5}\emph{\cite[Theorem 2.14]{MR3339446}}
An isometry of a CAT(0) space is contracting if and only if it is a Morse isometry.
\end{thm}

\noindent
By combining Theorem \ref{thm4} with Theorem \ref{thm5}, we deduce the implication $(i) \Rightarrow (ii)$ of our theorem, i.e., an acylindrically hyperbolic group acting geometrically on a CAT(0) space must contain a contracting isometry. This proves that, in the context of CAT(0) spaces, contracting isometries are fundamentally linked to acylindrical hyperbolicity.

\medskip \noindent
Thus, we get a dynamic characterisation of acylindrical hyperbolicity. In order to find a geometric characterisation, we need Charney and Sultan's \emph{contracting boundary} \cite{MR3339446}.

\begin{definition}
Let $X$ be a CAT(0) space. Its \emph{contracting boundary}, denoted $\partial_cX$, is the set of the contracting geodesic rays starting from a fixed basepoint up to finite Hausdorff distance. The definition does not depend on the choice of the basepoint.
\end{definition}

\noindent
It is clear that, if our group $G$ contains a contracting isometry, then our CAT(0) space $X$ have a non-empty contracting boundary, since it will contain any subray of an axis of this isometry. This proves the implication $(ii) \Rightarrow (iii)$ of our theorem. Conversely, as noticed in \cite[Corollary 2.14]{arXiv:1509.09314}, it follows from a result of Bullmann and Buyalo \cite{BallmannBuyalo} that $G$ necessarily contains a contracting isometry if $X$ contains a contracting ray, so that the acylindrical hyperbolicity of $G$ follows from Theorem \ref{thm4}. This proves the implication $(iii) \Rightarrow (ii)$ of our theorem.

\medskip \noindent
Finally, the equivalent $(iii) \Leftrightarrow (iv)$ was proved in \cite[Theorem 2.14]{MR3339446}. This concludes the proof of our theorem.

\medskip \noindent
We conclude this section with a last statement, which will be useful in the next section (in the context of CAT(0) cube complexes).

\begin{prop}\label{prop:loxoinCAT0}
Let $G$ be an acylindrically hyperbolic group acting geometrically on a CAT(0) space $X$. Then $g \in G$ is a generalised loxodromic element if and only if it is a contracting isometry of $X$.
\end{prop}

\noindent
Recall from \cite{OsinAcyl} that, given a group $G$, an element $g \in G$ is a \emph{generalised loxodromic element} if $G$ acts acylindrically on a hyperbolic space such that $g$ turns out to be a loxodromic isometry.

\begin{proof}
Let $g \in G$ be a generalised loxodromic element. According to \cite{Sistohypembed}, $g$ is a Morse element, so that $g$ must be a Morse isometry of $X$, and finally a contracting isometry according to Theorem \ref{thm5}. Conversely, supposed that $g \in G$ is a contracting isometry of $X$. Then \cite{arXiv:1112.2666} implies that $g$ is contained in a virtually cyclic subgroup which is hyperbolically embedded, so that $g$ must be a generalised loxodromic element according to \cite[Theorem 1.4]{OsinAcyl}.
\end{proof}

\subsection{Acylindrical models}\label{section:curvegraph}

\noindent
Given a group $G$, one of its elements $g \in G$ is a \emph{generalised loxodromic element} if $G$ acts acylindrically on some hyperbolic space so that $g$ induces a loxodromic isometry; see \cite{OsinAcyl} for equivalent characterisations. Loosely speaking, theses elements are those which have a ``hyperbolic behavior''. A \emph{universal action} is an action of $G$ on a hyperbolic space so that all its generalised loxodromic elements induce WPD isometries; and a \emph{universal acylindrical action} is an acylindrical action of $G$ on a hyperbolic space so that all its generalised loxodromic elements induce loxodromic isometries. For instance, the action of the mapping class group of a (non-exceptional) surface on its associated curve graph is a universal acylindrical action. This is the typical example, so that the hyperbolic graphs constructing in attempts to make some classes of groups act systematically on hyperbolic spaces are often referred to as curve graphs; see for instance \cite{GarsideCurveGraph, KimKoberdaGeometryCurveGraph, HHSI}. It was proved in \cite{AbbottDunwoody} that Dunwoody's inaccessible group does not admit a universal acylindrical action, but the existence or non-existence of such actions for finitely presented groups remains open. 

\medskip \noindent
A first naive attempt to define the curve graph of a CAT(0) cube complex $X$, inspired from curve graphs of surfaces, would be to consider the graph whose vertices are the hyperplanes of $X$ and whose edges link transverse hyperplanes. This is the \emph{crossing graph} $\Delta X$ of $X$. However, this graph may not be connected, and even worse, it was noticed in \cite{Roller, MR3217625} that every graph is the crossing graph of a CAT(0) cube complex; in particular, the crossing graph of a CAT(0) cube complex may not be hyperbolic. (Nevertheless, the crossing graph may be interesting, see Appendix \ref{section:crossing}.) Instead, Hagen introduced in \cite{MR3217625} the \emph{contact graph} $\Gamma X$ of $X$ as the graph whose vertices are the hyperplanes of $X$ and whose edges link two hyperplanes whose carriers intersect. 

\begin{thm}\label{thm:contact}
Let $G$ be a group acting geometrically on a CAT(0) cube complex $X$. Then $\Gamma X$ is a quasi-tree on which $G$ acts non-uniformly acylindrically, and, for every $g \in G$, either a power of $g$ stabilises a hyperplane of $X$ (and a fortiori fixes a vertex of $\Gamma X$) or $g$ is a contracting isometry and induces a loxodromic isometry on $\Gamma X$. 
\end{thm}

\begin{proof}
The fact that the contact graph is quasi-isometric to a tree is proved by \cite[Theorem 3.1.1]{Hagenthesis}. (Interestingly, the constants occurring in the quasi-isometry do not depend on the CAT(0) cube complex we consider.) The acylindricity of the action was proved in \cite{MoiAcylHyp}, and the third statement of the theorem comes from \cite[Corollary 6.3.1]{Hagenthesis}.
\end{proof}

\noindent
It remains unknown whether the action on the contact graph is always acylindrical. See \cite{HagenSusse} for more details.

\medskip \noindent
Notice that the contact graph does not provide a universal action for cubulable groups, since contracting isometries may stabilise hyperplanes. (This may happen for instance in right-angled Coxeter groups, even if the action is essential.) In fact, although the existence of a universal acylindrical action has been proved in some cases \cite{UniversalAcylHHS}, it remains an open question in full generality.

\begin{question}\label{question:universal}
If a group acts geometrically on a CAT(0) cube complex, does it admit a universal acylindrical action?
\end{question}

\noindent
In this section, we explain how to construct hyperbolic models of CAT(0) cube complexes. Question \ref{question:universal} is one of the motivations, but several applications will be given at the end of the section. 

\begin{definition}
Let $X$ be a CAT(0) cube complex and $L \geq 0$ an integer. Define the metric $\delta_L$ on (the vertices of) $X$ as the maximal number of pairwise $L$-well-separated hyperplanes separating two given vertices. 
\end{definition}

\noindent
It is worth noticing that one essentially recovers the contact graph when $L=0$.

\begin{fact}
Let $X$ be a CAT(0) cube complex. A map sending every vertex of $X$ to a hyperplane whose carrier contains it induces a quasi-isometry $(X, \delta_0) \to \Gamma X$. 
\end{fact}

\begin{proof}
Let $x,y \in X$ be two vertices and $J,H$ be two hyperplanes of $X$ such that $x \in N(J)$ and $y \in N(H)$. Let $S(J,H)$ denote the maximal number of pairwise strongly separated hyperplanes separating $J$ and $H$. Because any hyperplane separating $J$ and $H$ separates necessarily $x$ and $y$, one has $S(J,H) \leq \delta_0(x,y)$. Next, let $V_1, \ldots, V_r$ be a collection of pairwise strongly separated hyperplanes separating $x$ and $y$; without loss of generality, suppose that $V_i$ separates $V_{i-1}$ and $V_{i+1}$ for every $1 \leq i \leq r$ and that $V_1$ separates $x$ from $V_2, \ldots, V_r$. Notice that, because $x$ does not belong to $N(V_2)$ and that $V_2$ separates $x$ from $V_3$, if $J$ is transverse to $V_3$ then necessarily it must also be transverse to $V_2$, which is impossible since $V_2$ and $V_3$ are strongly separated. Consequently, $J$ and $V_3$ are disjoint. Similarly, one shows that $H$ and $V_{r-2}$ are disjoint. Therefore, $V_3, \ldots, V_{r-2}$ is a collection of pairwise strongly separated hyperplanes separating $J$ and $H$. This proves that $\delta_0(x,y) \leq S(J,H) +4$. 

\medskip \noindent
Thus, we have proved that our map $(X,\delta_0) \to \Gamma X$ is quasi-isometric when the contact graph $\Gamma X$ is endowed with $S( \cdot, \cdot)$. The conclusion follows since we know from \cite[Proposition 23]{MoiAcylHyp} that $S(\cdot, \cdot)$ is coarsely equivalent to $d_{\Gamma X}$. 
\end{proof}

\noindent
In the opposite direction, if one allows $L=+ \infty$ (which is not the case in the sequel), then one recovers the $\ell^{\infty}$-metric, since this distance turns out to be equal to the number of pairwise disjoint hyperplanes separating two given vertices \cite[Corollary 2.5]{depth}. Our next observation is that, if $X$ is hyperbolic, then $(X,\delta_L)$ turns out to be quasi-isometric to $X$ whenever $L$ is sufficiently large. This motivates the idea that $(X,\delta_L)$, for a sufficiently large $L$, captures all the hyperbolic properties of $X$.  

\begin{lemma}\label{lem:HXwhenXhyp}
Let $X$ be a hyperbolic CAT(0) cube complex. Fix a constant $L_0 \geq 0$ such that the joins of hyperplanes of $X$ are all $L_0$-thin. For every $L \geq L_0$, the canonical map $X \to (X,\delta_L)$ is a quasi-isometry.
\end{lemma}

\begin{proof}
Because $X$ is necessarily finite-dimensional, we may consider without loss of generality the $\ell^{\infty}$-metric $d_{\infty}$ on $X$. Let $x,y \in X$ be two vertices. Since any collection of pairwise $L$-well-separated hyperplanes separating $x$ and $y$ provides a collection of pairwise disjoint hyperplanes separating $x$ and $y$, necessarily $\delta_L(x,y) \leq d_{\infty}(x,y)$. Now, let $J_1, \ldots, J_r$ be a maximal collection of pairwise disjoint hyperplanes separating $x$ and $y$. So $r=d_{\infty}(x,y)$. Fix some $1 \leq i \leq r-L_0-1$ and let $\mathcal{K}$ be a collection of hyperplanes transverse to both $J_i$ and $J_{i+L_0+1}$ which does not contain any facing triple. By noticing that $\mathcal{K}$ and $\{J_i,J_{i+1}, \ldots, J_{i+L_0+1} \}$ define a join of hyperplanes, it follows that $\# \mathcal{K} \leq L_0$. A fortiori, $J_i$ and $J_{i+L_0+1}$ are $L$-well-separated. Therefore, 
$$\delta_L(x,y) \geq \frac{1}{L_0+1} \cdot d_{\infty}(x,y) - L_0-1.$$
This conclude the proof of our lemma. 
\end{proof}

\noindent
The main result of this section is the following:

\begin{thm}\label{thm:mainHX}
Let $G$ be a group acting geometrically on a CAT(0) cube complex $X$. Then:
\begin{itemize}
	\item for every $L \geq 0$, $(X,\delta_L)$ is $9(L+2)$-hyperbolic;
	\item for every $L \geq 0$, the action $G \curvearrowright (X,\delta_L)$ is non-uniformly acylindrical;
	\item an isometry $g \in \mathrm{Isom}(X)$ defines a contracting isometry of $X$ if and only if it induces a loxodromic isometry on $(X,\delta_L)$ when $L$ is sufficiently large; otherwise, $g$ induces an elliptic isometry of $(X,\delta_L)$ for every $L \geq 0$.
\end{itemize}
\end{thm}

\noindent
We emphasize that our metric space $(X,\delta_L)$ is not geodesic (although it follows from Lemma \ref{lem:Interval} below that it is quasi-geodesic), so the definition of hyperbolic metric spaces which have to use is the following: a metric space $(S,d)$ is $\delta$-hyperbolic if, for every four points $p,q,r,s \in S$, the following inequality holds:
$$d(p,r)+d(q,s) \leq \max ( d(p,q)+d(r,s), d(p,s)+d(q,r)) + 2 \delta.$$
We refer to \cite{GhysdelaHarpe} for more information on equivalent definitions of Gromov hyperbolicity.

\begin{prop}\label{prop:HXhyp}
Let $X$ be a CAT(0) cube complex and $L \geq 0$ an integer. Then $(X,\delta_L)$ is $9(L+2)$-hyperbolic.
\end{prop}

\noindent
Before proving Proposition \ref{prop:HXhyp}, we begin by noticing that combinatorial geodesics are unparametrised quasi-geodesics with respect to our new metrics.

\begin{lemma}\label{lem:Interval}
Let $X$ be a CAT(0) cube complex, $x,y \in X$ two vertices and $L \geq 0$ an integer. The inequalities
$$\delta_L(x,z)+ \delta_L(z,y)-2(L+3) \leq \delta_L(x,y) \leq \delta_L(x,z)+\delta_L(z,y)$$
holds for every $z \in I(x,y)$. 
\end{lemma}

\begin{proof}
Let $\mathcal{H}$ (resp. $\mathcal{V}$) be a maximal collection of pairwise $L$-well-separated hyperplanes separating $x$ and $z$ (resp. $z$ and $y$). Write $\mathcal{H}$ as $\{H_1, \ldots, H_r \}$ so that $H_i$ separates $H_{i-1}$ and $H_{i+1}$ for every $2 \leq i \leq r-1$ and $H_1$ separates $z$ from $H_2, \ldots, H_r$; and similarly $\mathcal{V}$ as $\{ V_1, \ldots, V_s \}$ so that $V_i$ separates $V_{i-1}$ and $V_{i+1}$ for every $2 \leq i \leq s-1$ and $V_1$ separates $z$ from $V_2, \ldots, V_s$. Notice that $r=\delta_L(x,z)$ and $s= \delta_L(z,y)$. Since $\delta_L(x,y) \geq \delta_L(x,z)$ and $\delta_L(x,y) \geq \delta_L(z,y)$, there is nothing to prove if $r \leq 2(L+3)$ or $s \leq 2(L+3)$, so we suppose that $r,s \geq 2(L+3)$. 

\medskip \noindent
Observe that, if there exist some $1 \leq i \leq r$ and some $1 \leq j \leq s$ such that $V_i$ and $H_j$ are transverse, then $V_p$ and $H_q$ must be transverse for every $1 \leq p \leq i$ and $j \leq q \leq r$. Because $H_1$ and $H_2$ are $L$-well-separated, necessarily $V_1, \ldots, V_{L+1}$ cannot be all transverse to both $H_1$ and $H_2$, so we deduce from our previous observation that $H_2$ and $V_{L+1}$ must be disjoint. Similarly, one shows that $V_2$ and $H_{L+1}$ are disjoint. Consequently, the hyperplanes
$$H_{L+2}, \ldots, H_r,V_{L+2}, \ldots, V_s$$ 
are pairwise disjoint. If $H_i$ and $V_j$ are not $L$-well-separated for some $i,j \geq L+3$, then there exists a collection $\mathcal{K}$ of at least $L+1$ hyperplanes transverse to both $H_i$ and $V_j$ which does not contain any facing triple. But then the hyperplanes of $\mathcal{K}$ must be all transverse to both $H_{L+2}$ and $H_{L+3}$, which are $L$-well-separated. Observe that we have proved the following statement:

\begin{fact}\label{fact:subcollectionwellseparated}
Let $x,y \in X$ and $z \in I(x,y)$ be three vertices, and $\mathcal{H}$ (resp. $\mathcal{V}$) a collection of pairwise $L$-well-separated hyperplanes separating $x$ and $z$ (resp. $z$ and $y$). There exist subcollections $\mathcal{H}' \subset \mathcal{H}$ and $\mathcal{V}' \subset \mathcal{V}$ satisfying $\# \mathcal{H}' = \# \mathcal{H}- L-3$ and $\# \mathcal{V}' = \# \mathcal{V}-L-3$ such that the hyperplanes of $\mathcal{H}' \cup \mathcal{V}'$ are pairwise $L$-well-separated. 
\end{fact}

\noindent
Consequently, the hyperplanes
$$H_{L+3}, \ldots, H_r, V_{L+3},\ldots, V_s$$
are pairwise $L$-well-separated. The inequality
$$\delta_L(x,y) \geq r+s-2(L+3) = \delta_L(x,z)+\delta_L(z,y)-2(L+3)$$
follows. The second inequality in our lemma is obtained from the triangle inequality. 
\end{proof}

\begin{proof}[Proof of Proposition \ref{prop:HXhyp}.]
Our goal is to prove that, for any four vertices $x_1,x_2,x_3,x_4 \in X$, the inequality
$$\delta_L(x_1,x_3)+\delta_L(x_2,x_4) \leq \max ( \delta_L(x_1,x_2)+\delta_L(x_3,x_4), \delta_L(x_1,x_4)+\delta_L(x_2,x_3)) + 18(L+2)$$
holds. Let $m_1,m_2,m_3,m_4$ be the vertices provided by Lemma \ref{lem:quadrangle}. For convenience, we set $m=\delta_L(m_1,m_2)=\delta_L(m_3,m_4)$ and $n=\delta_L(m_1,m_4)=\delta_L(m_2,m_3)$. One has
$$\begin{array}{lcl} \delta_L(x_1,x_3)+\delta_L(x_2,x_4)  & \leq & \delta_L(x_1,m_1)+\delta_L(m_3,x_3)+ \delta_L(x_2,m_2)+ \delta_L(m_4,x_4) \\ \\ & &+ 2(m+n) \\ \\ & \leq & \left( \delta_L(x_1,m_1)+m+\delta_L(m_2,x_2) \right) \\ \\ & & + \left( \delta_L(x_3,m_3)+m+ \delta_L(x_4,m_4) \right) +2n \\ \\ & \leq & \delta_L(x_1,x_2)+ \delta_L(x_3,x_4) +8(L+3)+2n \end{array}$$
One shows similarly that
$$ \delta_L(x_1,x_3)+\delta_L(x_2,x_4) \leq \delta_L(x_1,x_4)+\delta_L(x_2,x_3)+8(L+3)+ 2m.$$
Suppose without loss of generality that $\delta_L(x_1,x_4)+\delta_L(x_2,x_3) \leq \delta_L(x_1,x_2)+ \delta_L(x_3,x_4)$. Since
$$\delta_L(x_1,x_4)+\delta_L(x_2,x_3) \geq \sum\limits_{i=1}^4 \delta_L(x_i,m_i) +2n- 8(L+3)$$
and
$$\delta_L(x_1,x_2)+\delta_L(x_3,x_4) \geq \sum\limits_{i=1}^4 \delta_L(x_i,m_i)+2m-8(L+3),$$
it follows that $n -m \leq 4(L+3)$. Notice that, if $\delta_L(m_1,m_4) = n \geq 2$, necessarily $m \leq L$. Therefore, $n \leq \max( 2, L+4(L+3)) \leq 5L+12$. Finally, we conclude that
$$\begin{array}{lcl}  \delta_L(x_1,x_3)+\delta_L(x_2,x_4) & \leq & \delta_L(x_1,x_2)+\delta_L(x_3,x_4) +8(L+3)+2(5L+12) \\ \\ & \leq &\delta_L(x_1,x_2)+\delta_L(x_3,x_4) + 18(L+2) \end{array}$$
which is the desired inequality. 
\end{proof}

\noindent
Now, we focus on the acylindricity of the action.

\begin{prop}\label{prop:HXacyl}
Let $G$ be a group acting non-uniformly weakly acylindrically on a CAT(0) cube complex $X$, and $L \geq 0$ an integer. The action $G \curvearrowright (X,\delta_L)$ is non-uniformly acylindrical. 
\end{prop}

\noindent
Before proving our proposition, let us consider the following statement:

\begin{lemma}\label{lem:isommid}
Let $X$ be a CAT(0) cube complex, $g \in \mathrm{Isom}(X)$ an isometry, $L \geq 0$ an integer and $x,y \in X$ two vertices. Suppose that $\delta_L(x,gx) \leq \epsilon$ and $\delta_L(y,gy) \leq \epsilon$ for some $\epsilon \geq 0$. Then $\delta_L(z,gz) \leq 3(\epsilon+L+3)$ for every $z \in I(x,y)$. 
\end{lemma}

\begin{proof}
Let $\mathcal{H}$ (resp. $\mathcal{N}$) denote a maximal collection of pairwise $L$-well-separated hyperplanes separating $x$ and $z$ (resp. $z$ and $gz$). Notice that a hyperplane separating $z$ and $gz$ must separate $x$ and $gx$; or $y$ and $gy$; or $\{z,gx\}$ and $\{y,gz\}$; or $\{x,z\}$ and $\{gz,gy\}$. Since $\delta_L(x,gx) \leq \epsilon$ and $\delta_L(y,gy)$, there exists a subcollection $\mathcal{N}' \subset \mathcal{N}$ satisfying $\# \mathcal{N}' \geq \# \mathcal{N}-2 \epsilon$ such that no hyperplane of $\mathcal{N}'$ separates $x$ and $gx$ nor $y$ and $gy$. Because a hyperplane separating $\{z,gx\}$ and $\{y,gz\}$ is transverse to any hyperplane separating $\{x,z\}$ and $\{gz,gy\}$, the hyperplanes of $\mathcal{N}'$ either all separate $\{z,gx\}$ and $\{y,gz\}$, or all separate $\{x,z\}$ and $\{gz,gy\}$. Without loss of generality, say that we are in the former case. If $\# \mathcal{N}' \leq 1$, then $\delta_L(z,gz)=\# \mathcal{N} \leq 2\epsilon + \# \mathcal{N'} \leq 2 \epsilon+1$ and we are done, so we suppose that $\# \mathcal{N}' \geq 2$. 

\medskip \noindent
Next, notice that at most $\epsilon$ hyperplanes of $\mathcal{H}$ separate $x$ and $gx$ since $\delta_L(x,gx) \leq \epsilon$, and at most $L$ hyperplanes of $\mathcal{H}$ separate either $gx$ and $gy$ or $y$ and $gy$, since any such hyperplane must be transverse to all the hyperplanes of $\mathcal{N}'$. Therefore, there exists a subcollection $\mathcal{H}' \subset \mathcal{H}$ satisfying $\# \mathcal{H}' \geq \# \mathcal{H} - \epsilon-L$ such that any hyperplane of $\mathcal{H}'$ separates $gx$ and $gz$. 

\medskip \noindent
By applying Fact \ref{fact:subcollectionwellseparated}, we find subcollections $\mathcal{H}'' \subset \mathcal{H}'$ and $\mathcal{N}'' \subset \mathcal{N}'$ satisfying $\# \mathcal{H}'' \geq \# \mathcal{H}' - L-3$ and $\# \mathcal{N}'' \geq \# \mathcal{N}' - L-3$ such that the hyperplanes of $\mathcal{H}'' \cup \mathcal{N}''$ are pairwise $L$-well-separated. Consequently, we have
$$\begin{array}{lcl} \# \mathcal{H} & = & \delta_L(x,z)= \delta_L(gx,gz) \geq \# \mathcal{H}'' + \# \mathcal{V}'' \\ \\ & \geq & \# \mathcal{H}+ \# \mathcal{N} -2(L+3) - 3 \epsilon -L \end{array}$$
hence
$$\delta_L(z,gz)= \# \mathcal{N} \leq 3 (\epsilon+L+3),$$
which concludes the proof of our lemma. 
\end{proof}

\begin{proof}[Proof of Proposition \ref{prop:HXacyl}.]
Suppose that the action $G \curvearrowright H(X)$ is not non-uniformly acylindrical. So there exists some $\epsilon>0$ such that, for every $R_0 \geq 0$, there exist two vertices $x,y \in X$ satisfying $\delta_L(x,y) > R_0$ such that
$$F = \{ g \in G \mid \delta_L(x,gx) \leq \epsilon, \delta_L(y,gy) \leq \epsilon \}$$
is infinite. Suppose that $R_0 \geq 8L+10 \epsilon +25$, and for convenience write $R = R_0 - 8L+10 \epsilon +25$. 

\medskip \noindent
Fix an element $g \in F$. Let $\mathcal{H}$ be a maximal collection of pairwise $L$-well-separated hyperplanes separating $x$ and $y$. Because $\delta_L(x,gx) \leq \epsilon$ and $\delta_L(y,gy) \leq \epsilon$, there exist at most $2 \epsilon$ hyperplanes of $\mathcal{H}$ separating either $x$ and $gx$ or $y$ and $gy$. Moreover, notice that, if a hyperplane $J$ separating $x$ and $y$ separates $x$ and $gx$, then any hyperplane separating $x$ and $J$ must separate $x$ and $gx$ as well; similarly, if a hyperplane $J$ separating $x$ and $y$ separates $y$ and $gy$, then any hyperplane separating $y$ and $J$ must separate $y$ and $gy$ as well. Consequently, if $\mathcal{H}'$ denotes the collection of hyperplanes obtained from $\mathcal{H}$ by removing the first and last $\epsilon$ hyperplanes (ordering $\mathcal{H}$ by following a geodesic from $x$ to $y$), then the hyperplanes of $\mathcal{H}'$ separates $gx$ and $gy$. Write $\mathcal{H}'$ as $\{H_1, \ldots, H_k\}$ such that $H_i$ separates $H_{i-1}$ and $H_{i+1}$ for every $2 \leq i \leq k-1$ and such that $H_1$ separates $x$ and $H_2, \ldots, H_k$. Because 
$$k= \# \mathcal{H}' = \# \mathcal{H}-2 \epsilon = \delta_L(x,y)-2 \epsilon \geq R_0 - 2 \epsilon = R+8(L+ \epsilon)+25,$$
there exist $r \leq p \leq q \leq s$ such that $|p-q| > R+2(L+1)$ and $|p-r|,|q-s| > 3(\epsilon+L+3)$ and $|r-1|,|k-s| > \epsilon$. 

\medskip \noindent
We claim that, for every hyperplane $J$ separating $H_p$ and $H_q$, the hyperplane $gJ$ intersects the subspace delimited by $H_r$ and $H_s$. Indeed, let $z$ be a vertex of $N(J) \cap I(x,y)$. By applying Lemma \ref{lem:isommid}, we know that $\delta_L(z,gz) \leq 3(\epsilon+L+3)$. Consequently, $gz \in N(gJ)$ cannot be outside the subspace delimited by $H_r$ and $H_s$ since $|p-r|$ and $|q-s|$ are greater than $3(\epsilon+L+3)$.

\medskip \noindent
Next, let $\mathcal{A}_g$ denote the set of all the hyperplanes $J$ separating $H_p$ and $H_q$ such that $gJ$ is transverse to $H_{r-1}$. By noticing, thanks to our previous claim, that $g \mathcal{A}_g$ is a collection of pairwise $L$-well-separated hyperplanes transverse to both $H_{r-1}$ and $H_r$ which does not contain any facing triple, we deduce that $\# \mathcal{A}_g \leq L$. Similarly, if $\mathcal{B}_g$ denotes the set of all the hyperplanes $J$ separating $H_p$ and $H_q$ such that $gJ$ is transverse to $H_{s+1}$, then $\# \mathcal{B}_g \leq L$. Set $\mathcal{H}_g'  = \mathcal{H}(H_p \mid H_q) \backslash \left( \mathcal{A}_g \cup \mathcal{B}_g \right)$, where $\mathcal{H}(H_p \mid H_q)$ denotes the set of all the hyperplanes separating $H_p$ and $H_q$. 

\medskip \noindent
So, if a hyperplane $J$ belongs to $\mathcal{H}_g'$, then $gJ$ is included into the subspace delimited by $H_{r-1}$ and $H_{s+1}$. If $gJ$, $H_{r-1}$ and $H_{s+1}$ define a facing triple, then the halfspace delimited by $gJ$ which is disjoint from $H_{r-1}$ and $H_{s+1}$ must contain either $gx$ and $gy$, which is impossible: in the former case, $x$ and $gx$ would be separated by $H_1, \ldots, H_{r-1}$, contradicting the inequality $\delta_L(x,gx) \leq \epsilon$ since $r> \epsilon+1$; and in the latter case, $y$ and $gy$ would be separated by $H_{s+1}, \ldots, H_k$, contradicting the inequality $\delta_L(y,gy) \leq \epsilon$ since $k-s> \epsilon$. Therefore, $gJ$ separates $H_{r-1}$ and $H_{s+1}$. The conclusion is that $g$ induces a map $\mathcal{H}_g \to \mathcal{H}(H_{r-1} \mid H_{s+1})$ where we set
$$\mathcal{H}_g = \mathcal{H}_g' \cap \{H_p, \ldots, H_q\}= \{H_p, \ldots, H_q\} \backslash \left( \mathcal{A}_g \cup \mathcal{B}_g \right).$$
Notice that
$$\# \mathcal{H}_g \geq |p-q| - \# \mathcal{A}_g - \# \mathcal{B}_g \geq |p-q|-2L> R+2.$$
Thus, we have proved that every $g \in F$ naturally induces a map $\mathcal{H}_g \to \mathcal{H}(H_{r-1},H_{s+1})$ for some $\mathcal{H}_g \subset \{H_p, \ldots, H_q\}$ of cardinality more than $R+2$. Because $F$ is infinite, there must exist infinitely many pairwise distinct elements $g_0,g_1, \ldots \in F$ inducing the same map. So there exists a subcollection $\mathcal{V} \subset \{ H_p , \ldots, H_q \}$ of cardinality more than $R+2$ such that $g_iJ=g_jJ$ for every $i,j \geq 0$ and every $J \in \mathcal{V}$. As a consequence, there exist two $L$-well-separated hyperplanes $V_1,V_2 \in \mathcal{V}$ separated by more than $R$ other hyperplanes such that the intersection $\mathrm{stab}(V_1) \cap \mathrm{stab}(V_2)$ is infinite, since it contains the elements $g_0^{-1}g_1, g_0^{-1}g_2, \ldots$ which are pairwise distinct by assumption. So far, we have proved the following statement:

\begin{fact}\label{fact:nonWPD}
Let $G$ be a group acting on a CAT(0) cube complex $X$ and $L, \epsilon\geq 0$ two constants. If $x,y \in X$ are two vertices satisfying $\delta_L(x,y) \geq R_0$ where $R_0 \geq 8L+10\epsilon+25$ and such that 
$$\{ g \in G \mid \delta_L(x,gx) \leq \epsilon, \ \delta_L(y,gy) \leq \epsilon \}$$ 
is infinite, then there exist two hyperplanes $V_1,V_2$ separating $x$ and $y$ such that $\mathrm{stab}(V_1) \cap \mathrm{stab}(V_2)$ is infinite and such that $V_1$ and $V_2$ are separated by at least $R_0-8L+10\epsilon+25$ pairwise $L$-well-separated hyperplanes.
\end{fact}

\noindent
Now, by noticing that the intersection $\mathrm{stab}(V_1) \cap \mathrm{stab}(V_2)$ acts on the convex subcomplexes $\mathrm{proj}_{N(V_1)}(N(V_2))$ and $\mathrm{proj}_{N(V_2)}(N(V_1))$, which have finite diameters since $V_1$ and $V_2$ are well-separated, we deduce that that our subgroup stabilises two cubes of $N(V_1)$ and $N(V_2)$. A fortiori, there exist two vertices $a \in N(V_1)$ and $b \in N(V_2)$ such that $\mathrm{stab}(a) \cap \mathrm{stab}(b)$ is infinite. 

\medskip \noindent
Thus, we have proved that, for every $R \geq 0$, there exist two vertices $a,b \in X$ at distance at least $R$ apart such that $\mathrm{stab}(a) \cap \mathrm{stab}(b)$ is infinite. This concludes the proof of our proposition.
\end{proof}

\noindent
Finally, our last preliminary result towards the proof of our main theorem determines which isometries of the cube complex induce loxodromic isometries with respect to the new metric.

\begin{lemma}\label{lem:HXwhenloxo}
Let $X$ be a CAT(0) cube complex, $L \geq 0$ an integer and $g \in \mathrm{Isom}(X)$ an isometry. Then $g$ induces a loxodromic isometry of $(X,\delta_L)$ if and only if $g$ skewers a pair of $L$-well-separated hyperplanes; otherwise, $g$ induces an elliptic isometry of $(X,\delta_L)$.
\end{lemma}

\begin{proof}
Let $g \in \mathrm{Isom}(X)$ be an isometry. If $g$ is an elliptic isometry of $X$, then $g$ must induce an elliptic isometry of $(X,\delta_L)$. Suppose that $g$ is a loxodromic isometry of $X$. Up to subdividing $X$, we may suppose without loss of generality that $g$ acts by translation on a (combinatorial) geodesic line $\gamma$; fix a basepoint $x \in \gamma$. Suppose first that $\mathcal{H}(\gamma)$ contains at most three hyperplanes which are pairwise $L$-well-separated. A fortiori, $g$ does not skewer a pair of $L$-well-separated hyperplanes. Clearly, $\delta_L(x,g^nx) \leq 3$ for every $n \in \mathbb{Z}$, so that $g$ induces an elliptic isometry of $(X,\delta_L)$. Otherwise, suppose that $\mathcal{H}( \gamma)$ contains at least three pairwise $L$-well-separated hyperplanes, say $A,B,C$; by orienting $\gamma$ so that $g$ acts on it by positive translations, say that $A,B,C$ intersect $\gamma$ in that order. Let $n \geq 1$ be an integer so that $g^n \cdot A$ intersects $\gamma$ after $C$. Because $B$ and $C$ are well-separated, there must exists some $m \geq n$ such that $g^mA$ is disjoint from $B$. A fortiori, $g$ skewers the pair of $L$-well-separated hyperplanes $\{A,B\}$. Notice that, since any hyperplane transverse to both $A$ and $g^mA$ must be transverse to both $A$ and $B$, necessarily $A$ and $g^mA$ are $L$-well-separated. A fortiori, $\mathcal{A}= \{ g^{km}A \mid k \in \mathbb{Z} \}$ is an infinite collection of pairwise $L$-well-separated hyperplanes. Let $d$ denote the length of the subpath of $\gamma$ linking $N(A)$ and $N(g^mA)$. Because two vertices $a,b \in \gamma$ are separated by at least $\frac{1}{d} \cdot d(a,b)-1$ hyperplanes of $\mathcal{A}$, we deduce that
$$ \frac{1}{d} \cdot d(a,b)-1 \leq \delta_L(a,b) \leq d(a,b)$$
for every $a,b \in \gamma$. As a consequence, the axis $\gamma$ of $g$ is quasi-isometrically embedded into $(X,\delta_L)$, which implies that $g$ induces a loxodromic isometry of $(X,\delta_L)$. 
\end{proof}

\noindent
An immediate consequence of Lemma \ref{lem:HXwhenloxo} and Theorem \ref{thm:contractingisom} is:

\begin{cor}\label{cor:LoxoContracting}
Let $X$ be a CAT(0) cube complex and $g \in \mathrm{Isom}(X)$ an isometry. Then $g$ is a contracting isometry of $X$ if and only if it defines a loxodromic isometry of $(X,\delta_L)$ for $L$ sufficiently large.
\end{cor}

\begin{proof}[Proof of Theorem \ref{thm:mainHX}.]
The first two points of the theorem follows directly from Propositions \ref{prop:HXhyp} and \ref{prop:HXacyl}. The last point is a consequence of Corollary \ref{cor:LoxoContracting}. 
\end{proof}

\paragraph{Application 1: Acylindrical hyperbolicity.} Our hyperbolic models allow us to give an alternative and purely cubical proof of the fact a group acting on a CAT(0) cube complex with at least one WPD contracting isometry must be either virtually cyclic or acylindrically hyperbolic

\begin{prop}
Let $G$ be a group acting on a CAT(0) cube complex $X$ and $g \in G$ a WPD contracting isometry. There exists some $L_0 \geq 0$ such that, for every $L \geq L_0$, $g$ is a WPD loxodromic isometry of $(X,\delta_L)$.
\end{prop}

\begin{proof}
For convenience, we fix a combinatorial axis $\gamma$ of $g$ (which exists up to subdividing $X$). According to Theorem \ref{thm:WPDcontracting}, there exists some $L_0 \geq 0$ such that $g$ skewers a pair of $L_0$-well-separated hyperplanes $J_1$ and $J_2$ such that $\mathrm{stab}(J_1) \cap \mathrm{stab}(J_2)$ is finite. Notice that $J_1$ and $J_2$ necessarily intersect $\gamma$. Fix some $L \geq L_0$ and let $D$ denote the maximal number of pairwise $L$-well-separated hyperplanes separating $J_1$ and $J_2$. Notice that we already know from Lemma \ref{lem:HXwhenloxo} that $g$ defines a loxodromic isometry of $(X,\delta_L)$. If $g$ does not induce a WPD isometry of $(X,\delta_L)$, then we can find a vertex $x \in \gamma$, a constant $\epsilon \geq 0$ and a sufficiently large integer $m \geq 0$ such that $\delta_L(x,g^mx) \geq D+\|g\| +8L$ (where $\|g\|$ the translation length of $g$) and such that
$$\{h \in G \mid \delta_L(x,hx) \leq \epsilon, \ \delta_L(g^mx,hg^mx) \leq \epsilon \}$$ 
is infinite. It follows from Fact \ref{fact:nonWPD} that there exist two hyperplanes $H_1,H_2$ separating $x$ and $g^mx$ such that $\mathrm{stab}(H_1) \cap \mathrm{stab}(H_2)$ is infinite and such that $H_1$ and $H_2$ are separated by at least $D+ \|g\|+3$ pairwise $L$-well-separated hyperplanes. Up to translating $H_1$ and $H_2$ by a power of $g$, we may suppose without loss of generality that $J_1$ and $J_2$ both separate $H_1$ and $H_2$. Because there exist only finitely many hyperplanes separating $H_1$ and $H_2$, we know that $\mathrm{stab}(H_1) \cap \mathrm{stab}(H_2)$ contains a finite-index subgroup included into $\mathrm{stab}(J_1) \cap \mathrm{stab}(J_2)$, which is impossible since $\mathrm{stab}(J_1) \cap \mathrm{stab}(J_2)$ is finite and $\mathrm{stab}(H_1) \cap \mathrm{stab}(H_2)$ infinite. Consequently, $g$ must be a WPD isometry of $(X,\delta_L)$. 
\end{proof}

\begin{cor}\label{cor:altBBF}
If a group acts on a CAT(0) cube complex with at least one WPD contracting isometry, then it is either virtually cyclic or acylindrically hyperbolic.
\end{cor}

\paragraph{Application 2: Stable subgroups.} Notice that the third point of Theorem \ref{thm:mainHX} implies that, if $g \in G$ is a contracting isometry, or equivalently if $\langle g \rangle$ is a stable subgroup of $G$, then the orbits of $\langle g \rangle$ are quasi-isometrically embedded into $H(X)$. We generalise this observation to arbitrary stable subgroups. 

\begin{thm}\label{thm:HXstablesub}
Let $G$ be a group acting geometrically on a CAT(0) cube complex $X$ and $H \subset G$ a subgroup. If $H$ is a stable subgroup, then there exists some $L_0 \geq 0$ such that its orbits in $(X,\delta_L)$ are quasi-isometrically embedded for every $L \geq L_0$. 
\end{thm}

\noindent
Our statement will be a straightforward consequence of the following proposition:

\begin{prop}\label{prop:HXstable}
Let $X$ be a cocompact CAT(0) cube complex and $Y \subset X$ a convex subcomplex. Then $Y$ is a stable subcomplex if and only if there exists some $L_0 \geq 0$ such that the inclusion $Y \subset X$ induces a quasi-isometric embedding $Y \to (X,\delta_L)$ for every~$L \geq L_0$. 
\end{prop}

\noindent
Given a metric space $M$, a subspace $N \subset M$ is \emph{stable} if, for every $A \geq 1$ and $B \geq 0$, there exists a constant $K \geq 0$ such that the Hausdorff distance between any two $(A,B)$-quasi-geodesics linking two points of $H$ is at most $K$. It is worth noticing that $N$ is stable if and only if it is a Morse subspace in which any two quasi-geodesics stays at finite Hausdorff distance (depending only on the parameters of the quasi-geodesics), or equivalently, if it is a hyperbolic Morse subspace.

\medskip \noindent
As a preliminary result, we prove the following statement, which we think to be of independent interest.

\begin{lemma}\label{lem:HXcontracting}
Let $X$ be a CAT(0) cube complex and $Y \subset X$ a convex subcomplex. If $Y$ is a contracting subcomplex, then there exists some $C \geq 0$ such that 
$$\frac{1}{C} \cdot \delta_{L-C}^Y(x,y)-C \leq \delta_L^X(x,y) \leq \delta_L^Y(x,y)$$
for every $L \geq 0$ and every $x,y \in Y$, where $\delta^X_L$ denotes the distance $\delta_L$ defined on $X$ and $\delta^Y_L$ its restriction to $Y$.
\end{lemma}

\begin{proof}
Fix some $L \geq 0$ and a constant $C \geq 0$ so that $Y$ is $C$-contracting, i.e., every join of hyperplanes $(\mathcal{H},\mathcal{V})$ satisfying $\mathcal{H} \subset \mathcal{H}(Y)$ and $\mathcal{V} \cap \mathcal{H}(Y) = \emptyset$ must be $C$-thin (see Proposition \ref{prop:contracting}). Let $x,y \in Y$ be two vertices. Because two hyperplanes of $X$ which are $L$-well-separated in $X$ are clearly $L$-well-separated in $Y$, necessarily
$$\delta_L^X(x,y) \leq \delta_L^Y(x,y).$$
Now, let $H_1, \ldots, H_r$ be a maximal collection of hyperplanes of $Y$ which are pairwise $L$-well-separated in $Y$. So $r=\delta^Y_L(x,y)$. Fix some $1 \leq i \leq r-C-1 $ and let $\mathcal{K}$ be a collection of hyperplanes of $X$ transverse to both $H_i$ and $H_{i+C+1}$ which does not contain any facing triple. Because $H_i$ and $H_{i+C+1}$ are $L$-well-separated in $Y$, necessarily $\# \left( \mathcal{K} \cap \mathcal{H}(Y) \right) \leq L$. And because $Y$ is $C$-contracting, $\# \left( \mathcal{K} \cap \mathcal{H}(Y)^c \right) \leq C$. Therefore, $\# \mathcal{K} \leq L+C$. This shows that $H_i$ and $H_{i+C+1}$ are $(L+C)$-well-separated in $X$. Consequently,
$$\delta_{L+C+1}^X(x,y) \geq \delta_{L+C}^X(x,y) \geq \frac{1}{C+1} \cdot \delta_L^Y(x,y)-C-1.$$
This concludes the proof of our lemma.
\end{proof}

\begin{proof}[Proof of Proposition \ref{prop:HXstable}.]
Suppose that $Y$ is a stable subcomplex, and let $C,L_0 \geq 0$ denote the constants respectively given by Lemmas \ref{lem:HXcontracting} and \ref{lem:HXwhenXhyp}. Fix some $L \geq L_0+C$. We know from Lemma \ref{lem:HXwhenXhyp} that the metrics $\delta_L^Y$ and $\delta_{L-C}^Y$ are quasi-isometric to the metric of $Y$. We conclude from Lemma \ref{lem:HXcontracting} that the restriction of $\delta_L^X$ to $Y$ has to be quasi-isometric to the metric of $Y$.

\medskip \noindent
Conversely, suppose that there exists some $L \geq 0$ such that the canonical map $Y \to (X,\delta_L)$ is a quasi-isometric embedding. As a consequence, $Y$ is hyperbolic and there exists a constant $A \geq 0$ such that, for every vertices $x,y \in X$, the inequality $d(x,y) \geq A$ implies $\delta_L(x,y) \geq 2$. Let $(\mathcal{H}, \mathcal{V})$ be a grid of hyperplanes satisfying $\# \mathcal{V} \geq A+2$, $\mathcal{V} \subset \mathcal{H}(Y)$ and $\mathcal{H} \cap \mathcal{H}(Y) = \emptyset$. Write $\mathcal{V}$ as $\{V_1, \ldots, V_r\}$ so that $V_i$ separates $V_{i-1}$ and $V_{i+1}$ for every $2 \leq i \leq r-1$. Fix two vertices $x \in Y \cap N(V_1)$ and $y \in Y \cap N(V_r)$ minimising the distance between $Y \cap N(V_1)$ and $Y \cap N(V_r)$. A fortiori, $x$ and $y$ are separated by $V_2, \ldots, V_{r-1}$, hence $d(x,y) \geq A$, and finally $\delta_L(x,y) \geq 2$. So there exist two $L$-well-separated hyperplanes $J_1$ and $J_2$ separating $x$ and $y$. According to Lemma \ref{lem:twomin}, $J_1$ and $J_2$ separates $Y \cap N(V_1)$ and $Y \cap N(V_r)$. Moreover, because the projection of $N(V_1)$ onto $Y$ turns out to be $Y \cap N(V_1)$ (as a consequence of Lemma \ref{lem:projinter}), we deduce from Lemma \ref{lem:vertextoproj} and Proposition \ref{prop:proj} that any hyperplane intersecting $N(V_1)$ outside $Y$ must be disjoint from $Y$. Therefore, $J_1$ and $J_2$ must separate $V_1$ and $V_r$, so that $V_1$ and $V_r$ have to be $L$-well-separated as well, hence $\# \mathcal{H} \leq L$. It follows from Proposition \ref{prop:contracting} that $Y$ is contracting. Consequently, $Y$ is a stable subcomplex. 
\end{proof}

\begin{proof}[Proof of Theorem \ref{thm:HXstablesub}.]
Fix a basepoint $x \in X$ and suppose that $H$ is a stable subgroup. As a consequence of Corollary \ref{cor:MorseCriterion}, the convex hull $Y$ of $H \cdot x$ in $X$ is a contained into a neighborhood of $Y$. A fortiori, $Y$ is a stable subcomplex. It follows from Proposition~\ref{prop:HXstable} that there exists some $L_0 \geq 0$ such that $Y$ quasi-isometrically embeds into $(X,\delta_L)$ for every $L \geq L_0$. A fortiori, $H \cdot x$ quasi-isometrically embeds into $(X,\delta_L)$ for every $L \geq L_0$. 
\end{proof}

\noindent
It may be expected that the converse of Theorem \ref{thm:HXstablesub} holds, i.e., if the orbits of the subgroup $H$ quasi-isometrically embed into $(X,\delta_L)$ for some $L \geq 0$, then $H$ turns out to be a stable subgroup. In view of Proposition \ref{prop:HXstable}, the only point to verify is that $H$ is a convex-cocompact group, or equivalently, that the convex hull of an $H$-orbit lies in a neighborhood of this orbit. However, the implication which interests us is really the one proved by Theorem~\ref{thm:HXstablesub} because it implies restrictions on the possible stable subgroups of a given group acting geometrically on some CAT(0) cube complex. For instance, we are able to reprove a result which follows from \cite{PurelyLoxoRAAG, KimKoberdaGeometryCurveGraph}. (An alternative argument can also be found at the end of the proof of Theorem \ref{thm:MorseinRAAG}.)

\begin{prop}\label{prop:stableinRAAG}
A stable subgroup in a right-angled Artin group is necessarily free.
\end{prop}

\begin{proof}[Sketch of proof.]
Let $A$ be a free irreducible right-angled Artin group and $X$ the universal cover of the associated Salvetti complex. It is not difficult to show that two hyperplanes of $X$ are well-separated if and only if they are \emph{strongly separated}, i.e., no hyperplane of $X$ is transverse to both of them. Consequently, for every $L \geq 0$ the metric space $(X,\delta_L)$ is isometric to $(X,\delta_0)$, which turns out to be quasi-isometric to the contact graph of $X$. A fortiori, $(X,\delta_L)$ is a quasi-tree for every $L \geq 0$. It follows from Theorem \ref{thm:HXstablesub} that any stable subgroup of $A$ must be quasi-isometric to a tree, and so must be virtually free (see for instance \cite[Th\'eor\`eme 7.19]{GhysdelaHarpe}) and finally must be free since $A$ is torsion-free (see \cite{StallingsTorsionFree}). 
\end{proof}

\paragraph{Application 3: Regular elements.} Given a product $X=X_1 \times \cdots \times X_n$ of irreducible CAT(0) cube complexes, an isometry $g \in \mathrm{Isom}(X)$ is \emph{regular} if it induces a contracting isometry of $X_i$ for every $1 \leq i \leq n$. Regular isometries have been introduced in \cite{MR2827012} by analogy to regular semi-simple elements for symmetric spaces. Our goal is to give an alternative proof of \cite[Theorem 1.5]{RandomCC} (which is an improvement of \cite[Theorem~D]{MR2827012}), namely:

\begin{thm}\label{thm:Regular}
Let $G$ be a group acting essentially on a product $X=X_1 \times \cdots \times X_n$ of irreducible, unbounded and finite-dimensional CAT(0) cube complexes. Assume that $G$ does not a finite in $X$ nor in its visual boundary. Then $G$ contains a regular element.
\end{thm}

\noindent
The argument of \cite{RandomCC} is probabilistic. We propose here an argument based on cubical and hyperbolic geometries. In addition to the hyperbolic models we introduced, we need the following statement:

\begin{prop}\label{prop:simul}
Let $G$ be a group acting by isometries on quasi-geodesic hyperbolic spaces $X_1, \ldots, X_n$. Assume that:
\begin{itemize}
	\item for every $1 \leq i \leq n$, $G$ contains a loxodromic isometry of $X_i$;
	\item and for every $1 \leq i \leq n$, an element of $G$ either has bounded orbits in $X_i$ or is loxodromic.
\end{itemize}
Then there exists an element $g \in G$ which defines a loxodromic isometry of $X_i$ for every $1 \leq i \leq n$. 
\end{prop}

\noindent
An elementary proof of this result can be found in \cite{Simul} under two strengthened assumptions: hyperbolic spaces are supposed to be geodesic, and an element of $G$ which has a bounded orbit is supposed to fix a point. Proposition \ref{prop:simul} follows from \cite{Simul} as a consequence of the following two observations:
\begin{itemize}
	\item Let $X$ be a quasi-geodesic hyperbolic space. There exists a constant $C \geq 0$ such that, if $Y$ denotes the graph whose vertex-set is $X$ and whose edges link two points within distance $C$, then $Y$ is connected. Then $Y$ is geodesic hyperbolic space in which $X$ is quasi-dense and quasi-isometrically embedded. 
	\item Let $X$ be a geodesic $\delta$-hyperbolic space. Fix a bounded metric space $M$ and a basepoint $m \in M$. Let $Y$ denote the metric space obtained from $X$ by adding a copy $M_S$ of $M$ for every subset $S \subset X$ of diameter at most $5\delta$ and by linking the basepoint $m \in M_S$ to every point of $S$ by a segment $[0,1]$. Then $Y$ is geodesic hyperbolic space in which $X$ is quasi-dense and quasi-isometrically embedded. Moreover, $M$ can be chosen so that any isometry of $Y$ leaves $X$ invariant; for instance, take $M= [0,a] \times [0,a]$ with $a$ large compared to $\delta$.  If $g \in \mathrm{Isom}(Y)$ has a bounded orbit, then it has a bounded orbit in $X$. According to \cite[Lemma III.$\Gamma$.3.3]{}, $g$ must have an orbit $S$ of diameter at most $5\delta$. Therefore, $g$ fixes the basepoint of $M_S$. 
\end{itemize}
Theorem \ref{thm:Regular} is now an easy consequence of the combination of Proposition \ref{prop:simul} with our hyperbolic models of cube complexes.

\begin{proof}[Proof of Theorem \ref{thm:Regular}.]
Up to replacing $G$ with one of its finite-index subgroups, we suppose that $G$ preserves the product structure of $X$. For every $1 \leq i \leq n$, $G$ acts essentially on $X_i$ without fixing a point in the visual boundary. It follows from Theorem \ref{thm:CS} that $G$ contains a contracting isometry of $X_i$, so that, according to Corollary \ref{cor:LoxoContracting}, there exists $L(i) \geq 0$ such that $G$ contains a loxodromic isometry of $(X_i,\delta_{L(i)})$. Notice that, according to Lemma \ref{lem:HXwhenloxo}, we know that an element of $G$ either has bounded orbits in $(X_i,\delta_{L(i)})$ or is loxodromic. By applying Proposition \ref{prop:simul} to the actions of $G$ on $(X_1,\delta_{L(1)}), \ldots, (X_n, \delta_{L(n)})$, we deduce that $G$ contains an element $g$ defining a loxodromic isometry of $(X_i,\delta_{L(i)})$ for every $1 \leq i \leq n$. We conclude from Corollary \ref{cor:LoxoContracting} that $g$ defines a contracting isometry of $X_i$ for every $1 \leq i \leq n$, i.e., is a regular isometry. 
\end{proof}

\paragraph{Open questions.} We conclude this section by stating a few open questions about our hyperbolic models. First of all, is it really a model for universal acylindrical actions? Theorem \ref{thm:mainHX} does not completely prove this assertion, since our action is non-uniformly acylindrical. 

\begin{question}\label{q:Universal}
Let $G$ be a group acting geometrically on a CAT(0) cube complex $X$. Does there exist an $L\geq 0$ such that any two disjoint hyperplanes of $X$ are either $L$-well-separated or both transverse to infinitely many hyperplanes? If so, is the induced action $G \curvearrowright (X,\delta_L)$ acylindrical?
\end{question}

\noindent
Interestingly, the lack of acylindricity in the proof of Proposition \ref{prop:HXacyl} seems to have the same origin as the lack of acylindricity in the proofs of \cite[Theorem 7.1]{coningoff} and \cite[Theorem 22]{MoiAcylHyp}. Therefore, understanding this problem would be interesting. Another motivation would be to deduce from \cite{BowditchOneEndedSub} that a group acting geometrically on a CAT(0) cube complex contains only finitely many conjugacy classes of \emph{purely contracting subgroups} (i.e., subgroups containing only contracting isometries) isomorphic to a given finitely presented one-ended group. 

\medskip \noindent
From now on, given a CAT(0) cube complex $X$, we fix one of its hyperbolic models $H(X)$, hopefully the metric space $(X,\delta_L)$ where $L$ is the constant given by a positive answer to Question \ref{q:Universal}. 

\medskip \noindent
A natural question would be to study the behavior of $H(X)$ up to quasi-isometry. 

\begin{question}
Does a quasi-isometry $X \to Y$ between cocompact CAT(0) cube complexes induces a quasi-isometry $H(X) \to H(Y)$? a homeomorphism $\partial H(X) \to \partial H(Y)$?
\end{question}

\noindent
A positive answer to this question would allow us to define the \emph{hyperbolic boundary} $\partial_hG$ of a group $G$ acting geometrically on a CAT(0) cube complex $X$ as the Gromov boundary of the hyperbolic space $H(X)$. As a consequence, Lemma \ref{lem:HXcontracting} would imply that, for every group $G$ acting geometrically on some CAT(0) cube complex and for every Morse subgroup $H \subset G$, the hyperbolic boundary $\partial_hH$ of $H$ topologically embeds into the hyperbolic boundary $\partial_hG$ of $G$. As a particular case, if $H$ is a stable subgroup, then its Gromov boundary $\partial H$ topologically embeds into $\partial_hG$. With respect to this vocabulary, the proof of Proposition \ref{prop:stableinRAAG} amounts to saying that the hyperbolic boundary of a right-angled Artin group is a Cantor set, so that the Gromov boundary of any infinite stable subgroup must be a Cantor set as well, which implies that these groups must be free. 

\medskip \noindent
Basic (but non-trivial) results on Gromov boundaries of hyperbolic groups is that a multi-ended hyperbolic group splits over a finite subgroup and that a hyperbolic group with a Cantor set as its boundary must be virtually free. Are there similar statements with respect to our hyperbolic boundary?

\begin{question}
Let $X$ be a cocompact CAT(0) cube complex. When is $\partial H(X)$ a Cantor set? When is it connected? 
\end{question}

\appendix

\section{Crossing graphs as curve graphs}\label{section:crossing}

\noindent
Recall that the \emph{crossing graph} $\Delta X$ of a CAT(0) cube complex $X$ is the graph whose vertices are the hyperplanes of $X$ and whose edges link two transverse hyperplanes. This graph is a natural analogue of curve graphs of surfaces, but usually two objections are given against this analogy: first, the crossing graph may be disconnected; and next, every graph turns out to be the crossing graph of some CAT(0) cube complex, which prevents, in particular, the crossing graphs from being always hyperbolic. In this section, our goal is to show that these objections are not justified, and that crossing graphs are not so different from Hagen's contact graphs. 

\medskip \noindent
First, thanks to \cite[Lemma 2]{SplittingObstruction}, we understand precisely when the crossing graph is disconnected:

\begin{prop}
Let $X$ be a CAT(0) cube complex. The crossing graph $\Delta X$ is disconnected if and only if $X$ contains a cut vertex.
\end{prop}

\noindent
Therefore, when the crossing graph is disconnected, one can consider the graph $T$ whose vertices are the cut vertices of $X$ and the connected components of the complement, and whose edges link a cut vertex to all the components containing it. Because $X$ is simply connected, $T$ turns out to be a tree, so that Bass-Serre theory implies that any group acting on $X$ splits as a graph of groups such that vertex-groups are stabilisers of cut vertices or stabilisers of components. So, by the arboreal structure $T$ on $X$, we reduce the situation to actions on cube complexes whose crossing graphs are connected. As a particular case of the previous discussion, combined with Stallings' theorem, it follows that, if a one-ended group acts minimally and geometrically on some CAT(0) cube complex, then the crossing graph is necessarily connected.

\medskip \noindent
Next, if one considers only CAT(0) cube complexes which are uniformly locally finite, then crossing graphs turn out to be hyperbolic.

\begin{prop}\label{prop:crossinghyp}
Let $X$ be a uniformly locally finite CAT(0) cube complex without cut vertex. The crossing graph $\Delta X$ is a quasi-tree.
\end{prop}

\noindent
Our proof follows essentially the arguments used \cite[Theorem 3.1.1]{Hagenthesis}. In particular, our goal is to apply the bottleneck criterion \cite{bottleneck}:

\begin{prop}
A geodesic metric space $Y$ is quasi-isometric to a tree if and only if there exists a constant $\delta>0$ such that, for every $x,y \in Y$, there is a \emph{midpoint} $m$ between $x$ and $y$, i.e.
$$d(m,x) = \frac{1}{2} d(x,y) = d(m,y),$$
with the property that any path $\gamma : [a,b] \to Y$ joining $x$ to $y$ satisfies $d(\gamma(t),m)<\delta$ for some $t \in [a,b]$. 
\end{prop}

\noindent
We begin by stating and proving two preliminary lemmas about the metric in $\Delta X$. 

\begin{lemma}\label{lem:QT1}
Let $X$ be a CAT(0) cube complex. Let $J_1, \ldots, J_n$ be a path in $\Delta X$. For every hyperplane $H$ separating $J_1$ and $J_n$ in $X$, there exists some $1 \leq i \leq n$ such that $d_{\Delta X}(H,J_i) \leq 1$.
\end{lemma}

\begin{proof}
There must exist some $J_i$ such that either $H=J_i$ or $H$ transverse to $J_i$, since otherwise $J_1, \ldots, J_n$ would be included into the halfspace delimited by $H$ which does not contain $J_n$, which is absurd. A fortiori, $d_{\Delta X}(H,J_i) \leq 1$.
\end{proof}

\begin{lemma}\label{lem:QT2}
Let $X$ be a CAT(0) cube complex. Suppose that the link of every vertex of $X$ has diameter at most $R$ for some uniform $R \geq 0$. Let $J_1, \ldots, J_n$ be a geodesic in $\Delta X$. For every $1 \leq i \leq n$, there exists a hyperplane $H$ separating $J_1$ and $J_n$ in $X$ such that $d_{\Delta X}(J_i,H) \leq 3+ R$.
\end{lemma}

\begin{proof}
Let $H_1, \ldots, H_m$ be a maximal collection of pairwise disjoint hyperplanes separating $J_1$ and $J_2$. Suppose that $H_j$ separates $H_{j-1}$ and $H_{j+1}$ for every $2 \leq j \leq m-1$. As a consequence of Lemma \ref{lem:twomin}, for every $1 \leq j \leq m-1$, the hyperplanes $H_j$ and $H_{j+1}$ must be tangent, so that $d_{\Delta X}(H_j,H_{j+1}) \leq R$. Fix some $1 \leq i \leq n$. If $J_i$ is transverse or equal to some $H_j$, then $d_{\Delta X}(J_i,H_j) \leq 1$ and we are done. Notice also that $J_i$ cannot be separated by $J_1$ from $J_n$, and similarly by $J_n$ from $J_1$, because otherwise it would be possible to shorten the path $J_1, \ldots, J_n$ in $\Delta X$. The last possible configuration is when $J_i$ lies in the subspace delimited by $H_j$ and $H_{j+1}$ for some $1 \leq j \leq m-1$. As a consequence of Lemma \ref{lem:QT1}, there exist $1 \leq r,s \leq n$ satisfying $r < i <s$ such that $d_{\Delta X}(J_r,H_j) \leq 1$ and $d_{\Delta X}(J_s,H_{j+1}) \leq 1$. So
$$\begin{array}{lcl} d_{\Delta X}(J_i,H_j) & \leq & d_{\Delta X}(J_i,J_r) + d_{\Delta X}(J_r,H_s) \leq d_{\Delta X}(J_s,J_r) +1 \\ \\ & \leq & d_{\Delta X}(H_j,H_{j+1}) +3 \leq R+3 \end{array}$$
concluding the proof.
\end{proof}

\begin{proof}[Proof of Proposition \ref{prop:crossinghyp}.]
Because $X$ does not contain any cut vertex, the link of every vertex of $X$ is finite; and because $X$ is uniformly locally finite, we deduce that the links of vertices of $X$ have diameters uniformly bounded, say by some constant $R \geq 0$. Let $J,H$ be two hyperplanes of $X$. Fix a geodesic between $J$ and $H$, and let $K$ be one of its vertices at distance at most $1/2$ from its midpoint $M$. Let $A_1, \ldots, A_n$ be any path between $J$ and $H$. According to Lemma \ref{lem:QT1}, there exists a hyperplane $S$ separating $J$ and $H$ such that $d_{\Delta X}(K,S) \leq 3+ R$; and according to Lemma \ref{lem:QT2} that there exists some $1 \leq i \leq n$ such that $d_{\Delta X}(A_i,S) \leq 1$. Therefore,
$$d_{\Delta X}(A_i,M) \leq d_{\Delta X}(A_i,S) + d_{\Delta X}(S,M) \leq R+9/2.$$
It follows from the bottleneck criterion that $\Delta X$ is a quasi-tree.
\end{proof}

\noindent
As a consequence, if a one-ended group acts minimally and geometrically on some CAT(0) cube complex, then the crossing graph is connected and quasi-isometric to a tree. This observation make crossing graphs good candidate for curve graphs of CAT(0) cube complexes. In fact, the next proposition implies that crossing graphs and contact graphs are essentially identical. 

\begin{prop}\label{prop:QI}
Let $X$ be a uniformly locally finite CAT(0) cube complex without cut vertex. The canonical map $\Delta X \to \Gamma X$ is a quasi-isometry. 
\end{prop}

\noindent
The key point to prove this proposition is that the distance in $\Delta X$ coincides coarsely with the maximal number of pairwise strongly separated hyperplanes separating two given hyperplanes. (Recall that two hyperplanes are \emph{strongly separated} if no other hyperplane is transverse to both of them.) This idea is made precise by the next two lemmas. 

\begin{lemma}\label{qi1}
Let $X$ be a CAT(0) cube complex. Suppose that there exists some $R \geq 1$ such that the link of every vertex of $X$ has diameter at most $R$. If $J,H$ are two hyperplanes satisfying $d_{\Delta X}(J,H) \geq 11Rn$, there exist at least $n$ pairwise strongly separated hyperplanes separating $J$ and $H$ in $X$. 
\end{lemma}

\begin{proof}
Let $J=V_0,V_1, \ldots, V_{r-1},V_r=H$ be a geodesic in $\Delta X$ between $J$ and $H$. According to Lemma \ref{lem:QT2}, for every $1 \leq k \leq r-1$, there exists a hyperplane $S_k$ separating $J$ and $H$ such that $d_{\Gamma X}(V_k,S_k) \leq 3+R$. For every $1 \leq k \leq (r-1)/5$ and every $1 \leq j \leq (r-1)/5-k$, we have
\begin{center}
$\begin{array}{lcl} d_{\Delta X}(S_{11Rk},S_{11R(k+j)}) & \geq & d_{\Delta X}(V_{11Rk},V_{11R(k+j)})-d_{\Delta X}(V_{11Rk},S_{11Rk}) \\ \\ & & -d_{\Delta X}(V_{11R(k+j)},S_{11R(k+j)}) \\ \\ & \geq & 11Rj - 2(3+R) \geq 3; \end{array}$
\end{center}
A fortiori, $S_{11Rk}$ and $S_{11R(k+j)}$ are strongly separated. Therefore, $\{S_{11Rk} \mid 1 \leq k \leq n\}$ defines a collection $n$ pairwise strongly separated hyperplanes separating $J$ and $H$, concluding the proof.
\end{proof}

\begin{lemma}\label{qi2}
Let $X$ be a CAT(0) cube complex. Let $J$ and $H$ be two hyperplanes. If they are separated in $X$ by $n$ pairwise strongly separated hyperplanes $V_1, \ldots, V_n$, such that $V_i$ separates $V_{i-1}$ and $V_{i+1}$ for every $2 \leq i \leq n-1$, then $d_{\Delta X}(J,H) \geq n$. 
\end{lemma}

\begin{proof}
Let $J=S_0,S_1, \ldots, S_{r-1},S_r=H$ be a geodesic in $\Gamma X$ between $J$ and $H$. According to Lemma \ref{lem:QT1}, for every $1 \leq k \leq n$, there exists some $1 \leq n_k \leq r-1$ such that $d_{\Gamma X}(V_k,S_{n_k}) \leq 1$. Notice that, for every $1 \leq i< j \leq n$, because $V_i$ and $V_j$ are strongly separated, necessarily $n_i \neq n_j$. Let $\varphi$ be a permutation so that the sequence $(n_{\varphi(k)})$ is increasing. We have
\begin{center}
$\begin{array}{lcl} d_{\Delta X}(J,H)& = & \displaystyle \sum\limits_{k=1}^n d_{\Delta X}(S_{n_{\varphi(k)}}, S_{n_{\varphi(k+1)}}) \\ \\  & \geq & \displaystyle \sum\limits_{k=1}^n \left( d_{\Delta X}(V_{\varphi(k)},V_{\varphi(k+1)}) - d_{\Delta X}(V_{\varphi(k)}, S_{n_{\varphi(k)}}) -d_{\Delta X}(V_{\varphi(k+1)}, S_{n_{\varphi(k+1)}}) \right) \\ \\ & \geq & \displaystyle \sum\limits_{k=1}^n (3-1-1) = n, \end{array}$
\end{center}
where we used the inequality $d_{\Gamma X}(V_{\varphi(k)},V_{\varphi(k+1)}) \geq 3$, which precisely means that $V_{\varphi(k)}$ and $V_{\varphi(k+1)}$ are strongly separated. This completes the proof. 
\end{proof}

\begin{proof}[Proof of Proposition \ref{prop:QI}.]
Lemmas \ref{qi1} and \ref{qi2} show that the metric in $\Delta X$ is coarsely equivalent to the maximal number of separating pairwise strongly separated hyperplanes. The same conclusion holds for the metric in $\Gamma X$ according to \cite[Proposition 23]{MoiAcylHyp}. The conclusion follows. 
\end{proof}

\noindent
It is worth noticing that, if a group acts on the CAT(0) cube complex we are considering, the quasi-isometry provided by the previous proposition is equivariant. As a consequence, the conclusion of Theorem \ref{thm:contact} also holds with respect to the contact graph.

\section{Morse subgroups of right-angled Artin groups}\label{section:MorseInRAAG}

\noindent
As promised in Application \ref{app:MorseRaag}, this appendix is dedicated to the proof of the following statement:

\begin{thm}\label{thm:MorseinRAAG}
A Morse subgroup in a freely irreducible right-angled Artin group is either a finite-index subgroup or a free subgroup containing only contracting isometries.
\end{thm}

\noindent
Our proof to this theorem is based on the combinatorial boundary as introduced in \cite{article3}. (An alternative argument can be found in \cite{StronglyQC}.) We begin by defining the vocabulary which we will use below.

\medskip \noindent
Fix a CAT(0) cube complex $X$. For any subcomplex $Y \subset X$ we denote by $\mathcal{H}(Y)$ the set of hyperplanes of $X$ dual to some edge of $Y$. We define a partial order $\prec$ on the set of the combinatorial rays of $X$ by: $r_1 \prec r_2$ if all but finitely many hyperplanes of $\mathcal{H}(r_1)$ belong to $\mathcal{H}(r_2)$, denoted by $\mathcal{H}(r_1) \underset{a}{\subset} \mathcal{H}(r_2)$. Notice that, if $\partial^cX$ denotes the quotient of the set of combinatorial rays by the relation $\sim$ defined by: $r_1 \sim r_2$ if and only if $r_1 \prec r_2$ and $r_2 \prec r_1$; then $\prec$ induces naturally a partial order on $\partial^cX$, also denoted by $\prec$ for convenience. The poset $(\partial^cX, \prec)$ is the \emph{combinatorial boundary} of $X$. If $Y \subset X$ is a subcomplex, the \emph{relative combinatorial boundary $\partial^cY$ of $Y$ in $X$} is the subset of $\partial^cX$ corresponding to the set of the combinatorial rays included into $Y$.

\medskip \noindent
The boundary $\partial^cX$ can be endowed with a graph structure by adding an edge between two $\prec$-comparable rays. In this context, the \emph{$\prec$-components} of $\partial^cX$ correspond to the connected components of this graph. In particular, a point of $\partial^cX$ is \emph{isolated} if the $\prec$-component containing it is a single point. Finally, we denote by $d_{\prec}$ the graph metric on $\partial^cX$. 

\medskip \noindent
The following observation will be useful later:

\begin{lemma}\label{lem:increasingsequence}
Let $X$ be a finite-dimensional CAT(0) cube complex. Then any increasing chain in $(\partial^cX,\prec)$ has length at most $\dim(X)$.
\end{lemma}

\begin{proof}
Let $r_1 \prec \cdots \prec r_n$ be an increasing chain in $(\partial^cX,\prec)$. Our goal is to prove that $n \leq \dim(X)+1$. We begin by proving the following claim:

\begin{claim}\label{claim:raytransverse}
Let $\rho_1,\rho_2$ be two rays satisfying $\rho_1 \prec \rho_2$. All but finitely many hyperplanes of $\mathcal{H}(\rho_2) \backslash \mathcal{H}(\rho_1)$ are transverse to all but finitely many hyperplanes of $\mathcal{H}(\rho_1)$. 
\end{claim}

\noindent
Suppose that $J \in \mathcal{H}(\rho_2) \backslash \mathcal{H}(\rho_1)$ is a hyperplane which does not separate $\rho_1(0)$ and $\rho_2(0)$. Let $e$ denote the edge of $\rho_2$ which is dual to $J$ and let let $H \in \mathcal{H}(\rho_1)$ be a hyperplane which does not separate $\rho_1(0)$ and $\rho_2(0)$ nor $\rho_2(0)$ and $e$. Let $e_1,e_2$ denote the edges of $\rho_1,\rho_2$ respectively which are dual to $H$. By noticing that $J$ separates $\rho_2(0)$ and $e$, but does not separate $\rho_1(0)$ and $\rho_2(0)$ nor $\rho_1(0)$ and $e_1$, it follows that $J$ separates $e_1$ and $e_2$. A fortiori, $J$ and $H$ must be transverse. This proves our claim.

\medskip \noindent
Now, let us construct a sequence of hyperplanes $J_1, \ldots, J_{n-1}$ by applying iteratively Claim \ref{claim:raytransverse}. Let $J_1 \in \mathcal{H}(r_n) \backslash \mathcal{H}(r_{n-1})$ be a hyperplane which is transverse to all but finitely many hyperplanes of $\mathcal{H}(r_{n-1})$; up to replacing $r_{n-1}$ with a subray starting from $r_{n-1}(k)$ for some sufficiently large $k$, we may suppose without loss of generality that $J_1$ is transverse to all the hyperplanes of $\mathcal{H}(r_{n-1})$. Similarly, fix a hyperplane $J_2 \in \mathcal{H}(r_{n-1}) \backslash \mathcal{H}(r_{n-2})$ which is transverse to all the hyperplanes of $\mathcal{H}(r_{n-2})$ (up to replacing $r_{n-2}$ with a subray); and so on. Thus, we get a sequence of pairwise transverse hyperplanes $J_1, \ldots, J_{n-1}$. A fortiori, $n-1 \leq \dim(X)$, which proves our lemma. 
\end{proof}

\noindent
Now, let us show that the relative combinatorial boundary of a contracting subcomplex in the whole combinatorial boundary satisfies some specific properties.

\begin{definition}
Let $X$ be a CAT(0) cube complex. A subset $S \subset \partial^cX$ is \emph{full} if every point of $\partial^c X$ which is $\prec$-comparable to some point of $S$ must belong to $S$. 
\end{definition}

\begin{definition}
Let $X$ be a CAT(0) cube complex. A sequence of combinatorial rays $(r_n)$ satisfying $r_n(0)=r_m(0)$ for every $n,m \geq 0$ \emph{converges} to a combinatorial ray $r$ if, for every ball $B$ centered at $r_0(0)$, the sequence $(B \cap r_n)$ is eventually constant to $B \cap r$. A subset $\partial \subset \partial^c X$ is \emph{sequentially closed} if, for every sequence of combinatorial rays $(r_n)$ converging to some combinatorial ray $r$ and satisfying $r_n(+ \infty) \in \partial$ for every $n \geq 0$, $r(+ \infty) \in \partial$ holds.
\end{definition}

\begin{lemma}\label{lem:sequentiallyclosed}
Let $X$ be a CAT(0) cube complex and $Y \subset X$ a combinatorially convex subcomplex. If $Y$ is contracting then $\partial^c Y$ is a full and sequentially closed subset of $\partial^c X$.
\end{lemma}

\begin{proof}
The fact that $\partial^cY$ is full in $\partial^c X$ was noticed in \cite[Remark 4.15]{article3}. Let $(r_n)$ be a sequence of combinatorial rays such that: 
\begin{itemize}
	\item there exists some $x_0 \in X$ such that $r_n(0)=x_0$ for every $n \geq 0$; 
	\item $r_n(+ \infty) \in \partial^c Y$ for every $n \geq 0$; 
	\item and $(r_n)$ converges to some other combinatorial ray $r$. 
\end{itemize}
We want to prove that $r(+ \infty) \in \partial^c Y$. According to \cite[Lemma 4.5]{article3}, it is equivalent to show that $\mathcal{H}(r) \underset{a}{\subset} \mathcal{H}(Y)$. For convenience, set $D=d(r(0),Y)$. Suppose that there exists a finite subcollection $\mathcal{H} \subset \mathcal{H}(r) \backslash \mathcal{H}(Y)$ such that there exists some $k$ greater than $\max(D,\dim(X))$ so that $\# \mathcal{H} \geq \mathrm{Ram}(k)$; if such a $\mathcal{H}$ does not exist, then $| \mathcal{H}(r) \backslash \mathcal{H}(Y)| \leq \mathrm{Ram}( \max(D,\dim(X)))$ and there is nothing to prove. Notice that $\mathcal{H}$ contains a subcollection $\mathcal{H}_0$ with at least $k$ pairwise disjoint hyperplanes. Because there exist at most $D$ hyperplanes separating $r(0)$ from $Y$, $\mathcal{H}_0$ contains a subcollection $\mathcal{H}_1$ such that $\# \mathcal{H}_1 \geq \# \mathcal{H}_0 - D$ and such that no hyperplane of $\mathcal{H}_1$ separates $r(0)$ from $Y$. A fortiori, the hyperplanes of $\mathcal{H}_1$ separate some subray of $r$ from $Y$. 

\medskip \noindent
Now, choose some $n \geq 0$ sufficiently large so that the hyperplanes of $\mathcal{H}_1$ separate $Y$ and some subray $\rho_n \subset r_n$. Because $r_n(+ \infty) \in \partial^c Y$, we know that $\mathcal{H}(r) \underset{a}{\subset} \mathcal{H}(Y)$. As a consequence, we can choose some vertex $z \in r_n$ sufficiently far away from $r_n(0)$ so that there exists a collection $\mathcal{V}$ of at least $B+1$ hyperplanes intersecting both $\rho_n$ and $Y$, where $B$ is the constant given by Point $(ii)$ in Proposition \ref{prop:contracting} applied to $Y$. Since the hyperplanes of $\mathcal{H}_1$ separate $\rho_n$ and $Y$, and that the hyperplanes of $\mathcal{V}$ intersect both $\rho_n$ and $Y$, we deduce that any hyperplane of $\mathcal{H}_1$ is transverse to any hyperplane of $\mathcal{V}$. Moreover, $\mathcal{H}_1$ and $\mathcal{V}$ do not contain any facing triple, so $(\mathcal{H}_1, \mathcal{V})$ define a join of hyperplanes satisfying $\mathcal{H}_1 \cap \mathcal{H}(Y) = \emptyset$, $\mathcal{V} \subset \mathcal{H}(Y)$ and $\# \mathcal{V} \geq B+1$. From the definition of the constant $B$, it follows that $\# \mathcal{H}_1 \leq B$. Therefore,
$$k = \# \mathcal{H}_0 \leq \# \mathcal{H}_1+D \leq B+D,$$
hence $\# \mathcal{H} \leq \mathrm{Ram}(B+D)$. Consequently, $\mathcal{H}(r) \backslash \mathcal{H}(Y)$ is finite, which concludes the proof.
\end{proof}

\noindent
Now we are ready to turn to right-angled Artin groups. First of all, we recall some classical facts on their cubical geometry. So let $\Gamma$ be a simplicial graph. The Cayley graph $X(\Gamma)$ of the right-angled Artin group $A(\Gamma)$, constructed from its canonical generating set, is naturally a CAT(0) cube complex. (More precisely, the Cayley graph is a median graph, and the cube complex $X(\Gamma)$ obtained from it by \emph{filling in the cubes}, i.e., adding an $n$-cube along every induced subgraph isomorphic to the one-skeleton of an $n$-cube, turns out to be a CAT(0) cube complex.) For every vertex $u \in V(\Gamma)$, we denote by $J_u$ the hyperplane dual to the edge joining $1$ and $u$; every hyperplane of $X(\Gamma)$ is a translate of some $J_v$. It is worth noticing that, for every vertices $u,v \in V(\Gamma)$, the hyperplanes $J_u$ and $J_v$ are transverse if and only if $u$ and $v$ are adjacent vertices of $\Gamma$. Moreover, the carrier $N(J_u)$ of the hyperplane $J_u$ coincides with the subgraph generated by $\langle \mathrm{link}(u) \rangle \sqcup u \langle \mathrm{link}(u) \rangle$, where $\mathrm{link}(u)$ denotes the collection of the vertices of $\Gamma$ adjacent to $u$. As a consequence, the stabiliser of the hyperplane $J_u$ is the subgroup $\langle \mathrm{link}(u) \rangle$. 

\medskip \noindent
A key point in the proof of Theorem \ref{thm:MorseinRAAG} will be to understand the structure of the combinatorial boundary of $X(\Gamma)$. This is the purpose of our next statement. 

\begin{prop}\label{prop:boundaryRAAG}
Let $\Gamma$ be a connected simplicial graph not reduced to a single vertex. There exists a unique $\prec$-component of $\partial^c X(\Gamma)$ which is not reduced to a single point. Moreover, its sequential closure is the whole boundary $\partial^c X(\Gamma)$. 
\end{prop}

\noindent
Before proving this proposition, we will need several preliminary lemmas. 

\begin{lemma}\label{lem:fellowtravelhyp}
Let $X$ be a complete locally finite CAT(0) cube complex and $r \in \partial^cX$ a $\prec$-minimal combinatorial ray. Either there exists a hyperplane $J$ such that $r(+ \infty) \in \partial^c N(J) \subset \partial^c X$, or $\mathcal{H}(r)$ contains an infinite collection of pairwise strongly separated hyperplanes. In the latter case, $r(+ \infty)$ is an isolated point of $\partial^c X$.
\end{lemma}

\begin{proof}
According to \cite[Lemme 4.8]{article3}, there exists an infinite collection $\{ V_1, V_2, \ldots \} \subset \mathcal{H}(r)$ of pairwise disjoint hyperplanes. For convenience, suppose that $V_j$ separates $V_i$ and $V_k$ for every $1 \leq i<j<k$. 

\medskip \noindent
First, suppose that, for every $i \geq 1$, there exists some $j \geq i$ such that $V_i$ and $V_j$ are strongly separated. Notice that, for every $j_1 >j_2>j_3 \geq 1$, if $V_{j_1}$ and $V_{j_2}$ are strongly separated, as well as $V_{j_2}$ and $V_{j_3}$, then $V_{j_1}$ and $V_{j_3}$ are necessarily strongly separated. Consequently, $\{ V_1, V_2, \ldots \}$ (and a fortiori $\mathcal{H}(r)$) must contain an infinite subcollection of pairwise strongly separated hyperplanes. Up to taking a subcollection of $\{V_1, V_2, \ldots \}$, let us suppose that $V_i$ and $V_j$ are strongly separated for every $1 \leq i < j$. We want to prove that $r(+ \infty)$ is an isolated point of $\partial^cX$. 

\medskip \noindent
Let $\rho$ be a combinatorial ray. Up to taking a ray equivalent to $\rho$, we may suppose without loss of generality that $\rho(0)=r(0)$. If there exists some $i \geq 1$ such that $J_i \notin \mathcal{H}(\rho)$ then $V_i, V_{i+1}, \ldots \in \mathcal{H}(r) \backslash \mathcal{H}(\rho)$, and because no hyperplane intersects both $V_{i}$ and $V_{i+1}$, $\mathcal{H}(\rho) \cap \mathcal{H}(r)$ must be included into the set of the hyperplanes separating $r(0)$ from the edge $r \cap N(V_{i+1})$, so that it has to be finite. Thus, neither $r \prec \rho$ nor $\rho \prec r$ holds. From now on, up to extracting a subcollection of $\{ V_1,V_2, \ldots\}$, suppose that $V_1,V_2, \ldots \in \mathcal{H}(\rho)$. Let $J \in \mathcal{H}(r)$ be a hyperplane such that the edge $N(J) \cap r$ is between $V_j$ and $V_{j+1}$ for some $j \geq 2$. Because no hyperplane intersects both $V_{j-1}$ and $V_j$, nor both $V_{j+1}$ and $V_{j+2}$, we deduce that $J$ separates $V_{j-1}$ and $V_{j+2}$. On the other hand, we know that $V_{j-1}$ and $V_{j+2}$ intersect $\rho$, hence $J \in \mathcal{H}(\rho)$. Thus, we have proved that $r \prec \rho$. By symmetry, the same argument shows that $\rho \prec r$, hence $r \sim \rho$. As a consequence, we deduce that $r(+ \infty)$ is an isolated point of $\partial^c X$.

\medskip \noindent
Next, suppose that there exists some $i \geq 1$ such that $V_i$ and $V_j$ are not strongly separated for every $j \geq i$. Up to taking a subcollection of $\{ V_1, V_2, \ldots \}$, we may suppose without loss of generality that $i=1$. So we know that, for every $i \geq 1$, there exists a hyperplane $H_i$ intersecting both $V_1$ and $V_i$. Consequently, $(N(V_1), N(r), N(V_i),N(H_i))$ is a cycle of four convex subcomplexes. Let $D_i \hookrightarrow X$ be the flat rectangle given by Proposition \ref{prop:cycle}; for convenience, we identify $D_i$ with its image in $X$. Write $\partial D_i = u_i \cup \rho_i \cup v_i \cup h_i$ where $u_i \subset N(V_1)$, $\rho_i \subset N(r)$, $v_i \subset N(V_i)$ and $h_i \subset N(H_i)$ are combinatorial geodesics. Because $X$ is locally finite, up to taking a subsequence we may suppose without loss of generality that $(D_i)$ converges to a subcomplex $D_{\infty}$ in the sense that, for every ball $B$ centered at $r(0)$, the sequence $(B \cap D_i)$ is eventually constant to $B \cap D_{\infty}$. Noticing that each $D_i$ is a flat rectangle and that $\rho_i \underset{i \to + \infty}{\longrightarrow} + \infty$, we deduce that $D_{\infty}$ is isometric to either $[0,+ \infty) \times [0,+ \infty)$ or $[0,+ \infty) \times [0,L]$ for some $L \geq 1$ (depending on whether $(\mathrm{length}(u_i))$ is bounded or not). If $J$ denotes the hyperplane dual to the edge $\{ 0 \} \times [0,1]$ of $D_{\infty}$ and $\rho$ the combinatorial ray $[0,+ \infty) \times \{ 0 \} \subset D_{\infty}$ (which is also the limit of $(\rho_i)$), then $\rho(+ \infty) \in \partial^c N(J)$ since $\rho \subset N(J)$ by construction. On the other hand, we know that $\rho_i \subset N(r)$ for every $i \geq 1$, so $\rho \subset N(r)$. Because the hyperplanes of the subcomplex $N(r)$ are precisely the hyperplanes intersecting $r$, it follows that $\rho \prec r$. Finally, since $r$ is $\prec$-minimal by assumption, necessarily
$$r(+ \infty)= \rho(+ \infty) \in \partial^c N(J),$$
which concludes the proof.
\end{proof}

\begin{lemma}\label{lem:sequencehyp}
Let $\Gamma$ be a connected simplicial graph which is not reduced to a single vertex and $H,H'$ two hyperplanes of $X(\Gamma)$. There exist a sequence of hyperplanes
$$H_0=H, \ H_1, \ldots, H_{n-1}, \ H_n=J'$$
of $X(\Gamma)$ such that, for every $0 \leq i \leq n-1$, there exists two adjacent vertices $u,v \in V(\Gamma)$ and some $g \in A(\Gamma)$ such that $H_i = gJ_u$ and $H_{i+1}= gJ_v)$.
\end{lemma}

\begin{proof}
Up to translating by an element of $A(\Gamma)$, we suppose without loss of generality that $H=J_u$ and $H'=gJ_v$ for some $u,v \in V(\Gamma)$ and $g \in A(\Gamma)$. We argue by induction on the length of $g$. If $|g|=0$ then $H'=J_v$. Let
$$z_0 = u, \ z_1, \ldots, z_{r-1}, \ z_r=v$$
be a path in $\Gamma$ from $u$ to $v$. Then the sequence of hyperplanes
$$J_{z_0}=H, \ J_{z_1}, \ldots, \ J_{z_{r-1}}, \ J_{z_r}=H'$$
allows us to conclude. Next, suppose that $|g| \geq 1$. Write $g$ as a reduced word $hk$ where $h \in A(\Gamma)$ and $k \in \langle w \rangle \backslash \{ 1 \}$ for some $w \in V(\Gamma)$. Fix a vertex $x \in V(\Gamma)$ adjacent to $w$ (such a vertex exists since $\Gamma$ is a connected graph which is not reduced to a single vertex). Let
$$z_0=x , \ z_1, \ldots, z_{r-1}, \ z_r=v$$
be a path in $\Gamma$ from $x$ to $v$. Then
$$gJ_{z_0}=gJ_x=hJ_x, \ gJ_{z_1}, \ldots, \ gJ_{z_{r-1}}, \ gJ_{z_r}=gJ_v$$
defines a suitable sequence of hyperplanes from $hJ_x$ to $gJ_v$. Noticing that $|h|<|g|$, we deduce from our induction hypothesis that there exists a suitable sequence of hyperplanes from $J_u=H$ to $hJ_x$. By concatenating our two sequence of hyperplanes, we get a suitable sequence of hyperplanes from $H$ to $gJ_v=H'$, which concludes the proof.
\end{proof}

\begin{lemma}\label{lem:diamhyp}
Let $\Gamma$ be a simplicial graph and $u \in V(\Gamma)$ a vertex which is not isolated. Then $1 \leq \mathrm{diam}_{\prec} \partial^c N(J_u) \leq 4$.
\end{lemma}

\begin{proof}
We know that $N(J_u) \subset \langle \mathrm{star}(u) \rangle = \langle u \rangle \times \langle \mathrm{link}(u) \rangle$. Because $u$ is not an isolated vertex of $\Gamma$, $\mathrm{link}(u)$ is non-empty, so that $N(J_u)$ is included into the convex subcomplex $\langle \mathrm{star}(u) \rangle$ which decomposes as a Cartesian product of two unbounded subcomplexes, hence
$$\mathrm{diam}_{\prec} \partial^cN(J_u) \leq \mathrm{diam}_{\prec} \partial^c \langle \mathrm{star}(u) \rangle =4.$$
Moreover, since $\langle \mathrm{star}(u) \rangle$ contains a combinatorial copy of $\mathbb{R}^2$, it is clear that the $ \partial^cN(J_u)$ contains at least two two points, hence $\mathrm{diam}_{\prec} \partial^cN(J_u) \geq 1$.
\end{proof}

\begin{lemma}\label{lem:diamtransversehyp}
Let $\Gamma$ be a simplicial graph and $u,v \in V(\Gamma)$ two adjacent vertices. Then $\mathrm{diam}_{\prec} \left( \partial^c N(J_u) \cup \partial^c N(J_v) \right) \leq 10$. 
\end{lemma}

\begin{proof}
Let $r_u$ (resp. $r_v$) denote the combinatorial ray starting from $1$ and labelled by $u \cdot u \cdot \cdots$ (resp. labelled by $v \cdot v \cdots$). Because $u \in \mathrm{link}(v)$, we know that $\xi_u := r_u(+ \infty)$ belongs to $\partial^c N(J_v)$; similarly, $\xi_v:= r_v(+ \infty) \in \partial^c N(J_u)$. Now, let $\rho$ denote the combinatorial ray starting from $1$ and labelled by $u \cdot v \cdot u \cdot v \cdots$. The situation is the following: $\langle u,v \rangle$ defines a convex subcomplex isomorphic to $\mathbb{R}^2$, and $r_u$ corresponds to the horizontal ray $[0,+ \infty) \times \{0 \}$, $r_v$ to the vertical ray $\{0 \} \times [0,+ \infty)$ and $\rho$ to the ``diagonal'' ray starting from the origin included into the upper-right quadrant. In particular, $r_u \prec \rho$ and $r_v \prec \rho$. For convenience, set $\xi := \rho(+ \infty)$. Thanks to Lemma \ref{lem:diamhyp}, we deduce that
$$\mathrm{diam}_{\prec} \left( \partial^c N(J_u) \cup \partial^c N(J_v) \right) \leq \mathrm{diam}_{\prec} \partial^c N(J_u) + d_{\prec}(\xi_v, \xi_u) + \mathrm{diam}_{\prec} \partial^c N(J_v) \leq 10,$$
since $d_{\prec}(\xi_v, \xi_u) \leq d_{\prec}(\xi_v, \xi) + d_{\prec}(\xi,\xi_u) = 2$. This concludes the proof.
\end{proof}

\begin{proof}[Proof of Proposition \ref{prop:boundaryRAAG}.]
Let $r_1',r_2'$ be two combinatorial rays such that $r_1'(+ \infty)$ and $r_2'(+ \infty)$ are not isolated points of $\partial^cX(\Gamma)$. We want to prove that $r_1'(+ \infty)$ and $r_2'(+ \infty)$ belong to the same $\prec$-component of $\partial^c X(\Gamma)$. 

\medskip \noindent
First, as a consequence of Lemma \ref{lem:increasingsequence}, there exist two $\prec$-minimal combinatorial rays $r_1,r_2$ such that $r_1 \prec r_1'$ and $r_2 \prec r_2'$. So it is sufficient to prove that $r_1(+ \infty)$ and $r_2(+ \infty)$ belong to the same $\prec$-component of $\partial^c X(\Gamma)$. We deduce from Lemma \ref{lem:fellowtravelhyp} that there exist two hyperplanes $H_1,H_2$ of $X(\Gamma)$ such that $r_1(+ \infty) \in \partial^c N(H_1)$ and $r_2(+ \infty) \in \partial^c N(H_2)$. Let
$$J_1=H_1, \ J_2, \ldots, J_{n-1}, \ J_n=H_2$$
be the sequence of hyperplanes provided by Lemma \ref{lem:sequencehyp}. We deduce from Lemma \ref{lem:diamtransversehyp} that
$$d_{\prec}(r_1(+ \infty),r_2(+ \infty)) \leq \sum\limits_{k=1}^{n-1} \mathrm{diam} \left( \partial^c N(J_k) \cup \partial^c N(J_{k+1}) \right) \leq 10(n-1) < + \infty.$$
A fortiori, $r_1(+ \infty)$ and $r_2(+ \infty)$ belong to the same $\prec$-component of $\partial^c X(\Gamma)$. 

\medskip \noindent
Thus, we have prove that $\partial^c X(\Gamma)$ contains at most one $\prec$-component which is not reduced to a single point. On the other hand, we assumed that $\Gamma$ is not reduced to a single vertex, so $X(\Gamma)$ contains a combinatorial copy of $\mathbb{R}^2$, which implies that $\partial^c X(\Gamma)$ contains at least one $\prec$-component which is not reduced to a single point. Consequently, we have proved the first assertion of our proposition. Let us denote by $\partial$ the unique connected component of $\partial^c X(\Gamma)$.

\medskip \noindent
Let $r$ be a combinatorial ray such that $r(0)=1$ and such that $r(+ \infty)$ is an isolated point of $\partial^c X(\Gamma)$, and let
$$w = \ell_1 \cdot \ell_2 \cdot \ell_3 \cdots$$
denote the infinite reduced word labelling $r$ (where $\ell_1, \ell_2 \in V(\Gamma) \cup V(\Gamma)^{-1}$). Fix some $n \geq 1$. Say that $\ell_n \in \langle u \rangle$ for some $u \in V(\Gamma)$ and let $v \in V(\Gamma)$ be a vertex adjacent to $u$ (such a vertex exists since $\Gamma$ is a connected graph which we supposed not reduced to a single vertex). Set
$$w_n^{\pm} = \ell_1 \cdots \ell_{n-1} \cdot \ell_n \cdot v^{\pm 1} \cdot v^{\pm 1} \cdot v^{\pm 1} \cdots,$$
and $w_n=w_n^+$ if $w_n^+$ is a reduced word and $w_n=w_n^-$ otherwise. Notice that at least one of $w_n^+$ and $w_n^-$ must be reduced, so that $w_n$ has to be reduced. In particular, if we denote by $r_n$ the path in $X(\Gamma)$ starting from $1$ and labelled by $w_n$, then $r_n$ is a combinatorial ray. Moreover, $r_n$ eventually lies in $\ell_1 \cdots \ell_n \cdot N(J_u)$, so that $r_n(+ \infty) \in \ell_1 \cdots \ell_n \cdot \partial^c N(J_u)$. As a consequence of Lemma \ref{lem:diamhyp}, $\partial^c N(J_u)$ is not reduced to a point and its $\prec$-diameter is finite, so that $\partial^c N(J_u)$, and a fortiori $\ell_1 \cdots \ell_n \cdot \partial^c N(J_u)$, cannot contain isolated points of $\partial^c X(\Gamma)$. We conclude that $r_n(+ \infty) \in \partial$. 

\medskip \noindent
By construction, our sequence $(r_n)$ is eventually constant to $r$ on each ball, so that $(r_n)$ converges to $r$. Since we know that $r_n(+ \infty) \in \partial$ for every $n \geq 1$, we deduce that $r(+ \infty)$ belongs to the sequential closure of $\partial$, which concludes the proof of our proposition. 
\end{proof}

\noindent
We are finally ready to prove Theorem \ref{thm:MorseinRAAG}. 

\begin{proof}[Proof of Theorem \ref{thm:MorseinRAAG}.]
Let $\Gamma$ be a connected simplicial graph which is not reduced to a single vertex, and let $H$ be a Morse subgroup of $A(\Gamma)$. According to Corollary \ref{cor:MorseCriterion}, there exists a contracting convex subcomplex $Y \subset X(\Gamma)$ on which $H$ acts cocompactly. Therefore, it follows from \cite[Remark 4.15]{article3} and Lemma \ref{lem:sequentiallyclosed} that $\partial^c Y$ is a full and sequentially closed subset of $\partial^c X(\Gamma)$. We deduce from Proposition \ref{prop:boundaryRAAG} that, if $\partial^c Y$ contains an isolated point of $\partial^c X(\Gamma)$, then $\partial^cY= \partial^c X(\Gamma)$, so that $X(\Gamma)$ is a neighborhood of $Y$ according to Lemma \ref{lem:wholeboundary}. It follows that $H$ acts cocompactly on $X(\Gamma)$, so that $H$ must be a finite-index subgroup of $A(\Gamma)$. 

\medskip \noindent
From now on, suppose that $\partial^c Y$ contains only isolated points of $\partial^c X(\Gamma)$. As a consequence, the endpoints at infinity of an axis of any non-trivial element of $H$ must be isolated in $\partial^c X(\Gamma)$, since they necessarily belong to $\partial^c Y$, so we deduce from Theorem \ref{thm:contractingboundary} that any non-trivial isometry of $H$ is contracting. Now, we want to prove that $H$ is free. Let $J$ be a hyperplane. As a consequence of Lemma \ref{lem:diamhyp}, $\partial^c N(J)$ does not contain any isolated point of $\partial^cX(\Gamma)$, so that $\partial^c N(J) \cap \partial^cY = \emptyset$. Since a locally finite CAT(0) cube complex of infinite diameter must contain a combinatorial ray, we deduce that the intersection $N(J) \cap Y$ is necessarily finite. Therefore, because the hyperplanes of $Y$ are precisely the intersections of the hyperplanes of $X(\Gamma)$ with $Y$, it follows that the hyperplanes of $Y$ are finite. In fact, since $H$ acts cocompactly on $Y$, we know that the hyperplanes of $Y$ are uniformly finite, so that $Y$ must be quasi-isometric to a tree according to \cite[Proposition 3.8]{coningoff}. A fortiori, $H$ must be quasi-isometric to a tree, which implies that $H$ is virtually free (see for instance \cite[Th\'eor\`eme 7.19]{GhysdelaHarpe}), and in fact free since $H$ is also torsion-free (see \cite{StallingsTorsionFree}).
\end{proof}

\addcontentsline{toc}{section}{References}

\bibliographystyle{alpha}
{\footnotesize\bibliography{HCCC}}

\end{document}